\documentclass[reqno,11pt]{amsart}
\usepackage{xcolor}
\usepackage[margin=1.1in]{geometry}
\usepackage{amsthm,amsmath,amssymb,dsfont}
\usepackage{mathrsfs,amsfonts,functan,extarrows,mathtools}

\usepackage{hyperref}
\hypersetup{hidelinks}

 \allowdisplaybreaks
\usepackage{cite}
\usepackage{indentfirst, latexsym, amssymb, enumerate,amsmath,graphicx}
\usepackage{float}
\usepackage{relsize}
\usepackage{marginnote}
\usepackage{stmaryrd}
\usepackage{esint}
\usepackage{graphicx}
\usepackage{bm}
\usepackage{caption}
\usepackage{subfigure}
\usepackage{color}
\usepackage{graphicx}
\usepackage{subfigure}
\usepackage{float}
\usepackage{paralist}
\usepackage{indentfirst}

\allowdisplaybreaks[4]
\theoremstyle{plain}
\newtheorem{thm}{Theorem}[section]

\newtheorem{lem}[thm]{Lemma}
\newtheorem{prop}[thm]{Proposition}

\theoremstyle{definition}

\theoremstyle{remark}
\newtheorem{rem}{Remark}[section]

\numberwithin{equation}{section}

\usepackage{appendix}
\usepackage{xcolor}

\newcommand{\eps}{{\varepsilon}}

\begin{document}

\title[Compressible Navier-Stokes-Vlasov-Fokker-Planck system]
{Global well-posedness and inviscid limit of the compressible Navier-Stokes-Vlasov-Fokker-Planck system with density-dependent friction force}

\author[F. Li]{Fucai Li}
\address{School of Mathematics, Nanjing University, Nanjing 210093, China}
\email{fli@nju.edu.cn}

\author[J. Ni]{Jinkai Ni}
\address{School  of Mathematics, Nanjing University, Nanjing 210093, China}
\email{jni@smail.nju.edu.cn}

\author[D. Wang]{Dehua Wang} 
\address{Department of Mathematics, University of Pittsburgh, Pittsburgh, PA 15260, USA}
\email{dhwang@pitt.edu} 

\begin{abstract}

This paper investigates the global dynamics of a three-dimensional fluid-particle interaction system that couples the compressible barotropic Navier-Stokes equations with the Vlasov-Fokker-Planck equation through a density-dependent friction force. The study establishes the global well-posedness, uniform-in-viscosity estimates, the global inviscid limit, and optimal large-time decay rates for classical solutions near equilibrium. First, for initial perturbations in $H^3$ sufficiently close to equilibrium, regularity estimates that are uniform in the viscosity coefficient are derived, and the existence of global classical solutions to the Cauchy problem is obtained.  These uniform bounds enable us to rigorously justify the global-in-time inviscid limit as viscosity vanishes, with an explicit convergence rate proportional to the viscosity coefficient. This behavior differs significantly from that of the pure compressible Navier-Stokes system in the absence of particle interactions, emphasizing the stabilizing influence of kinetic coupling. Consequently, we establish for the first time the global existence of classical solutions to the compressible Euler-Vlasov-Fokker-Planck system. Moreover, under an additional mild assumption on the initial data, optimal time decay rates for both the solution and its spatial derivatives are obtained. Notably, the dissipative and microscopic components decay at a rate half an order faster than the macroscopic solution itself, indicating a novel relaxation mechanism induced by fluid-particle interactions. The analysis introduces new energy and dissipation structures for the coupled system, overcoming substantial difficulties arising from fluid-particle interactions.

\end{abstract}



\date{\today}
\maketitle

\tableofcontents
%-----------------section one-------------------------------------------------------------
%\renewcommand{\theequation}{\thesection.\arabic{equation}}
%\setcounter{equation}{0}
\setcounter{equation}{0}
 \indent % \allowdisplaybreaks

\section{Introduction}
%\subsection{The models}
%Fluid-particle models, in which dispersed particles suspended in fluids, are widely applied across various fields such as combustion theory \cite{Wfa-1958,Wfa-1985}, chemical engineering \cite{CP-siam-1983}, medical biosprays \cite{BBJM-esaim-2005}, sedimentation phenomena \cite{BWC-zamm-2000}, and diesel engine operations \cite{RM-1952,RM-1952-a}. In this paper, we consider the following fluid-particle model, which consists of the compressible barotropic Navier-Stokes  equations coupled with the Vlasov-Fokker-Planck   equation featuring a density-dependent friction force (cf. \cite{LMW-SIAM-2017,LNW-2025-preprint}):
Fluid-particle flows, in which dispersed particles are suspended in a fluid, have wide applications across various fields, including combustion theory \cite{Wfa-1958,Wfa-1985}, chemical engineering \cite{CP-siam-1983}, medical biosprays \cite{BBJM-esaim-2005}, sedimentation phenomena \cite{BWC-zamm-2000}, and diesel engine operations \cite{RM-1952,RM-1952-a}. In this paper, we consider the following fluid-particle system, consisting of the compressible barotropic Navier-Stokes  equations coupled with the Vlasov-Fokker-Planck   equation featuring a density-dependent friction force (cf. \cite{LMW-SIAM-2017,LNW-2025-preprint}):
\begin{equation}\label{A1}
\left\{\begin{aligned}
&\partial_{t} \rho+\mathrm{div}(\rho u)=0,  \\
&\partial_{t}(\rho u)+\mathrm{div}(\rho u\otimes u)+\nabla P-\mu\Delta u =-\rho\int_{\mathbb{R}^{3}}(u-v)F {\rm d}v,\\
&\partial_{t}F+v\cdot\nabla F+\rho\mathrm{div}_{v}[(u-v)F-\nabla_{v}F]=0,
 \end{aligned}
 \right.
\end{equation}
with the initial data
\begin{equation}\label{A-1}
(\rho,u,F)|_{t=0}=(\rho_0(x),u_0(x),F_0(x)), \quad x\in\mathbb{R}^3. 
\end{equation}
Here,   $\rho = \rho(t, x) \geq 0$ and $u = u(t, x) \in \mathbb{R}^3$, defined for $(t, x) \in \mathbb{R}^+ \times \mathbb{R}^3$, represent the mass density and velocity of the fluid, respectively. The function $F = F(t, x, v) \geq 0$, defined for $(t, x, v) \in \mathbb{R}^+ \times \mathbb{R}^3 \times \mathbb{R}^3$, denotes the particle density distribution function. In \eqref{A1}, $P = P(\rho)$ is the pressure function  depending only on $\rho$, with $P'(\cdot) > 0$ (As for a prototype, we recall that $P(\rho) = c_0 \rho^\gamma$ with $\gamma > 1$ and $c_0 > 0$ for the case of an isentropic gas). Furthermore, the parameter $\mu > 0$ represents the viscosity coefficient of the fluid.

The equations \eqref{A1}$_2$ and \eqref{A1}$_3$ interact with each other through the friction force $\rho(u - v)$. As noted by Mellet and Vasseur \cite{MV-MMMAS-2007}, this friction force generally depends on the fluid density, which reflects its influence on the particle motion. This formulation has greater physical significance compared to the density-independent friction force $u - v$, which is widely used in different fluid-particle models \cite{CG-CPDE-2006,CKL-JHDE-2013,DL-KRM-2013,EHM-N-2021,Hd-pmp-2022,
HM-MAMS-2024,HMM-arma-2020,LNW-2025 SAM,LLY22,LWW-ARMA-2022,
 LS-JDE-2021,LSZ-2025-KRM,MV-MMMAS-2007,MV-CMP-2008,Mu-Wang-20,
 SY-2020-JDE,SWYZ-2023-JDE,Ww-CMS-2024}. 
 However, due to the presence of the fluid density $\rho$ in \eqref{A1}, certain trilinear terms arise in fluid-particle models. Consequently, obtaining global solutions for such models becomes more challenging, and only a few rigorous mathematical results are available; for example, see \cite{Wang-Yu2015,CJ-SIAM-2021,BD-JHDE-2006,Choi-Kwon-15,LMW-SIAM-2017}. Specifically, Li, Mu, and Wang \cite{LMW-SIAM-2017} established the global well-posedness of classical solutions within the $H^4(\mathbb{R}^3)$ framework under the condition that the initial data is a small perturbation of the global Maxwellian. Moreover, they demonstrated that the solutions decay at a rate of $(1+t)^{- {1}/{2}}$ in the $L^\infty$ norm over $\mathbb{R}^3$. However,  the decay behavior in the natural $L^2$-norm remains unaddressed in their paper. Recently, Li, Ni and  Wu \cite{LNW-2025-preprint}  improved these results by reducing the required regularity of the initial data from $H^4(\mathbb{R}^3)$ to $H^2(\mathbb{R}^3)$ and deriving optimal time-decay rates for strong solutions and their gradients in both the $L^2$-norm and the $L^p$-norm for $2 \leq p \leq 6$. It is noted that the \emph{fixed} viscosity coefficient $\mu>0$ plays a crucial role in the arguments of  \cite{LMW-SIAM-2017, LNW-2025-preprint}. 

When the viscosity term $\mu \Delta u$ is omitted, the system \eqref{A1} reduces to the following compressible Euler-Vlasov-Fokker-Planck (Euler-VFP) system (cf.\cite{CGL-JCP-2008}):  
\begin{equation}\label{A2}
\left\{
\begin{aligned}
&\partial_{t} \rho + \mathrm{div}(\rho u) = 0, \\
&\partial_{t}(\rho u) + \mathrm{div}(\rho u \otimes u) + \nabla P = -\rho \int_{\mathbb{R}^{3}}(u - v)F \, {\rm d}v, \\
&\partial_{t}F + v \cdot \nabla F + \rho\,\mathrm{div}_{v}[(u - v)F - \nabla_{v}F] = 0.
\end{aligned}
\right.
\end{equation}
Formally, the inviscid fluid-particle system \eqref{A2} can be obtained by taking the limit $\mu \to 0$ in the compressible Navier-Stokes-Vlasov-Fokker-Planck (NS-VFP) system \eqref{A1}–\eqref{A-1}. However, 
to our best knowledge, until now, there are  no mathematical results on global well-posedness for this system \eqref{A2}. One of our objectives is to address this issue and establish the existence of global solutions to it.
%system \eqref{A2}.

More precisely, in this paper, we study the \emph{vanishing viscosity limit} to the NS-VFP system \eqref{A1}, and establish the global well-posedness of the Euler-VFP system \eqref{A2}. Furthermore, we derive the sharp convergence rate of $\mathcal{O}(\mu)$ in this limiting process, which significantly improves upon the previously known rate of $\mathcal{O}(\mu^{\frac{1}{2}})$ obtained in \cite{LWW-ARMA-2022}. Under the additional assumption that the initial data lie in $L^q$ with $q \in [1, \frac{6}{5})$, we obtain the optimal time-decay rates for classical solutions and their spatial gradients in the $L^2$-norm, as well as in the $L^p$-norm for $2 \leq p \leq 6$. Finally, we demonstrate the exponential decay of $(\rho, u, F)$ for the Euler-VFP system \eqref{A2} in $\mathbb{T}^3$.

\subsection{Brief reviews on previous results of fluid-particle systems}
For our Euler-VFP system \eqref{A2}, when the diffusion term $\rho\Delta_v F$ is omitted, it reduces to the following inviscid compressible system:  
\begin{equation}\label{A3}  
    \left\{  
    \begin{aligned}  
    &\partial_{t} \rho + \mathrm{div}(\rho u) = 0, \\  
    &\partial_{t}(\rho u) + \mathrm{div}(\rho u\otimes u) + \nabla P = -\rho\int_{\mathbb{R}^{3}}(u-v)F \, {\rm d}v, \\  
    &\partial_{t}F + v\cdot\nabla F + \rho\mathrm{div}_{v}[(u-v)F] = 0.  
    \end{aligned}  
    \right.  
\end{equation}  
%Baranger and   Desvillettes \cite{BD-JHDE-2006} first established the local existence of classical solutions to \eqref{A3}, assuming that the initial data are sufficiently smooth and their supports satisfy appropriate conditions. 
The local existence of classical solutions to \eqref{A3} was established by Baranger and   Desvillettes \cite{BD-JHDE-2006} when the initial data are sufficiently smooth and their supports satisfy appropriate conditions.
For the incompressible version corresponding to the system \eqref{A3}:  
\begin{equation} \label{A4}  
    \left\{  
      \begin{aligned}  
        &\partial_{t} \rho + \mathrm{div}(\rho u) = 0,\quad{\rm div}u=0, \\  
        &\partial_{t}(\rho u) + \mathrm{div}(\rho u\otimes u) + \nabla \pi = -\rho\int_{\mathbb{R}^{3}}(u-v)F \, {\rm d}v, \\  
        &\partial_{t}F + v\cdot\nabla F + \rho\mathrm{div}_{v}[(u-v)F] = 0,  
      \end{aligned}  
    \right.  
\end{equation}  
%Wang and Yu \cite{Wang-Yu2015} obtained the global existence of weak solutions to \eqref{A4} with large initial data in a bounded domain under the reflection boundary condition. Choi and Kwon \cite{Choi-Kwon-15} established the existence of strong solutions on $[0,T]$ to \eqref{A4} in both $\mathbb{T}^3$ and $\mathbb{R}^3$ for any $T>0$, provided that the initial data are sufficiently small and regular.
the global weak solutions to \eqref{A4} were obtained by Wang and Yu \cite{Wang-Yu2015} for large initial data in a bounded domain with the reflection boundary condition, 
 and the existence of strong solutions on $[0,T]$ for any fixed $T>0$ to \eqref{A4} was proved in \cite{Choi-Kwon-15} by Choi and Kwon in both $\mathbb{T}^3$ and $\mathbb{R}^3$ provided that the initial data are sufficiently small and regular.

%  In this work, we extend these findings to our Euler-VFP system \eqref{A2}.

For the compressible NS-VFP system \eqref{A1}, when the density-independent friction force $u - v$ is considered, the system reduces to the following form:
\begin{equation}\label{A5}
\left\{\begin{aligned}
&\partial_{t} \rho+\mathrm{div}(\rho u)=0,  \\
&\partial_{t}(\rho u)+\mathrm{div}(\rho u\otimes u)+\nabla P-\mu\Delta u =- \int_{\mathbb{R}^{3}}(u-v)F {\rm d}v,\\
&\partial_{t}F+v\cdot\nabla F+ \mathrm{div}_{v}[(u-v)F-\nabla_{v}F]=0.
 \end{aligned}
 \right.
\end{equation}
%Mellet and Vasseur \cite{MV-MMMAS-2007} established the global existence of weak solutions to \eqref{A5} in a bounded domain $\Omega\subset\mathbb{R}^3$. Furthermore, they investigated the asymptotic behavior of these solutions in \cite{MV-CMP-2008}. Chae, Kang, and Lee \cite{CKL-JHDE-2013} proved the global existence and exponential decay of classical solutions to \eqref{A5} in $\mathbb{T}^3\times\mathbb{R}^3$. Very recently, Wang \cite{Ww-CMS-2024} established the global existence of strong solutions in $H^2$ and derived the optimal decay rates for all spatial derivatives of the solutions to \eqref{A5}, including an additional term $\nabla{\rm div}u$ in the momentum equation \eqref{A5}$_2$.
%
The existence and asymptotic behavior of weak solutions to \eqref{A5} were proved by Mellet and Vasseur \cite{MV-MMMAS-2007,MV-CMP-2008} in a three-dimensional bounded domain.
%, and their asymptotic behavior was studied in \cite{MV-CMP-2008}.  
The global existence and exponential decay of classical solutions were investigated by Chae, Kang, and Lee \cite{CKL-JHDE-2013} in $\mathbb{T}^3\times\mathbb{R}^3$. In a recent work by Wang \cite{Ww-CMS-2024} the global existence of strong solutions was established in $H^2$ with the optimal decay rates for solutions' spatial derivatives, including an added term $\nabla{\rm div}u$ in the momentum equation \eqref{A5}$_2$.
 For the compressible Euler-VFP equations with friction force $u - v$, i.e., by neglecting the viscosity term $\mu \Delta u$ in \eqref{A5}$_2$, 
\begin{equation}\label{A6}  
\left\{  
\begin{aligned}  
&\partial_{t} \rho + \mathrm{div}(\rho u) = 0, \\  
&\partial_{t}(\rho u) + \mathrm{div}(\rho u \otimes u) + \nabla P = - \int_{\mathbb{R}^{3}} (u - v) F \, {\rm d}v, \\  
&\partial_{t}F + v \cdot \nabla F + \mathrm{div}_{v}[(u - v)F - \nabla_{v}F] = 0.  
\end{aligned}  
\right.  
\end{equation}  
Carrillo and Goudon \cite{CG-CPDE-2006} performed asymptotic analysis and hydrodynamic expansions to  \eqref{A6} in both the flowing and bubbling regimes. Duan and Liu \cite{DL-KRM-2013} established the global well-posedness of classical solutions to   \eqref{A6} under small initial perturbations around the equilibrium state 
$(\rho,u,F)=(1, 0, M)$. Recently, Li, Wang and Wang \cite{LWW-ARMA-2022} studied the global asymptotic stability of the systems   \eqref{A5} and  \eqref{A6} in $\mathbb{R}^3$.  For additional results on compressible fluid-particle models, interested readers can refer to \cite{LLY22,Mu-Wang-20,CJ-SIAM-2021,LS-JDE-2021,LNW-24-1} and the references cited therein.

For the incompressible NS-VFP equations with a density-independent friction force of the form $\rho( u - v) $,
it takes the  following form:  
\begin{equation} \label{A7}  
\left\{  
\begin{aligned}  
&\partial_{t} \rho + \mathrm{div}(\rho u) = 0, \quad \mathrm{div}\, u = 0, \\  
&\partial_{t}(\rho u) + \mathrm{div}(\rho u \otimes u) + \nabla \pi-\mu\Delta u = - \int_{\mathbb{R}^{3}} (u - v) F \, \mathrm{d}v, \\  
&\partial_{t} F + v \cdot \nabla F + \mathrm{div}_{v}[(u - v) F] = \Delta_v F.  
\end{aligned}  
\right.  
\end{equation}  
This system is known as the incompressible inhomogeneous NS-VFP system. However, only a few mathematical results are available for this model. Su and Yao \cite{SY-2020-JDE} studied the hydrodynamic limit of system \eqref{A7} in a bounded domain $\Omega \subset \mathbb{R}^2$.
Jiang, Li, and Ni \cite{JLN-preprint-2024} proved the global well-posedness and large-time behavior of solutions to system \eqref{A7}.
When the diffusion term $\Delta_v F$ in \eqref{A7}$_3$ is omitted, the corresponding hydrodynamic limit was established in a bounded domain $\Omega \subset \mathbb{R}^3$ in \cite{SWYZ-2023-JDE}.  Li, Shou, and Zhang \cite{LSZ-2025-KRM} obtained exponential stability of solutions to the NS–Vlasov system near vacuum in $\mathbb{R}^3$. Very recently, %Collaborated with Shou 
in \cite{lnsw-2025} we joint with Shou obtained the global well-posedness, optimal decay rates,  and inviscid limit to the system \eqref{A7} 
with more general friction force $\rho(u-v)$.

In  the system \eqref{A7}, when the density $\rho$ is a positive constant,  it  is reduced to the so-called incompressible homogeneous NS-VFP system:  
\begin{equation} \label{A8}  
  \left\{  
  \begin{aligned}  
    &\partial_{t} u + u \cdot \nabla u + \nabla \pi - \mu \Delta u = - \int_{\mathbb{R}^{3}} (u - v) F \, \mathrm{d}v, \\  
    &\mathrm{div}\, u = 0, \\  
    &\partial_{t} F + v \cdot \nabla F + \mathrm{div}_{v}[(u - v) F] = \Delta_v F.  
  \end{aligned}  
  \right.  
\end{equation}  
%There is a considerable body of literature devoted to the above system and related models. 
%Hamdache \cite{Hk-JJIAM-1998} first established the global existence and long-time behavior of weak solutions for the Vlasov-Stokes system in a bounded domain. 
A substantial amount of literature exists on this system and its related models. Hamdache \cite{Hk-JJIAM-1998} first proved the global existence and long-term behavior of weak solutions to the Vlasov-Stokes system in a bounded domain. 
The global existence of weak solutions to the incompressible homogeneous Navier-Stokes-Vlasov (NSV) system  (i.e., ignoring the $\Delta_v F$ in \eqref{A8}) on a three-dimensional torus $\mathbb{T}^3$ was proved by Boudin, Desvillettes, Grandmont, Moussa \cite{BDGM-DIE-2009}. Yu \cite{Yc-JMPA-2013} investigated the global existence of weak solutions to the NSV system under reflective boundary conditions in a bounded domain $\Omega \subset \mathbb{R}^d$, where $d = 2, 3$, and further established uniqueness in $\mathbb{R}^2$. Subsequently, Boudin, Grandmont, and Moussa \cite{BGM-2017-JDE} proved the existence of global weak solutions for the NSV system in a time-dependent domain $\Omega_t\subset \mathbb{R}^3$. 
Recently, %Han-Kwan et al. 
various results were obtained in \cite{HMM-arma-2020, EHM-N-2021, Hd-pmp-2022} on the global existence and large-time behavior of weak solutions to the NSV system on a torus $\mathbb{T}^3$, the whole space $\mathbb{R}^3$, and a bounded domain $\Omega \subset \mathbb{R}^3$, respectively. For a regular equilibrium $(u,F)=(0,M)$, 
Goudon, He, Moussa, and Zhang in \cite{GHMZ-sjma-2010} established the global existence of classical solutions to the system \eqref{A8} in a torus $\mathbb{T}^3$. Furthermore, Carrillo, Duan, and Moussa \cite{CDM-krm-2011} studied the corresponding inviscid case of the system \eqref{A8} (i.e. ignoring the term $-\mu \Delta u$ in \eqref{A8}).
For more results on the  incompressible fluid-particle models,   interested readers can refer to \cite{Danchin-2024,LNW-2025 SAM,GJV-2004-2-IUMJ,GJV-IUMJ-2004,HM-MAMS-2024}.

\subsection{Main results}
We focus on the stability of global classical solutions to the system \eqref{A1} around the equilibrium  state:  
$$  
(\rho, u, F) = (\rho_{\infty}, u_{\infty}, M_{[\rho_{\infty},u_{\infty}]}),  
$$  
where $\rho_{\infty} > 0$ is a given constant, $u_{\infty} = (u_{\infty1}, u_{\infty2}, u_{\infty3})$  
a given constant vector, and  %$M_{[\rho_{\infty},u_{\infty}]}$ 
$$M_{[\rho_{\infty},u_{\infty}]} =\frac{\rho_\infty}{(2\pi)^{\frac{3}{2}}} e^{-\frac{|v-u_\infty|^2}{2}}$$
is the Maxwellian distribution.  %, i.e.,  $M_{[\rho_{\infty},u_{\infty}]} =\frac{\rho_\infty}{(2\pi)^{\frac{3}{2}}} e^{-\frac{|v-u_\infty|^2}{2}}$.

Without loss of generality, we can choose $\rho_{\infty} = 1$ and $u_{\infty} = (0,0,0)$ throughout this paper for notational  and presentation
simplicity. We denote $M=M(v)=M_{[1,0]}$. To achieve the objectives of our study, we consider the following standard transformations:  
$$
\rho ^\mu= 1 + \varrho^\mu, \quad F^\mu = M + \sqrt{M} f^\mu.
$$  
Accordingly, the Cauchy problem \eqref{A1}–\eqref{A-1} can be rewritten as
\begin{equation}\label{C1}
\left\{\begin{aligned}
&\partial_{t}\varrho^\mu+(\varrho^\mu+1){\rm div} u^\mu+\nabla\varrho^\mu\cdot u^\mu=0,\\  
&\partial_{t}u^\mu+u^\mu\cdot\nabla u^\mu+{\frac{P^{\prime}(1+\varrho^\mu)}{1+\varrho^\mu}}\nabla\varrho^\mu-\frac{\mu\Delta u^\mu}{1+\varrho^\mu}={b^\mu-u^\mu-a^\mu u^\mu},\\
& \partial_{t}f^\mu+v\cdot\nabla_x f^\mu+u^\mu\cdot\nabla_{v}f^\mu-{\frac{1}{2}}u^\mu\cdot vf^\mu-u^\mu\cdot v\sqrt{M}\\
&\quad= \mathcal{L}f^\mu+\varrho^\mu\Big(\mathcal{L}f^\mu-u^\mu\cdot\nabla_v f^\mu+\frac{1}{2}u^\mu\cdot vf^\mu+u^\mu\cdot v\sqrt{M}\Big),\\ 
&(\varrho^\mu,u^\mu,f^\mu)|_{t=0}=(\varrho_0^\mu(x),u_0^\mu(x),f_0^\mu(x)),\\
 \end{aligned}
 \right.
\end{equation}
with $\varrho^{\mu}_{0}(x) = \rho^{\mu}_{0}(x) - 1$ and $f^{\mu}_0 = \frac{F_0^\mu(x,v) - M}{\sqrt{M}}$. 

For clarity, we have introduced the superscript $\mu$ in the unknowns to explicitly indicate their dependence on the parameter $\mu$.
Here, $\mathcal{L}$ in \eqref{C1} denotes the linearized Fokker-Planck operator, which is defined as  
\begin{align*}  
\mathcal{L}f^\mu = \frac{1}{\sqrt{M}} \mathrm{div}_{v} \left[ M \nabla_{v} \left( \frac{f^\mu}{\sqrt{M}} \right) \right].  
\end{align*}  
Moreover, $a^\mu= a^{f^\mu}$ and $b^\mu = b^{f^\mu}$ represent the corresponding moments of $f^\mu$, defined by  
\begin{align*}  
a^\mu=a^{f^{\mu}}(t,x) = \int_{\mathbb{R}^{3}} \sqrt{M} f^\mu(t,x,v)\, {\rm d}v, \quad 
b^\mu=b^{f^\mu}(t,x) = \int_{\mathbb{R}^{3}} v \sqrt{M} f^\mu(t,x,v)\, {\rm d}v.  
\end{align*}  
We define the dissipation functional as follows:  
\begin{equation}\label{D0}  
\begin{aligned}  
\mathcal{D}(\varrho^\mu ,u^\mu , f^\mu )(t) := &\, \|b^\mu - u^\mu \|_{H^3}^2 + \|\nabla (\varrho^\mu,a^\mu ,b^\mu )\|_{H^2}^2 \\  
& + \sum_{|\alpha|\leq 3} \|\{\mathbf{I}-\mathbf{P}\}\partial^\alpha f^\mu \|_{\nu}^2  
+ \sum_{\substack{1 \leq |\beta| \leq 3\\ |\alpha|+|\beta| \leq 3}} \|\partial_\beta^\alpha\{\mathbf{I}-\mathbf{P}\} f^\mu \|_\nu^2,  
\end{aligned}  
\end{equation}  
where $\{\mathbf{I}-\mathbf{P}\}f^\mu$ denotes the microscopic component of $f^\mu$, as introduced in Section 2.3.

Similarly, for $\mu=0$, by applying the transformations  
$$  
\rho = 1 + \varrho, \quad F = M + \sqrt{M} f,  
$$  
the Cauchy problem for \eqref{A2} can be reformulated as   
\begin{equation}\label{C2}
\left\{\begin{aligned}
&\partial_{t}\varrho +(\varrho +1){\rm div} u +\nabla\varrho \cdot u =0,\\  
&\partial_{t}u +u\cdot\nabla u +{\frac{P^{\prime}(1+\varrho )}{1+\varrho }}\nabla\varrho ={b -u -a  u },\\
& \partial_{t}f +v\cdot\nabla_x f +u \cdot\nabla_{v}f -{\frac{1}{2}}u \cdot vf -u \cdot v\sqrt{M}\\
&\quad= \mathcal{L}f +\varrho \Big(\mathcal{L}f -u\cdot\nabla_v f+\frac{1}{2}u \cdot vf +u \cdot v\sqrt{M}\Big),\\ 
&(\varrho ,u,f )|_{t=0}=(\varrho_0 (x),u_0 (x),f_0 (x)), 
 \end{aligned}
 \right.
\end{equation}
where $a=a^f$ and $b=b^f$.

Our first result addresses the global existence of classical solutions to the compressible NS-VFP system \eqref{C1}. Specifically, we derive regularity estimates that are uniform in both $\mu$ and time.

\begin{thm}\label{T1.1} 
Let $0<\mu<1$.
Assume that $(\varrho_0^\mu,u_0^\mu,f_0^\mu)$
satisfies  $(\varrho_0^\mu,u_0^\mu )\in H^3$, $f_0^\mu\in H^3_{x,v}$ and $F_0^\mu=M+\sqrt{M}f_0^\mu\geq 0$.
There exists a constant $\varepsilon_0>0$ independent of $\mu$ such that if
\begin{align}\label{a-1}
\mathcal{E}_0^\mu:=\|(\varrho^{\mu}_0,u^{\mu}_0 )\|_{H^3}^2+\|f_0^\mu\|_{H_{x,v}^3}^2\leq \varepsilon_0,
\end{align}
then the compressible  NS-VFP system  \eqref{C1} 
admits a unique global classical solution $(\varrho^{\mu},u^{\mu}, \linebreak f^{\mu})$, which satisfies $\rho^{\mu}=1+\varrho^\mu>0$, $F^\mu=M+\sqrt{M}f^\mu\geq 0$, and
\begin{align}\label{a-2}
&\sup_{t\geq 0}\big(\|(\varrho^{\mu},u^{\mu}  )(t)\|_{H^{3}}^2+\|f^\mu(t)\|_{H_{x,v}^3}^2\big)\nonumber\\
&\quad +\int_{0}^{t} \mathcal{D}(\varrho^\mu,u^\mu, f^\mu)(\tau){\rm d}\tau 
 +\mu\int_0^{t}\|\nabla u (\tau)\|_{H^3}^2{\rm d}\tau \leq C_0\mathcal{E}_0^\mu,\quad t\in \mathbb{R}_{+},
\end{align}
where $C_0 > 0$ is a constant that is independent of  $\mu$ and time $t$.
\end{thm}

Based on the uniform estimates derived in Theorem \ref{T1.1}, we investigate the vanishing viscosity limit as $\mu \to 0$ and prove the global existence of classical solutions to the compressible Euler-VFP system \eqref{C2}.

\begin{thm}\label{T1.2} Assume that $(\varrho_0, u_0, f_0)$ satisfies $(\varrho_0, u_0) \in H^3$ and $f_0 \in H_{x,v}^3$, with $F_0 = M + \sqrt{M}f_0 \geq 0$. Let $\{(\varrho_0^\mu, u_0^\mu, f_0^\mu)\}_{0 < \mu < 1}$ be a sequence such that $(\varrho^\mu_0, u^\mu_0) \to (\varrho_0, u_0)$ in $H^3$ and $f_0^\mu \to f_0$ in $H^3_{x,v}$. There exists a constant $\varepsilon_1 > 0$, independent of $\mu$, such that if  
\begin{align}\label{b-1}
\tilde{\mathcal{E}}_0 := \|(\varrho_0, u_0)\|_{H^3}^2 + \|f_0\|_{H_{x,v}^3}^2 \leq \varepsilon_1,
\end{align}  
then, for the global classical solution $(\varrho^\mu, u^\mu, f^\mu)$ of the compressible NS-VFP system \eqref{C1} subject to the initial data $(\varrho_0^\mu, u_0^\mu, f_0^\mu)$ obtained in Theorem \ref{T1.1},  there exists a limit $(\varrho, u, f)$ as $\mu \to 0$, such that, up to a subsequence, the following convergence holds:  
\begin{equation}\label{b2}
\left\{
\begin{aligned}
(\varrho^\mu, u^\mu) &\rightharpoonup (\varrho, u) \quad \text{weakly}^* \text{ in } L^{\infty}(\mathbb{R}_{+}; H^3(\mathbb{R}^3)), \\
f^\mu &\rightharpoonup f \quad \quad\,\,\,\,\text{weakly}^* \text{ in } L^{\infty}(\mathbb{R}_{+}; H_{x,v}^3(\mathbb{R}^3 \times \mathbb{R}^3)), \\
(\varrho^\mu, u^\mu) &\to (\varrho, u) \quad \text{strongly in } C_{\rm loc}(\mathbb{R}_{+}; H^2_{\rm loc}(\mathbb{R}^3)).
\end{aligned}
\right.
\end{equation}  
The limit $(\varrho, u, f)$ is the unique global classical solution to the compressible Euler-VFP system \eqref{C2} associated with the initial data $(\varrho_0, u_0, f_0)$. Moreover, $(\varrho, u, f)$ satisfies $\rho = 1 + \varrho > 0$, $F = M + \sqrt{M}f \geq 0$, and  
\begin{align}\label{b3}
&\sup_{t \geq 0} \big( \|(\varrho, u)(t)\|_{H^{3}}^2 + \|f(t)\|_{H_{x,v}^3}^2 \big) + \int_{0}^{t} \mathcal{D}(\varrho, u,   f)(\tau)\, {\rm d}\tau \leq C_1 \tilde{\mathcal{E}}_0, \quad \forall\, t \in \mathbb{R}_{+},
\end{align}  
where the $C_1 > 0$ is a constant independent of $\mu$ and time $t$.  
\end{thm}

To establish global convergence over time as $\mu \rightarrow 0$, we derive an error estimate that quantifies the difference between the global solutions of the systems \eqref{C1} and \eqref{C2}.

\begin{thm}\label{T1.3}
Let $(\varrho^\mu, u^\mu, f^\mu)$ and $(\varrho, u,  f)$ be the global classical solutions to the compressible NS-VFP system \eqref{C1} and the compressible Euler-VFP system \eqref{C2}, as obtained in Theorem \ref{T1.1} and Theorem \ref{T1.2}, respectively, associated with the initial data $(\varrho_0^\mu, u_0^\mu,   f_0^\mu)$ and $(\varrho_0, u_0,  f_0)$. 
Furthermore, assume that 
\begin{align}\label{c1}
\|(\varrho_0^\mu-\varrho_0,u_0^\mu-u_0 )\|_{H^1}+\|f_0^\mu-f_0\|_{H_{x,v}^1}\leq \mu.
\end{align}
Then, it follows that  
\begin{align}\label{c2}
&\sup_{t\geq 0}\big(\|(\varrho^\mu-\varrho,u^\mu-u)(t)\|_{H^1}^2+\|(f^\mu-f)(t)\|_{H_{x,v}^1}^2\big)+\int_0^t   \|\big((b^\mu-u^\mu)-(b-u)\big)(\tau)\big\|_{H^1}^2   {\rm d}\tau\nonumber\\
& \quad+\int_0^t\Big( \|\nabla(\varrho^\mu-\varrho)(\tau)\|_{L^2}^2+ \sum_{|\alpha|\leq 1}\|\{\mathbf{I}-\mathbf{P}\}\partial^\alpha (f^\mu-f)(\tau)\|_{\nu}^2+\|\nabla (f^\mu-f)(\tau)\|_{L_{x,v}^2}^2\Big){\rm d}\tau\nonumber\\
&\quad\quad \leq  C_2\mu^2 ,\quad t\in \mathbb{R}_{+},
\end{align}
where the   constant $C_2>0$ is a constant independent of $\mu$ and  time $t$. 
\end{thm}

\begin{rem}
For $\mu\geq 1$, the well-posedness  of classical solutions stated in Theorem \ref{T1.1} remain valid (cf. \cite{LMW-SIAM-2017,LNW-2025-preprint}). In this paper, we focus on the case $0<\mu<1$, aiming to derive uniform estimates with respect to $\mu$ within this range so as to study the inviscid limit as $\mu\rightarrow 0$.
\end{rem}

The next result concerns the time decay of the classical solutions to the Euler-VFP system \eqref{C2}.
\begin{thm}\label{T1.4}
 Assume that
\eqref{b-1} holds and $\|(\varrho_0,u_0)\|_{L^q}+\|f_0\|_{\mathcal{Z}_q}$ is bounded for $1\leq q<\frac{6}{5}$, then
the solution $(\varrho,u ,f)$ to the compressible Euler-VFP system \eqref{C2} satisfies
\begin{align} 
 \|\nabla^k(\varrho, u )(t)\|_{L^2 } +\|\nabla^kf(t)\|_{{L_v^2(L^{2})}}
&\leq C_3(1+t)^{-\frac{3}{2} (\frac{1}{q} - \frac{1}{2} )-\frac{k}{2}},\quad k=0,1,\label{d1-1}\\
\|\nabla^k(\varrho, u)(t)\|_{L^2 } +\|\nabla^kf(t)\|_{{L_v^2(L^{2})}}
&\leq C_3(1+t)^{-\frac{3}{2} (\frac{1}{q} - \frac{1}{2} )-\frac{1}{2}},\,\,\,\,\,\, k=2,3.\label{d1}
\end{align}
and
\begin{align}
\| (\varrho, u )(t)\|_{L^p}+\|  f(t)\|_{L_v^2(L^p)} &\,\leq C_3(1+t)^{-\frac{3}{2}\big(\frac{1}{q}-\frac{1}{p}\big)},\,\,\,\quad\quad 2\leq p\leq 6,\label{d2-1}\\  
\| (\varrho, u)(t)\|_{L^p}+\|  f(t)\|_{L_v^2(L^p)} &\,\leq C_3(1+t)^{-\frac{3}{2}\big(\frac{1}{q}-\frac{1}{6}\big)},\,\,\,  \quad\,\,\,\,\,\, 6\leq p\leq \infty. \label{d2}
\end{align}
{Furthermore, we have
\begin{align}\label{LNW-new}
\|(b-u)(t)\|_{L^2}+\|\{\mathbf{I}-\mathbf{P}\}f(t)\|_{L_{v}^2(L^2)}\leq  C_3(1+t)^{-\frac{3}{2} (\frac{1}{q} - \frac{1}{2} )-\frac{1}{2}},
\end{align}}
where  the  constant $C_3>0$ is a constant independent of    time $t$.
\end{thm}

\begin{rem}
We point out that Theorem \ref{T1.4} still holds for the NS-VFP system \eqref{C1} for fixed  $\mu>0$.
Here we choose to present the case ``$\mu=0$" to  emphasize the uniform of $\mu$ in the arguments. 
 
\end{rem}

\begin{rem}
  In \cite{DL-KRM-2013} on the compressible Euler-VFP system \eqref{A6}, the initial data $(\varrho_0, u_0, f_0)$ is required to be sufficiently small in the $L^1$-norm. In this paper, we only need a weaker requirement that $(\varrho_0, u_0, f_0)$ is bounded in the $L^q$-norm for $q \in [1, \frac{6}{5})$. 
  %Thus, our assumption on the $L^q$-norm constitutes a strictly weaker condition than the $L^1$-norm requirement imposed in \cite{DL-KRM-2013}.
\end{rem}

\begin{rem}
The classical $L^q-L^2$ type optimal decay rates, for $q\in [1,\frac{6}{5})$, are consistent with the significant results established %by Duan et al. 
in \cite{DLUY-2007-JDE} for the barotropic Navier-Stokes equations.
\end{rem}

\begin{rem}
The microscopic part $\{\mathbf{I}-\mathbf{P}\}f$ and the relative velocity $b-u$ decay in time at an additional $\frac{1}{2}$-order faster than $(\varrho,u,f)$. This newly observed phenomenon in this paper reveals the relaxation effects induced by the linear Fokker-Planck operator $\mathcal{L}$ and the friction term for fluid-particle interactions.
\end{rem}

\begin{rem}
In the proof of Theorem \ref{T1.4}, the nonlinear term $\varrho\mathcal{L}f$ in \eqref{C2}$_3$ introduces significant challenges in estimating the time decay rates of $(\varrho,u,f)$. For the friction force $u - v$ considered in \cite{DL-KRM-2013,Ww-CMS-2024}, the nonlinear terms in the perturbation of the particle equation $f$ do not involve the term $v^2f$. %, and their structure is relatively simpler due to the presence of the friction force $u - v$. 
As a result, the a priori estimates can be closed without requiring energy estimates on the mixed space-velocity derivatives of the solutions; see \cite{DL-KRM-2013,Ww-CMS-2024} for further details. In contrast, our density-dependent friction force $\varrho(u - v)$ leads to the appearance of the nonlinear term $\varrho\mathcal{L}f$, which necessitates incorporating energy estimates involving mixed space-velocity derivatives of the solutions in order to handle the nonlinear term $R_5(t)$ (see estimates \eqref{G6.9} and \eqref{G6.18}).
\end{rem}

\begin{rem}
  As mentioned earlier, the system \eqref{C2} was first proposed by Carrillo, Goudon and Lafitte \cite{CGL-JCP-2008}. However, to our best knowledge, there are no mathematical results on it to date. Our results (Theorems \ref{T1.2} and \ref{T1.4}) provide the first global well-posedness and time-decay estimates for classical solutions to this system. 
\end{rem}

\subsection{Difficulties and strategies in our proofs}

To begin with, we employed the refined energy method to establish uniform-in-$\mu$ a priori estimates for $(\varrho^\mu,u^\mu,f^\mu)$ in the compressible NS-VFP system \eqref{C1}, under the assumption that the initial data possess $H^3$ regularity. While the zero-order estimates for $(\varrho^\mu, u^\mu, f^\mu)$ are relatively straightforward, the conventional energy method fails to provide effective higher-order estimates for these variables. The main difficulty arises when dealing with the integral  
\begin{align*}  
\int_{\mathbb{R}^3}\partial^\alpha (\varrho^\mu \nabla\cdot u^\mu)\partial^\alpha\varrho^\mu\,{\rm d}x  
\end{align*}  
during the derivation of high-order energy estimates, because it cannot be controlled directly via standard techniques.  

Inspired by \cite{DL-KRM-2013}, we exploit a symmetric structure involving the term $\frac{P^{\prime}(1+\varrho^\mu)}{(1+\varrho^\mu)^2}\partial^\alpha\varrho^\mu$. By combining this term with the velocity equation, we transform the problematic term into the following integral:  
\begin{align*}  
\int_{\mathbb{R}^3}\nabla\left(\frac{P^\prime(1+\varrho^\mu)}{1+\varrho^\mu}  \right)\cdot\partial^\alpha u^\mu\,\partial^\alpha\varrho^\mu\,{\rm d}x.  
\end{align*}  
As a result, applying Lemmas \ref{L2.1}–\ref{L2.2}, we are able to close the energy estimate uniformly for all $\mu > 0$ (see also the estimates of $\varrho$ in Lemma \ref{L3.2}). Furthermore, we derive the uniform dissipation of $u^\mu$ by analyzing the dissipation of both $b^\mu - u^\mu$ and $b^\mu$, specifically through the inequality  
\begin{align*}  
\|\nabla u^\mu\|_{H^2} \lesssim \|b^\mu - u^\mu\|_{H^3} + \|\nabla b^\mu\|_{H^2},  
\end{align*}  
where $b^\mu - u^\mu$ serves as a relaxation term and $b^\mu$ behaves similarly to a solution of an elliptic equation (cf. \cite{DL-KRM-2013,LNW-2025-preprint}).  
Then, the uniformity allows us to establish the global well-posedness of classical solutions to system \eqref{C1} and analyze the vanishing viscosity limit as $\mu \rightarrow 0$. This result significantly improves upon our previous findings in \cite{LNW-2025-preprint}, where the $H^2$-regularity was required and $\mu > 0$ was assumed to be a fixed constant. Indeed, in the current work, the $H^3$-regularity is necessary, as demonstrated by the estimate for $\partial_t\varrho^\mu$ in \eqref{NJK-aaaa}.

To rigorously demonstrate the global-in-time convergence of the limiting process, it is essential to establish a uniform convergence rate by analyzing the difference between the global classical solutions of the systems \eqref{C1} and \eqref{C2}. To this end, we construct the following energy functional (see \eqref{G4.4}):  
\begin{align*}  
\widetilde{\mathcal{X}}^\mu(t) :=&\sup_{\tau\in[0,t]}\big(\|(\widetilde{\varrho}^\mu,\widetilde{u}^\mu )(\tau)\|_{H^1}^2 +\|\widetilde{f}^\mu(\tau)\|_{H_{x,v}^1}^2\big)+\int_{0}^t  \|\nabla(\widetilde{\varrho}^\mu,\widetilde{a}^\mu,\widetilde{b}^\mu )(\tau)\|_{L^2}^2 {\rm d}\tau \nonumber\\  
& +\int_0^t \|(\widetilde{b}^\mu-\widetilde{u}^\mu)(\tau)\|_{H^1}^2   {\rm d}\tau+\int_0^t\sum_{|\alpha|\leq 1}\|\{\mathbf{I}-\mathbf{P}\}\partial^\alpha\widetilde{f}^\mu(\tau)\|_{\nu}^2{\rm d}\tau\nonumber\\  
&+\int_0^t \|\nabla_v \{\mathbf{I}-\mathbf{P}\}\widetilde{f}^\mu(\tau)\|_{\nu}^2{\rm d}\tau.  
\end{align*}  
Since the VFP equations in the systems \eqref{C1} and \eqref{C2} involves the macro-micro decomposition of particles, the estimates of $\widetilde{a}^\mu$ and $\widetilde{b}^\mu$ become  significantly more difficult and challenging, which complicates the energy estimates for these variables (see Lemma \ref{L4.3} below). By carefully estimating $\widetilde{\varrho}^\mu$, $\widetilde{u}^\mu$, and $\widetilde{f}^\mu$ through using a refined energy method, we derive Lemmas \ref{L4.1}-\ref{L4.5}, which enable us to rigorously justify the global-in-time vanishing viscosity limit with a convergence rate of $\mu$, as shown by the following uniform estimate:  
\begin{align*}  
\widetilde{\mathcal{X}}^\mu(t) \lesssim \mu^2.  
\end{align*}  
Therefore, the proof of   Theorem \ref{T1.3} is rigorously established.

Next, we shall investigate the optimal time-decay rates of $(\varrho, u, f)$ for the Euler-VFP system \eqref{C2}. To achieve this, we first linearize system \eqref{C2}, and then apply Fourier analysis techniques to the linearized problem \eqref{G5.1}-\eqref{G5.2} in order to derive sharp pointwise estimates of $(\varrho, u, f)$. Based on these estimates and by employing low-frequency and high-frequency decomposition techniques to construct a general decay framework, we establish Theorem \ref{T5.1}. This theorem demonstrates that $(\varrho, u, f)$ exhibits heat-like diffusion behavior at low frequencies and damping effects at high frequencies. Moreover, we obtain $L^p$-$L^q$ estimates for the problem \eqref{G5.1}-\eqref{G5.2}, which enable us to determine the optimal time-decay rates under the assumption that the initial data belong to $L^q$ for $q \in [1, \frac{6}{5})$.

Before analyzing the nonlinear Euler-VFP system \eqref{C2}, we establish the following higher-order Lyapunov inequality (see Proposition \ref{P5.2}):  
\begin{align*}  
&\frac{\rm d}{{\rm d}t}\big( \|\nabla^k(\varrho,u )\|_{L^2}^2+\|\nabla^k f\|_{L_{x,v}^2}^2\big) + \lambda_{11}\big(  \|\nabla^k(\varrho,u)\|_{L^2}^2+\|\nabla^k f\|_{L_{x,v}^2}^2\big)  
\lesssim \|\nabla^k(\varrho^L,a^L,b^L)\|_{L^2}^2,  
\end{align*}  
for any $k=2,3$, based on Lemmas \ref{L5.3}–\ref{L5.5}, which involve careful computations. This inequality is a crucial step in deriving the optimal time-decay rates of $(\varrho,u,f)$ without requiring the smallness assumption on $\|(\varrho_0,u_0,f_0)\|_{\mathcal{Z}_q}$ for $q\in[1,\frac{6}{5})$. In contrast to \cite{LNW-2025-preprint}, where the strong assumption that $\|(\varrho_0,u_0,f_0)\|_{\mathcal{Z}_1}$ is bounded is imposed for the NS-VFP system \eqref{C1}, here we assume only the weaker condition that $\|(\varrho_0,u_0,f_0)\|_{\mathcal{Z}_q}$ is bounded for $q\in[1,\frac{6}{5})$.

Then, we address the nonlinear Euler-VFP system \eqref{C2}. According to Duhamel's principle, the solution $U(t)$ of the nonlinear problem can be expressed as  
\begin{align*}  
U(t) = \mathbb{A}(t)U_0 + \int_0^t \mathbb{A}(t - s)\big(S_\varrho(s), S_u(s), S_f(s)\big) \, \mathrm{d}s.  
\end{align*}  
Next, we estimate the nonlinear terms $S_\varrho(s), S_u(s)$, and $S_f(s)$ in the $\mathcal{Z}_2$-norm. To remove the smallness assumption, we first derive slower time decay rates for $(\varrho, u, f)$ (see \eqref{decay-a}). Subsequently, by employing an iterative approach, we obtain the optimal time decay rates of $(\varrho, u, f)$ and their spatial derivatives, which allow us to bound $\mathcal{E}_{1,\infty}$ (see \eqref{G6.14}) via the inequality  
\begin{align*}  
\mathcal{E}_{1,\infty}(t) \lesssim \|U_0\|_{\mathcal{H}^3 \cap \mathcal{Z}_q}^2.  
\end{align*}  
By applying interpolation inequalities, we further establish the optimal time decay rates of $(\varrho, u, f)$ in the $L^p$-norm for all $2 \leq p \leq 6$. As noted in \cite{LNW-2025-preprint}, since the term $\varrho\mathcal{L}f$ cannot be fully incorporated into the absolute energy functional $\mathcal{E}(t)$, the corresponding dissipation functional $\mathcal{D}(t)$ cannot be derived. Consequently, as shown in the estimates of $\|R_5(t)\|_{\mathcal{Z}_2}^2$ given by \eqref{G6.9} and \eqref{G6.18}, the optimal time decay rates for the second-order and third-order spatial derivatives of $(\varrho, u, f)$ cannot be obtained. 
 Therefore, to close the energy estimates, the energy estimate involving the mixed space-velocity derivatives of $f$, as stated in Lemma \ref{L3.5}, becomes necessary -- this is due to the fact that the nonlinear term $\varrho\mathcal{L}f$ contains terms such as $\varrho v^2f$, whose estimate, which prevents us from closing the bound on $\|R_5(t)\|_{\mathcal{Z}_2}^2$, must be used. 
 This approach differs from that in \cite{DL-KRM-2013,Ww-CMS-2024}, where consideration of the mixed derivatives of $f$ may not be required.

Furthermore, we derive improved decay estimates for the relative velocity $b - u$ and the microscopic part $\{\mathbf{I}-\mathbf{P}\}f$. To analyze $b - u$, we take advantage of the damping effects and rewrite the equation of $b - u$ according to \eqref{LNWG6.23}. Then, by semigroup theory and the decay estimates \eqref{d1-1}--\eqref{d1}, we successfully handle the nonlinear terms (see \eqref{LNWG6.24}). As for the microscopic part, we write the equation of $\{\mathbf{I}-\mathbf{P}\}f$ as
\begin{align}
\label{LNWGGG}
&\partial_t \{\mathbf{I}-\mathbf{P}\}f-(\varrho + 1)\mathcal{L}{\{\mathbf{I}-\mathbf{P}\}}f+v\cdot\nabla\{\mathbf{I}-\mathbf{P}\}f \nonumber\\
&\quad=-\partial_t\mathbf{P}f-v\cdot\nabla\mathbf{P}f-(\varrho + 1) u\cdot\nabla_v f+(\varrho + 1)(u - b)\cdot v\sqrt{M}+\frac{1}{2}(\varrho + 1)u\cdot v f.
\end{align}
Note that the Fokker-Planck operator $\mathcal{L}$ provides additional relaxation effects.
 Meanwhile, the terms on the right-hand side of \eqref{LNWGGG} can be controlled 
 by the improved decay of $b - u$, $\nabla (a,b)$, and $\partial_t (a,b)$ 
 (see \eqref{LNWG6.26}--\eqref{LNWG6.27}), which enables us to achieve a $\tfrac{1}{2}$-order enhancement in the decay rate of $\{\mathbf{I}-\mathbf{P}\}f$. 
% Therefore, the proof of Theorem \ref{T1.4} is completed.

Finally, by applying the conservation laws and the Poincar\'{e}'s inequality, we establish the exponential decay in time of $(\varrho, u, f)$ for the Euler-VFP system \eqref{C2} in $\mathbb{T}^3$.

\subsection{Outline of the paper}
The rest of this paper is organized as follows. 
In Section 2, we introduce notation, useful lemmas, the macro-micro decomposition framework, and properties of the Fokker-Planck operator $\mathcal{L}$, which are used throughout the paper. 
In Section 3, we apply the refined energy method to establish the existence of global classical solutions to the compressible NS-VFP system \eqref{C1} in $\mathbb{R}^3$ with a uniform viscosity coefficient $\mu > 0$, and prove Theorem \ref{T1.1}. 
In Section 4, we investigate the vanishing viscosity limit problem and thereby proving Theorem \ref{T1.2}, while deriving a convergence rate of $\mathcal{O}(\mu)$ for the solutions, as stated in Theorem \ref{T1.3}. 
In Section 5, we systematically derive time-decay estimates for $(\varrho, u, f)$ associated with the linearized compressible Euler-VFP system \eqref{C2}, and establish a higher-order Lyapunov inequality using low- and high-frequency decomposition techniques. 
In Section 6, we first obtain the slower time-decay rates of $(\varrho, u, f)$. Subsequently, utilizing an iterative approach, we derive the optimal time-decay rates under the additional assumption that the norm $\|(\varrho, u, f)\|_{\mathcal{Z}_q}$ is bounded for $q \in [1, \frac{6}{5})$. 
Finally, in Section 7, we extend our results to the periodic domain $\mathbb{T}^3$ and establish exponential time-decay rates of $(\varrho, u, f)$ for the Euler-VFP system \eqref{C2}.

\medskip

 \section{Preliminaries}
 
In this section, we introduce the necessary notation, useful lemmas, and the main properties of the Fokker-Planck operator $\mathcal{L}$, which will be frequently used in the subsequent analysis.

\subsection{Notation}
The letter $C$ denotes a generic positive constant that is independent of time $t$, and its value may vary from line to line. The notation $a \lesssim b$ means that there exists a constant $C > 0$ such that $a \leq C b$. The expression $A \sim B$ indicates the two-sided inequality $\frac{A}{C} \leq B \leq CA$ holds for some universal constant $C$. We define $\|(g,h)\|_{X} := \|g\|_{X} + \|h\|_{X}$, where $g(x)$ and $h(x)$ are functions belonging to a Banach space $X$.
We use $\langle\cdot,\cdot\rangle$ to denote the inner product on the Hilbert space $L^2_v = L^2(\mathbb{R}^3_v)$, i.e.,  
\begin{align*}
\langle g,h \rangle := \int_{\mathbb{R}^3} g(v) h(v) \, \mathrm{d}v, \quad g,h \in L_v^2,
\end{align*}  
with the corresponding norm $\|\cdot\|_{L^2}$.

We denote the weight function $\nu(v) := 1 + |v|^2$ and define the corresponding norm $|\cdot|_{\nu}$ by  
$$
|g|_{\nu}^2 := \int_{\mathbb{R}^3} \left(|\nabla_v g(v)|^2 + \nu(v)|g(v)|^2 \right) {\rm d}v, \quad g = g(v).
$$  
The norm $\|\cdot\|_{\nu}$ denotes the spatial integration of $|\cdot|_{\nu}$, defined as  
$$
\|g\|_{\nu}^2 :=  \int_{\mathbb{R}^3 \times \mathbb{R}^{3}} \left(|\nabla_{v} g(x,v)|^2 + \nu(v)|g(x,v)|^2 \right) {\rm d}x{\rm d}v.
$$  
For $q \geq 1$, the standard space-velocity mixed Lebesgue space $Z_q = L_v^2(L_x^q) = L^2(\mathbb{R}_v^3; L^q(\mathbb{R}^3_x))$ is defined by  
$$
\|g\|_{Z_q}^2 := \int_{\mathbb{R}^3}\Big(\int_{\mathbb{R}^3}|g(x,v)|^q\mathrm{d}x\Big)^{\frac{2}{q}}\mathrm{d}v.
$$  
Furthermore, we introduce the norms $\|\cdot\|_{\mathcal{Z}_{q}}$ and $\|\cdot\|_{\mathcal{H}^{m}}$ for the integer $m \geq 0$ and the real number $q \geq 1$ as follows:  
$$
\begin{aligned}
\|(f,\varrho, u)\|_{\mathcal{Z}_{q}} &= \|f\|_{Z_{q}}+\|\varrho\|_{L^{q}}+\|u\|_{L^{q}}, \\
\|(f,\varrho, u)\|_{\mathcal{H}^{m}} &= \|f\|_{L^{2}_{v}(H^m)}+\|\varrho\|_{H^{m}}+\|u\|_{H^{m}}.
\end{aligned}
$$

For an integrable function $g:\mathbb{R}^{3}\rightarrow\mathbb{R}$, its Fourier transform is defined as  
$$  
\widehat{g}(\xi)=\mathcal{F}g(\xi)=\int_{\mathbb{R}^{3}}e^{-ix\cdot \xi}g(x)\,{\rm d}x, \quad x\cdot \xi=\sum\limits_{j=1}^{3}x_{j}\xi_{j},  
$$  
where $\xi\in\mathbb{R}^{3}$. Here, $i=\sqrt{-1}\in\mathbb{C}$ denotes the imaginary unit. For two complex-valued functions $f$ and $g$, the expression $(g|h) := g \cdot \overline{h}$ represents the dot product of $g$ with the complex conjugate of $h$.

For any multi-indices $\alpha=(\alpha_{1},\alpha_{2},\alpha_{3})$ and $\beta=(\beta_{1},\beta_{2},\beta_{3})$, we denote  
$$
\partial_{x}^{\alpha}\partial_{v}^{\beta} = \partial_{x_{1}}^{\alpha_{1}}\partial_{x_{2}}^{\alpha_{2}}\partial_{x_{3}}^{\alpha_{3}}\partial_{v_{1}}^{\beta_{1}}\partial_{v_{2}}^{\beta_{2}}\partial_{v_{3}}^{\beta_{3}},
$$  
where $x = (x_1, x_2, x_3)$ denotes the spatial variable and $v = (v_1, v_2, v_3)$ denotes the microscopic velocity of the particles. The lengths of $\alpha$ and $\beta$ are defined as $|\alpha| = \alpha_1 + \alpha_2 + \alpha_3$ and $|\beta| = \beta_1 + \beta_2 + \beta_3$, respectively. We define  
$$
\|g\|_{H^s} = \sum_{|\alpha| \leq s} \|\partial^\alpha g\|_{L^2}, \quad \|g\|_{H_{x, v}^s} = \sum_{|\alpha| + |\beta| \leq s} \|\partial_{x}^{\alpha}\partial_{v}^{\beta} g\|_{L_{x,v}^2}.
$$  
The sysmbol $C_\alpha^\beta$ denotes the generalized binomial coefficient. For notational clarity, in this paper we use the symbols $\Delta$ and $\nabla$ to indicate $\Delta_x$ and $\nabla_x$, respectively.

Finally, we recall the decomposition of  low-frequency and high-frequency components. For a function $g(x) \in L^2(\mathbb{R}^3)$, we define its low- and high-frequency decompositions, denoted by $g^L$ and $g^H$, respectively, as  
\begin{align}\label{G2.1}
g^{L}(x) = \phi_0(D_x)g(x), \quad g^{H}(x) = \phi_1(D_x)g(x).
\end{align}  
Here, $D_x = \frac{1}{\sqrt{-1}} (\partial_{x_1}, \partial_{x_2}, \partial_{x_3})$, and $\phi_0(D_x)$ and $\phi_1(D_x)$ are pseudo-differential operators corresponding to the smooth cut-off functions $\phi_0(z)$ and $\phi_1(z)$ satisfying $0 \leq \phi_0(z) \leq 1$ and $\phi_1(z) = 1 - \phi_0(z), z\in \mathbb{R}^3$, where  
\begin{align}\label{G2.2}
\phi_0(z) =
\begin{cases}
1, & |z| \leq \frac{r_0}{2}, \\
0, & |z| > r_0,
\end{cases}
\end{align}  
for a fixed constant $r_0 > 0$. From \eqref{G2.1}--\eqref{G2.2}, it follows directly that  
\begin{align*}
g(x) = g^{L}(x) + g^{H}(x).
\end{align*}  
For any $g \in H^2(\mathbb{R}^3)$, applying the   Plancherel theorem yields the following estimates:  
\begin{align}  \label{G2.3}
\|g^H\|_{L^2} \leq \frac{C}{r_0} \|\nabla g\|_{L^2}, \quad \|g^H\|_{L^2} \leq \frac{C}{r_0^2} \|\nabla^2 g\|_{L^2}, \quad \|\nabla^2 g^L\|_{L^2} \leq C r_0 \|\nabla g^L\|_{L^2}.
\end{align}

\subsection{Useful lemmas}
Next, we present several useful lemmas that will be frequently employed throughout this paper.

\begin{lem}[{\!\!{\cite[Lemma 2.1]{CDM-krm-2011}}}]\label{L2.1}   
For any $g,h\in H^3(\mathbb{R}^3)$ and any multi-index $\alpha$  with $1\leq|\alpha|\leq3$, it holds
\begin{align*}
\|g\|_{L^{\infty}(\mathbb{R}^{3})} \lesssim \,& \|\nabla g\|_{L^{2}(\mathbb{R}^{3})}^{\frac{1}{2}}
\|\nabla^{2}g\|_{L^{2}(\mathbb{R}^{3})}^{\frac{1}{2}}, \\
\|gh\|_{H^{2}(\mathbb{R}^{3})}\lesssim\,&  \|g\|_{H^{2}(\mathbb{R}^{3})}\|\nabla h\|_{H^{2}(\mathbb{R}^{3})}, \\
\|\partial^{\alpha}(gh)\|_{L^{2}(\mathbb{R}^{3})}
 \lesssim\,& \|\nabla g\|_{H^{2}(\mathbb{R}^{3})}\|\nabla h\|_{H^{2}(\mathbb{R}^{3})}. 
\end{align*}
\end{lem} 

\begin{lem}[{\!\!{\cite[Lemmas 2.1--2.2]{Dk-MZ-1992}}}]\label{L2.2}
For any $g \in H^1(\mathbb{R}^3)$, one has
\begin{align*}
\|g\|_{L^6(\mathbb{R}^3)} \lesssim\,& \|\nabla g\|_{L^2(\mathbb{R}^3)}\lesssim\|g\|_{H^1(\mathbb{R}^3)},
\end{align*}
and
\begin{align*}
  \|g\|_{L^q(\mathbb{R}^3)} \lesssim\,& \|g\|_{H^1(\mathbb{R}^3)}, 
\end{align*}
for $2 \leq q \leq 6$.
\end{lem}

\begin{lem}[{\!\!\cite[Appendix]{commutator1},\cite{commutator2}}]\label{L2.3}
Let $g$ and $h$ be two Schwarz functions. For $k\geq 0$, one has
\begin{align*}
\|\nabla^{k}(gh) \|_{L^r} \lesssim &\, \|g\|_{L^{r_1} }\|\nabla^{k}h\|_{L^{r_2} }+C\|h\|_{L^{r_3} }\|\nabla^{k}g\|_{L^{r_4} },\\
\|\nabla^{k}(gh)-g\nabla^k h \|_{L^r} \lesssim  &\, \|\nabla g\|_{L^{r_1}}\|\nabla^{k-1}h\|_{L^{r_2}}+C\|h\|_{L^{r_3}}\|\nabla^{k}g\|_{L^{r_4}},    
\end{align*}
where $1<r,r_2,r_4<\infty$ and $r_i(1\leq i\leq 4)$ satisfy 
\begin{align*}
\frac{1}{r_1}+\frac{1}{r_2}=\frac{1}{r_3}+\frac{1}{r_4}=\frac{1}{r}.   
\end{align*}
\end{lem}

\begin{lem} [{\!\!{\cite[Lemma 3.2]{CDM-krm-2011}}}]\label{L2.4}   
Given any $0<\beta_1\ne 1$ and $\beta_2>1$, then
\begin{align*}
\int_0^t (1+t-s)^{-\beta_1}(1+s)^{-\beta_2} {\rm d}s \lesssim (1+t)^{-\min\{\beta_1,\beta_2\}}, 
\end{align*}
for all $t\geq 0$.
\end{lem} 

\begin{lem} [{\!\!\cite[Lemma 3.3]{CDM-krm-2011}}]\label{L2.5}   
Let $\gamma>1$ and $g_1,g_2\in C(\mathbb{R}_+,\mathbb{R}_+)$ with
$g_1(0)=0$. For $A\in \mathbb{R}_+$, define 
$$\mathcal{B}_{A}:=\{y\in C(\mathbb{R}_+,\mathbb{R}_+)|\, \,y\leq A+g_1(A)y+g_2(A)y^{\gamma},\,\, y(0)\leq A\}.$$
Then, there exists a constant $A_0\in (0,\,\min\{A_1,A_2\})$ such that for
any $0<A\leq A_0$, 
\begin{align*}
y\in \mathcal{B}_A \Rightarrow \sup_{t\geq 0}y(t)\leq 2A.
\end{align*}
\end{lem}

\subsection{Macro-micro decomposition and some properties of $\mathcal{L}$}
We present the macro-micro decomposition of our distribution function $f(t,x,v)$, which was originally proposed in the context of the Boltzmann equation with hard sphere potential \cite{GY-iumj-2004} and later adapted to the Fokker-Planck equation \cite{DFT-2010-CMP}. This decomposition expresses $f(t,x,v)$ as the sum of the fluid (macroscopic) component $\mathbf{P}f$ and the kinetic (microscopic) component   $\{\mathbf{I}-\mathbf{P}\}f$, such that  
\begin{align}\label{G2.4}  
f = \mathbf{P}f + \{\mathbf{I}-\mathbf{P}\}f.  
\end{align}
Here, $\mathbf{P}$ denotes the velocity orthogonal projection operator satisfying  
$$
\mathbf{P}: L^2 \rightarrow {\rm Span}\big\{\sqrt{M}, v_i\sqrt{M} \big\}, \quad 1 \leq i \leq 3.
$$  
Furthermore, $\mathbf{P}$ can be decomposed as  
$$
\mathbf{P} := \mathbf{P}_0 \oplus \mathbf{P}_1.
$$  
We define  
$$
a^f = \int_{\mathbb{R}^3} \sqrt{M} f \, {\rm d}v, \quad b^f = \int_{\mathbb{R}^3} v \sqrt{M} f \, {\rm d}v.
$$  
Accordingly, we obtain $\mathbf{P}f = \mathbf{P}_0 f + \mathbf{P}_1 f$, where  
$$
\mathbf{P}_0 f := a^f \sqrt{M}, \quad \mathbf{P}_1 f := b^f v \sqrt{M}.
$$  

It is well-known that  $\mathcal{L}$ is self-adjoint, and its kernel and range are given by  
$$
{\rm Ker}\,\mathcal{L} := {\rm Span}\big\{\sqrt{M}\big\}, \quad {\rm Range}\,\mathcal{L} := {\rm Span}\big\{\sqrt{M}\big\}^\perp.
$$  
Thus, $\mathcal{L}f$ can be expressed as  
$$
\mathcal{L}f = \mathcal{L}(\mathbf{I} - \mathbf{P})f + \mathcal{L}\mathbf{P}f = \mathcal{L}(\mathbf{I} - \mathbf{P})f - \mathbf{P}_1 f.
$$  
In addition, there exists a positive constant $\lambda_0 > 0$ such that (see \cite{LMW-SIAM-2017,DFT-2010-CMP,CDM-krm-2011}) 
\begin{align}
-\langle \mathcal{L}f, f\rangle
&\geq\lambda_{0}|\{\mathbf{I}-\mathbf{P}_0\}f|_{\nu}^{2},\label{G2.5-1}\\
-\langle \mathcal{L}\{\mathbf{I}-\mathbf{P}\}f,f\rangle
&\geq\lambda_{0}|\{\mathbf{I}-\mathbf{P}\}f|_{\nu}^{2}, \label{G2.5-2}\\
-\langle \mathcal{L}f, f\rangle
&\geq\lambda_{0}|\{\mathbf{I}-\mathbf{P}\}f|_{\nu}^{2}+|b|^{2}.\label{G2.5}
\end{align}

\medskip          
\section{Global Existence of Classical Solutions for the NS-VFP  System} % \eqref{C1}with Uniform Viscosity $\mu > 0$ in $\mathbb{R}^3$}

Throughout this section, we simplify notation by omitting the superscript $\mu$ on 
the unknowns. Our goal is to establish the global existence of classical solutions to the compressible NS-VFP system \eqref{C1} in the whole space $\mathbb{R}^3$.

\subsection{A priori estimates}
Now, we proceed to establish a uniform-in-time a priori estimate for $(\varrho, u, f)$ under the assumption that  
\begin{align}\label{G3.1}  
\sup_{0 \leq t < T} \big\{\|(\varrho,u)\|_{H^3} + \|f\|_{H^3_{x,v}} \big\} \leq \delta,  
\end{align}  
where $0 < \delta < 1$ is a sufficiently small constant independent of $\mu$. Here, $(\varrho, u, f)$ denotes the classical solution to the compressible NS-VFP system \eqref{C1} on the time interval $0 \leq t < T$, for some given $T > 0$.
By using the embedding $H^2(\mathbb{R}^3)\hookrightarrow L^\infty(\mathbb{R}^3)$ and  \eqref{G3.1}, we arrive at
\begin{align}\label{G3.2}
 \frac{1}{2}\leq \varrho+1\leq \frac{3}{2}.   
\end{align}

First, we establish the estimate of $\varrho$, $u$ and $f$ in $L^2$-norm.
\begin{lem}\label{L3.1}
For the classical solution $(\varrho, u,f)$ to the   NS-VFP system \eqref{C1}, there exist a positive constant $\lambda_1$ independent of $\mu$, such that
\begin{align}\label{G3.3}
& \frac{\rm d}{{\rm d}t}\big(P^{\prime}(1)\|\varrho\|_{L^2}^{2}+\|u\|_{L^2}^{2}+\|f\|_{L_{x,v}^2}^{2}\big)\nonumber\\
&\quad+\lambda_1\big(\|b-u\|_{L^2}^{2}
+\mu\|\nabla u\|_{L^2}^2+\|\{\mathbf{I}-\mathbf{P}\}f\|_{\nu}^{2}\big)
\lesssim \delta\|\nabla( \varrho,u,a,b)\|_{L^2}^2,
\end{align}
for any $0 \leq t < T$. 
\end{lem}
\begin{proof}
Multiplying the equations in \eqref{C1} by $P^{\prime}(1)\varrho$, $u$, and $f$, respectively, then integrating 
the results over $\mathbb{R}^3$ and $\mathbb{R}^3\times \mathbb{R}^3$, and summing the results yields  
\begin{align}\label{G3.4}
\frac{1}{2}\frac{\rm d}{{\rm d}t}&\Big(P^{\prime}(1)\|\varrho\|_{L^2}^{2}+\|u\|_{L^2}^{2}+\|f\|_{L_{x,v}^2}^{2}\Big)+\|b-u\|_{L^2}^{2}
+\mu \|\nabla u\|_{L^2}^2\nonumber\\
&+\int_{\mathbb{R}^3}(\varrho+1)\langle-\mathcal{L}\{\mathbf{I}-\mathbf{P}\}f,f\rangle {\rm d}x \nonumber\\
=\,&-\int_{\mathbb{R}^3}(u\cdot\nabla u)\cdot u{\rm d}x+\int_{\mathbb{R}^3} u\cdot
\Big\langle\frac{1}{2}vf, f\Big\rangle {\rm d}x-\int_{\mathbb{R}^3}a|u|^2{\rm d}x-\frac{P^{\prime}(1)}{2}\int_{\mathbb{R}^3}\varrho^2{\rm div}u\mathrm{d}x\nonumber\\
&\,-\int_{\mathbb{R}^3}\Big(\frac{P^{\prime}(\varrho+1) }{\varrho+1}-P^{\prime}(1)  \Big)\nabla\varrho\cdot u \mathrm{d}x-\mu\int_{\mathbb{R}^3}  \frac{\varrho}{1+\varrho}  \Delta u\cdot u   \mathrm{d}x\nonumber\\
&\,+\int_{\mathbb{R}^3}\varrho\Big\langle \mathcal{L}\mathbf{P}f-u\cdot\nabla_v f+\frac{1}{2}u\cdot vf+u\cdot v\sqrt{M},f\Big\rangle {\rm d}x\nonumber\\
\equiv:& \sum_{j=1}^7 I_j.
 \end{align}
For the term $I_1$, applying Hölder’s inequality, the assumption \eqref{G3.1}, and Lemma \ref{L2.2}, we obtain  
\begin{align}\label{G3.5}
I_1\lesssim \|u\|_{L^3}\|u\|_{L^6} \|\nabla u\|_{L^2}\lesssim \|u\|_{H^1}\|\nabla u\|_{L^2}^2
\leq \delta \|\nabla u\|_{L^2}^2.
\end{align}
For the terms $I_2$ and $I_3$, we can deduce from the decomposition \eqref{G2.4} that  
\begin{align}\label{G3.6}
I_2+I_3=\,&\frac{1}{2}\int_{\mathbb{R}^3}u\cdot\langle v\{\mathbf{I}-\mathbf{P}\}f,\{\mathbf{I}-\mathbf{P}\}f\rangle\mathrm{d}x+\int_{\mathbb{R}^3}u\cdot\langle v\mathbf{P}f,\{\mathbf{I}-\mathbf{P}\}f\rangle\mathrm{d}x\nonumber\\
&+\frac{1}{2}\int_{\mathbb{R}^3}u\cdot\langle v\mathbf{P}f,\mathbf{P}f\rangle\mathrm{d}x-\int_{\mathbb{R}^3}a|u|^2\mathrm{d}x\nonumber\\
\lesssim\,&\|b-u\|_{L^2}\|u\|_{L^3}\|a\|_{L^6}+\|(a,b)\|_{L^6}\|\{\mathbf{I}-\mathbf{P}\}f\|_{\nu}\|u\|_{L^3}+\|u\|_{L^{\infty}}\|\{\mathbf{I}-\mathbf{P}\}f\|_{\nu}^2\nonumber\\
\lesssim\,&\delta\big(\|b-u\|_{L^2}^2+\|\nabla(a,b)\|_{L^2}^2+ \|\{\mathbf{I}-\mathbf{P}\}f\|_{\nu}^2        \big),
\end{align}
where we have used  H\"{o}lder’s and Young's inequalities, the assumption \eqref{G3.1} and Lemmas \ref{L2.1}--\ref{L2.2}.
Applying   similar arguments to those done in  \eqref{G3.5},  we have
\begin{align} 
I_4\lesssim\,&  \|\varrho\|_{L^3}\|\varrho\|_{L^6}\|\nabla u\|_{L^2}\lesssim \|\varrho\|_{H^1}\|\nabla(\varrho,u)\|_{L^2}^2 \lesssim \delta \|\nabla(\varrho,u)\|_{L^2}^2,\label{G3.7-1}\\
I_5\lesssim\,&  \|\varrho\|_{L^3}\| u\|_{L^6}\|\nabla\varrho\|_{L^2}\lesssim \|\varrho\|_{H^1}\|\nabla(\varrho,u)\|_{L^2}^2 \lesssim \delta \|\nabla(\varrho,u)\|_{L^2}^2.\label{G3.7}
\end{align}

For the term $I_6$, utilizing
integration by parts, H\"{o}lder’s inequality,
and Lemmas \ref{L2.1}--\ref{L2.2}, we have
\begin{align}\label{G3.8}
I_6=&\,\mu\int_{\mathbb{R}^3}\frac{\varrho}{1+\varrho} {\rm  div}u \nabla u {\rm d}x+\mu\int_{\mathbb{R}^3}\nabla\Big(\frac{\varrho}{1+\varrho}  \Big)\cdot \nabla u\cdot u{\rm d}x\nonumber\\
\lesssim&\, \mu \|\varrho\|_{L^\infty}\|\nabla u\|_{L^2}^2+\mu \Big\|\nabla\Big(\frac{\varrho}{1+\varrho}   \Big)\Big\|_{L^3}\|\nabla u\|_{L^2}\|u\|_{L^6}\nonumber\\
\lesssim&\, \mu \|\varrho\|_{H^2}\|\nabla u\|_{L^2}^2\nonumber\\
\lesssim&\, \mu\delta \|\nabla u\|_{L^2}^2.
\end{align}

For the remaining term $I_7$, by applying the decomposition \eqref{G2.4} and utilizing H\"{o}lder’s, Sobolev’s, and Young’s inequalities along with Lemmas \ref{L2.1}–\ref{L2.2}, we derive
\begin{align}\label{G3.9}
I_7=\,&\int_{\mathbb{R}^3}\varrho(u-b)\cdot b\mathrm{d}x+\frac{1}{2}\int_{\mathbb{R}^3}\varrho u\cdot\langle vf,f\rangle\mathrm{d}x\nonumber\\
\lesssim\,& \|\varrho\|_{L^3}\|b-u\|_{L^2}\|b\|_{L^6}+\frac{1}{2}\int_{\mathbb{R}^3}\varrho u\cdot\langle v\{\mathbf{I}-\mathbf{P}\}f,\{\mathbf{I}-\mathbf{P}\}f\rangle\mathrm{d}x\nonumber\\
&+\int_{\mathbb{R}^3}\langle v\mathbf{P}f,\{\mathbf{I}-\mathbf{P}\}f\rangle\mathrm{d}x+\int_{\mathbb{R}^3}\varrho u \cdot ab\mathrm{d}x\nonumber\\
\lesssim\,& \|\varrho\|_{H^1}\|b-u\|_{L^2}\|\nabla b\|_{L^2}+\|\varrho\|_{H^2}\|u\|_{H^2}\|\{\mathbf{I}-\mathbf{P}\}\|_{\nu}^2
\nonumber\\
&+\|\varrho\|_{H^1}\|u\|_{H^1}\|(a,b)\|_{L^6}\|\{\mathbf{I}-\mathbf{P}\}\|_{\nu}
+\|\varrho\|_{H^1}\|u\|_{H^1}\|\nabla(a,b)\|_{L^2}\nonumber\\
\lesssim\,&\delta\big( \|b-u\|_{L^2}^2+\|\nabla(a,b)\|_{L^2}^2+\|\{\mathbf{I}-\mathbf{P}\}\|_{\nu}^2\big).
\end{align}
 Then, by substituting the estimates \eqref{G3.5}–\eqref{G3.9} into \eqref{G3.4} and utilizing \eqref{G2.5} and \eqref{G3.2}, we ultimately derive \eqref{G3.3}.
\end{proof}

Similarly to \cite[Lemma 2.2]{DL-KRM-2013}, we present the $L^2$-norm estimates of $\nabla^k(\varrho, u, f)$ for $k = 1, 2,3$.

\begin{lem}\label{L3.2}
For the classical solution $(\varrho, u,f)$ to the   NS-VFP system \eqref{C1}, there exist a positive constant $\lambda_2$ independent of $\mu$, such that
\begin{align}\label{G3.10}
& \frac{\rm d}{{\rm d}t}\sum_{1\leq |\alpha|\leq 3}
\bigg(\bigg\|\frac{\sqrt{P^\prime(1+\varrho)}}{1+\varrho}\partial^\alpha \varrho\bigg\|_{L^2}^2+\|\partial^\alpha u\|_{L^2}^2+\|\partial^\alpha f\|_{L_{x,v}^2}^2\bigg)\nonumber\\
&\quad +\lambda_2\sum_{1\leq|\alpha|\leq 3}\big(\|\partial^\alpha(b-u)\|_{L^2}^2+\mu\|\nabla\partial^\alpha u\|_{L^2}^2+\|\partial^\alpha\{\mathbf{I}-\mathbf{P}\}f\|_{\nu}^2\big) 
 \lesssim \delta\|\nabla(\varrho,u,a,b)\|_{H^2}^2,
\end{align}
for any $0 \leq t<T$. 
\end{lem}
\begin{proof}
Applying $\partial^\alpha$ with $1\leq |\alpha|\leq 3$
to the equations in \eqref{C1}, we have 
\begin{equation}\label{G3.11}
\left\{\begin{aligned}
&\partial_{t} (\partial^\alpha\varrho)+(1+\varrho){\rm div}\partial^\alpha u +u\cdot\nabla\partial^\alpha\varrho
  =[-\partial^\alpha,\varrho{\rm div}]u+[-\partial^\alpha,u\cdot\nabla]\varrho,  \\
&\partial_{t}(\partial^\alpha u)+u\cdot\nabla(\partial^\alpha u)+\frac{P^{\prime}(1+\varrho)}{1+\varrho}\nabla\partial^\alpha\varrho+\bigg[\partial^\alpha,\frac{P^{\prime}(1+\varrho)}{1+\varrho}\nabla\bigg]\varrho
-\mu\partial^\alpha \Delta u \\
&\quad =[-\partial^\alpha,u\cdot\nabla]u+\partial^\alpha(b-u-au)-\mu\partial^\alpha\Big(\frac{\varrho}{1+\varrho}\Delta u\Big),\\
&\partial_{t}(\partial^\alpha f)+v\cdot\nabla (\partial^\alpha f)+u\cdot\nabla_v (\partial^{\alpha}f)-\partial^{\alpha}u\cdot v\sqrt{M}-\mathcal{L}(\partial^\alpha f)\\
&\quad =[-\partial^\alpha,u\cdot\nabla_v]f+\frac{1}{2}\partial^\alpha\big((1+\varrho)u\cdot vf\big)+\partial^{\alpha}\big(\varrho(\mathcal{L}f-u\nabla_v f+u\cdot v\sqrt{M})\big),\\
 \end{aligned}
 \right.
\end{equation}
where  $[A, B]=AB-BA$ denotes the commutator of two operators $A$ and $B$.
Multiplying the equations \eqref{G3.11}$_1$, \eqref{G3.11}$_2$, and \eqref{G3.11}$_3$ by $\frac{P^{\prime}(1+\varrho)}{(1+\varrho)^2}\partial^\alpha \varrho$, $\partial^\alpha u$, and $\partial^\alpha f$, respectively, integrating the  results, and adding all these products up yields 
\begin{align}\label{G3.12}
\frac12\frac {\rm d}{{\rm d}t}&\bigg(\bigg\|\frac{\sqrt{P^\prime(1+\varrho)}}{1+\varrho}\partial^\alpha \varrho\bigg\|_{L^2}^2+\|\partial^\alpha u\|_{L^2}^2+\|\partial^\alpha f\|_{L_{x,v}^2}^2\bigg)
+\|\partial^{\alpha}(b-u)\|_{L^2}^{2}+  \mu  \|\nabla\partial^\alpha u\|_{L^2}^2 
 \nonumber\\
&+\int_{\mathbb{R}^3}\left\langle-\mathcal{L}\{\mathbf{I}-\mathbf{P}\}\partial^{\alpha}f,
\partial^{\alpha}f\right\rangle {\rm d}x                  \nonumber\\
=\,&\int_{\mathbb{R}^3}\langle[-\partial^{\alpha},u\cdot\nabla_{v}]f,\partial^{\alpha}f\rangle {\rm d}x 
+\frac{1}{2}\int_{\mathbb{R}^3}
\langle\partial^{\alpha}\big((1+\varrho)u\cdot vf\big),\partial^{\alpha}f\rangle {\rm d}x\nonumber\\
&+\int_{\mathbb{R}^3}
\langle\partial^{\alpha}\big(\varrho(\mathcal{L}f-u\nabla_v f+u\cdot v\sqrt{M})\big),\partial^{\alpha}f\rangle {\rm d}x+\frac{1}{2}\int_{\mathbb{R}^3}\partial_t \bigg[\frac{P^\prime(1+\varrho)}{(1+\varrho)^2}\bigg] |\partial^\alpha\varrho|^2{\rm d}x\nonumber\\
&+\int_{\mathbb{R}^3} \frac{P^\prime(1+\varrho)}{(1+\varrho)^2} \big([-\partial^\alpha,\varrho{\rm div}]u+[-\partial^\alpha,u\cdot\nabla]\varrho     \big)\partial^\alpha\varrho{\rm d}x+\int_{\mathbb{R}^3}\partial^{\alpha}(au)\cdot\partial^{\alpha}u\mathrm{d}x \nonumber\\
&+\frac{1}{2}\int_{\mathbb{R}^3}|\partial^{\alpha} u|^2{\rm div}u\mathrm{d}x-\int_{\mathbb{R}^3}\bigg[\partial^\alpha,\frac{P^{\prime}(1+\varrho)}{1+\varrho}\nabla\bigg]\varrho\cdot\partial^\alpha u{\rm d}x\nonumber\\
&-\int_{\mathbb{R}^3}
 \nabla \Big(\frac{P^{\prime}(1+\varrho)}{1+\varrho}-{P^{\prime}(1)}\Big)\partial^\alpha\varrho \cdot\partial^\alpha u\mathrm{d}x+\int_{\mathbb{R}^3}[-\partial^\alpha,u\cdot\nabla]u\cdot\partial^\alpha u{\rm d}x\nonumber\\
&-\mu\sum_{0\leq|\beta|<|\alpha|}C_{\alpha,\beta}
\int_{\mathbb{R}^3}\partial^{\alpha-\beta}\Big(\frac{\varrho}{1+\varrho}\Big)\partial^\beta\Delta u\cdot\partial^\alpha u \mathrm{d}x-\mu\int_{\mathbb{R}^3}\frac{\varrho}{1+\varrho}\partial^\alpha \Delta u\cdot \partial^\alpha u{\rm d}x\nonumber\\
\equiv:\,&\sum_{i=1}^{12} J_i,
\end{align}
where $C_{\alpha,\beta}$ denote  constants depending only on $\alpha $ and $\beta$.

For the terms $J_1$ and $J_2$, using the assumption \eqref{G3.1}, Lemmas \ref{L2.1}–\ref{L2.3}, as well as H\"{o}lder’s and Young's inequalities, we have  
\begin{align}\label{G3.13}
J_1\lesssim&\, \|\nabla_v
\partial^\alpha f\|_{L_{x,v}^2}\|\nabla u\|_{H^2}\|\nabla f\|_{L_v^2(H ^2)}\nonumber\\ 
\lesssim&\,\|u\|_{H^3}\bigg(\sum_{1\leq|\alpha|\leq 3}\|\partial^\alpha\{\mathbf{I}-\mathbf{P}\}f\|_{\nu}^2+\|\nabla(a,b)\|_{H^2}^2\bigg)\nonumber\\
\lesssim&\,\delta\bigg(\sum_{1\leq|\alpha|\leq 3}\|\partial^\alpha\{\mathbf{I}-\mathbf{P}\}f\|_{\nu}^2+\|\nabla(a,b)\|_{H^2}^2\bigg),\\\label{G3.14}
J_2\lesssim&\, \big\|\nabla\big((1+\varrho)u\big)\big\|_{H^2}\|\nabla f\|_{L_v^2(H^2)}\|v\partial^\alpha f\|_{L_{x,v}^2}\nonumber\\
\lesssim&\, \big(\|\nabla u\|_{H^2}+\|\nabla\varrho\|_{H^2}\|\nabla u\|_{H^2}\big)
\bigg(\sum_{1\leq|\alpha|\leq 3}\|\partial^\alpha\{\mathbf{I}-\mathbf{P}\}f\|_{\nu}^2+\|\nabla(a,b)\|_{H^2}^2\bigg)\nonumber\\
\lesssim&\,\delta\bigg(\sum_{1\leq|\alpha|\leq 3}\|\partial^\alpha\{\mathbf{I}-\mathbf{P}\}f\|_{\nu}^2+\|\nabla(a,b)\|_{H^2}^2\bigg).
\end{align}
For the term $J_{3}$, it follows from the decomposition  \eqref{G2.4} and  the assumption \eqref{G3.1}, together with Lemmas \ref{L2.1}--\ref{L2.3}, that  
\begin{align}\label{G3.15}
J_{3}\lesssim &\,\Big\|\partial^\alpha\Big(\varrho\sqrt{M}\nabla_v\Big(\frac{f}{\sqrt{M}}\Big)\Big)\Big\|_{L_{x,v}^2}
\Big\|\partial^\alpha\Big(\sqrt{M}\nabla_v\Big(\frac{f}{\sqrt{M}}\Big)\Big)\Big\|_{L_{x,v}^2} \nonumber\\
&+\|\partial^\alpha(\varrho u f)\|_{L_{x,v}^2}\|\nabla_v \partial^\alpha f\|_{L_{x,v}^2}+\|\partial^\alpha(\varrho u)\|_{L^2}\|\partial^\alpha b\|_{L^2}\nonumber\\
\lesssim&\,\big(\|\nabla\varrho\|_{H^2}+\|\nabla u\|_{H^2}\|\nabla\varrho\|_{H^2}\big)
\bigg(\sum_{1\leq|\alpha|\leq 3}\|\partial^\alpha\{\mathbf{I}-\mathbf{P}\}f\|_{\nu}^2+\|\nabla(u,a,b)\|_{H^2}^2\bigg)\nonumber\\
\lesssim&\,\delta\bigg(\sum_{1\leq|\alpha|\leq 3}\|\partial^\alpha\{\mathbf{I}-\mathbf{P}\}f\|_{\nu}^2+\|\nabla( u,a,b)\|_{H^2}^2\bigg).
\end{align}

For the term $J_4$, according to \eqref{C1}$_1$,  the assumption  \eqref{G3.1}, and Lemma \ref{L2.1}, one has
\begin{align}\label{NJK-aaaa}
\sup_{0\leq t<T,\,x\in\mathbb{R}^3} |\partial_t\varrho(t,x)|\lesssim (\|\varrho\|_{L^\infty}+1)\|\nabla u\|_{L^\infty}+\|u\|_{L^\infty}\|\nabla \varrho\|_{L^\infty}\lesssim \|u\|_{H^3}, 
\end{align}
this estimate consequently implies that  
\begin{align}\label{G3.16}
J_4\lesssim&\, \|u\|_{H^3}\|\nabla\varrho\|_{H^2}^2\lesssim\delta  \|\nabla\varrho\|_{H^2}^2.  
\end{align}
For the terms $J_5, \dots, J_{11}$, applying H\"{o}lder's and Young's inequalities,   \eqref{G3.1} and Lemmas \ref{L2.1}--\ref{L2.3}, we obtain
\begin{align}
J_5\lesssim&\, (1+\|\varrho\|_{L^\infty})\|\nabla \varrho\|_{H^2}\big(\|\nabla\varrho\|_{L^\infty} \|\nabla u\|_{H^2}+\|\nabla u\|_{L^\infty}\|\nabla\varrho\|_{H^2}
\big) \nonumber\\
\lesssim &\,\delta \|\nabla(\varrho,u)\|_{H^2}^2,\label{G3.17-1}\\
J_6\lesssim&\,  \|\nabla(au)\|_{H^2}\|\nabla u\|_{H^2} 
\lesssim  \|\nabla a\|_{H^2}\|\nabla u\|_{H^2}^2 
\lesssim  \delta \|\nabla u\|_{H^2}^2,\label{G3.17-2}\\
J_7\lesssim&\, \|\nabla u\|_{L^\infty}\|\nabla u\|_{H^2}^2 
\lesssim  \|u\|_{H^3}\|\nabla u\|_{H^2}^2 
\lesssim \delta \|\nabla u\|_{H^2}^2,\label{G3.17-3}\\
J_8+J_9\lesssim&\, \|\nabla\varrho\|_{L^\infty}\|\nabla\varrho\|_{H^2}\|\nabla u\|_{H^2} 
\lesssim  \|u\|_{H^3}\|\nabla\varrho\|_{H^2}^2 
 \lesssim  \delta \|\nabla\varrho\|_{H^2}^2,\label{G3.17-4}\\
J_{10}\lesssim&\, \|\nabla u\|_{L^\infty}\|\nabla u\|_{H^2}^2 
\lesssim  \delta \|\nabla u\|_{H^2}^2,\label{G3.17-5}\\
J_{11}\lesssim&\, \mu\sum_{0\leq |\beta|<|\alpha|} \Big\| \partial^{\alpha-\beta} \Big(\frac{\varrho}{1+\varrho}\Big)\Big\|_{L^3}\|\partial^\beta \Delta u\|_{L^2}\|\partial^\alpha u\|_{L^6} \nonumber\\
\lesssim &\, \mu \|\varrho\|_{H^3}\|\nabla^2 u\|_{H^2}^2
\lesssim \mu \delta \|\nabla^2 u\|_{H^2}^2.\label{G3.17}
\end{align}
For the last term $J_{12}$, applying integration by parts and Lemmas \ref{L2.1}--\ref{L2.2}, we have  
\begin{align}\label{G3.18}
J_{12}=&\,\mu\int_{\mathbb{R}^3}\nabla \Big(\frac{\varrho}{1+\varrho}  \Big) \cdot \partial^\alpha \nabla u\cdot \partial^\alpha u{\rm d}x+\mu \int_{\mathbb{R}^3} \frac{\varrho}{1+\varrho}    \partial^\alpha \nabla u\cdot \partial^\alpha {\rm div} u{\rm d}x \nonumber\\
\lesssim&\, \mu \|\nabla\varrho\|_{L^3}\|\partial^\alpha\nabla  u\|_{L^2}\|\partial^\alpha u\|_{L^6}+\mu \|\varrho\|_{L^\infty}\|\nabla\partial^\alpha u\|_{L^2}^2\nonumber\\
\lesssim&\, \mu \|\varrho\|_{H^3}\|\nabla^2 u\|_{H^2}^2\nonumber\\
\lesssim&\, \mu \delta \|\nabla^2 u \|_{H^2}^2.
\end{align}
Substituting the estimates \eqref{G3.13}–\eqref{G3.18} into \eqref{G3.12}, summing over $1 \leq |\alpha| \leq 3$, and applying \eqref{G2.5} and \eqref{G3.2} once again, we ultimately obtain the desired estimate \eqref{G3.10}.
\end{proof}

To incorporate the norm $\|\nabla (a,b)\|_{H^2}$ into the energy dissipation rate, as done in \cite{LMW-SIAM-2017,LNW-2025-preprint}, we analyze the following equations for $a$ and $b$:

\begin{equation}\label{G3.19}
\left\{\begin{aligned}
&\partial_{t}a+\nabla\cdot b=0,\\
&\partial_{t} b_i+\partial_{i} a+\sum_{j=1}^3\partial_j\Gamma_{ij}(\{\mathbf{I}-\mathbf{P}\}f)=(1+\varrho)(u_i-b_i)+(1+\varrho)u_ia,  \\
&\partial_{i}b_j+\partial_j b_i-(1+\varrho)(u_ib_j+u_jb_i)=-\partial_t \Gamma_{ij}(\{\mathbf{I}-\mathbf{P}\}f)+\Gamma_{ij}(\mathfrak{l}+\mathfrak{r}+\mathfrak{s}),  
 \end{aligned}
 \right.
\end{equation}
for $1\leq i,j\leq 3$,
where $\mathfrak{l},\mathfrak{r}$ ,and $\mathfrak{s}$ are given by
\begin{align}
\mathfrak{l}& :=\mathcal{L}\{\mathbf{I}-\mathbf{P}\}f-v\cdot\nabla\{\mathbf{I}-\mathbf{P}\}f, \label{lrs-1}\\
\mathfrak{r}& :=-u\cdot\nabla_v\{\mathbf{I}-\mathbf{P}\}f+\frac{1}{2}u\cdot v\{\mathbf{I}-\mathbf{P}\}f,\label{lrs-2}\\
\mathfrak{s}& :=\frac{\varrho}{\sqrt{M}}\nabla_v\cdot\big(\nabla_v\big(\sqrt{M}\{\mathbf{I}-\mathbf{P}\}f\big)
+v\sqrt{M}\{\mathbf{I}-\mathbf{P}\}f-u\sqrt{M}\{\mathbf{I}-\mathbf{P}\}f             \big), \label{lrs-3}
\end{align}
and 
\begin{align}\label{gammaa}
\Gamma_{ij}(\cdot):=\big\langle (v_iv_j-1)\sqrt{M},\cdot   \big\rangle 
\end{align}
is the moment functional. See \cite{LMW-SIAM-2017,LNW-2025-preprint} for
the derivation of equations \eqref{G3.19}. 
Define a temporal functional $\mathcal{E}_0(t)$ by
\begin{align*}
\mathcal{E}_0(t):=\sum_{|\alpha|\leq 2}\sum_{i,j=1}^3\int_{\mathbb{R}^3}\partial^\alpha(\partial_ib_j
+\partial_jbi)\partial^{\alpha}\Gamma_{ij}(\{\mathbf{I}-\mathbf{P}\}f)\mathrm{d}x-\sum_{|\alpha|\leq 2}\int_{\mathbb{R}^3}\partial^\alpha a\partial^{\alpha}\nabla\cdot b\mathrm{d}x. 
\end{align*}

Following a similar approach to that in \cite{LMW-SIAM-2017,LNW-2025-preprint}, we can establish the following lemma. The proof is omitted here for brevity.
\begin{lem}\label{L3.3}
For the classical solution $(\varrho, u,f)$ to the   NS-VFP system \eqref{C1}, there exist a positive constant $\lambda_3$ independent of $\mu$, such that
\begin{align}\label{G3.24}
\frac{\rm d}{{\rm d}t}\mathcal{E}_{0}(t)+\lambda_3\|\nabla(a,b)\|_{H^{1}}^{2}
\lesssim&\, \|\{\mathbf{I}-\mathbf{P}\}f\|_{L_{v}^{2}(H^{3})}^{2}+\|b-u\|_{H^{2}}^{2} , 
\end{align} 
for any $0 \leq t < T$.   
\end{lem}

Next, we give the estimate of $\|\nabla\varrho\|_{H^2}$.
\begin{lem}\label{L3.4}
For the classical solution $(\varrho, u,f)$ to the   NS-VFP system \eqref{C1},  there exist a positive constant $\lambda_4$ independent of $\mu$, such that
\begin{align}\label{G3.25}
\frac{\rm d}{{\rm d}t}\sum_{|\alpha|\leq 2}\int_{\mathbb{R}^3}\partial^\alpha u\cdot\partial^\alpha\nabla\varrho\mathrm{d}x+\lambda_4\|\nabla\varrho\|_{H^{2}}^{2}
\lesssim&\, \|\nabla u\|_{H^2}^2+\mu\delta \|\nabla^2 u\|_{H^2}^2+\|b-u\|_{H^{2}}^{2} , 
\end{align} 
for any $0 \leq t <T$. 
\end{lem}

\begin{proof}
Applying $\partial^\alpha$ with $|\alpha|\leq 2$ to \eqref{C1}$_2$, one computes
\begin{align}\label{G3.26}
P^\prime(1)\|\nabla\partial^\alpha\varrho\|_{L^2}^2=
\, &-\int_{\mathbb{R}^3}\nabla\partial^\alpha\varrho\cdot\partial^\alpha\partial_t u\mathrm{d}x+\int_{\mathbb{R}^3}\nabla\partial^\alpha\varrho\cdot\partial^\alpha(b-u)\mathrm{d}x   \nonumber\\
&+\int_{\mathbb{R}^3}\nabla\partial^\alpha\varrho\cdot\partial^\alpha(au)\mathrm{d}x-\int_{\mathbb{R}^3}\nabla\partial^\alpha\varrho\cdot\partial^\alpha(u\cdot\nabla u)\mathrm{d}x\nonumber\\
&-\int_{\mathbb{R}^3}\nabla\partial^\alpha\varrho\cdot\partial^\alpha\Big(\Big(\frac{P^{\prime}(1)}{1+\varrho}-P^{\prime}(1)       \Big)\nabla\varrho\Big)\mathrm{d}x+\mu\int_{\mathbb{R}^3}\nabla\partial^\alpha\varrho\cdot\partial^\alpha\Big(\frac{ \Delta u}{1+\varrho}        \Big)\mathrm{d}x\nonumber\\
\equiv:\,&\sum_{j=1}^{6}K_j.
\end{align}
Regarding the term $K_{1}$, by using \eqref{C1}$_2$ and applying Lemmas \ref{L2.1}–\ref{L2.2}, we have
\begin{align}\label{G3.27}
K_{1}=& -\frac{{\rm d}}{{\rm d}t}\int_{\mathbb{R}^3}\nabla  \partial^\alpha\varrho\cdot\partial^\alpha u\mathrm{d}x+\int_{\mathbb{R}^3}\nabla  \partial^\alpha\partial_t\varrho\cdot\partial^\alpha u\mathrm{d}x\nonumber\\
=& -\frac{{\rm d}}{{\rm d}t}\int_{\mathbb{R}^3}\nabla  \partial^\alpha\varrho\cdot\partial^\alpha u\mathrm{d}x+\int_{\mathbb{R}^3}\partial^\alpha\big((\varrho+1){\rm div} u+\nabla\varrho\cdot u            \big) \partial^\alpha{\rm div} u\mathrm{d}x\nonumber\\
\lesssim& -\frac{{\rm d}}{{\rm d}t}\int_{\mathbb{R}^3}\nabla  \partial^\alpha\varrho\cdot\partial^\alpha u\mathrm{d}x+\|\nabla  u
\|_{H^2}^2+\|\varrho\|_{H^3}\|\nabla u\|_{H^2}^2\nonumber\\
\lesssim& -\frac{{\rm d}}{{\rm d}t}\int_{\mathbb{R}^3}\nabla  \partial^\alpha\varrho\cdot\partial^\alpha u\mathrm{d}x+\|\nabla u\|_{H^2}^2.
\end{align}
By utilizing \eqref{G3.1}--\eqref{G3.2},   H\"{o}lder’s and Young's inequalities, and invoking Lemmas \ref{L2.1}--\ref{L2.2}, we arrive at 
\begin{align}
K_{2}\leq\,&\frac{1}{5}P^{\prime}(1)\|\nabla\partial^\alpha\varrho\|_{L^2}^2+C\|u-b\|_{H^2}^2, \label{G3.28-1} \\
K_{3}\leq\,&\frac{1}{5}P^{\prime}(1)\|\nabla\partial^\alpha\varrho\|_{L^2}^2+C\|a\|_{H^2}^2\|\nabla u\|_{H^2}^2\nonumber\\
\leq&\,\frac{1}{5}P^{\prime}(1)\|\nabla\partial^\alpha\varrho\|_{L^2}^2+C\delta\|\nabla u\|_{H^2}^2, \label{G3.28-2}\\
K_{4}\leq\,&\frac{1}{5}P^{\prime}(1)\|\nabla\partial^\alpha\varrho\|_{L^2}^2+C\|u\|_{H^2}^2\|\nabla u\|_{H^2}^2\nonumber\\
\leq&\,\frac{1}{5}P^{\prime}(1)\|\nabla\partial^\alpha\varrho\|_{L^2}^2+C\delta\|\nabla u\|_{H^2}^2,\label{G3.28-3}\\
K_{5}\leq\,& \frac{1}{5}P^{\prime}(1)\|\nabla\partial^\alpha\varrho\|_{L^2}^2+C\|\varrho\|_{H^3}^2\|\nabla\varrho\|_{H^2}^2\nonumber\\
\leq&\,\frac{1}{5}P^{\prime}(1)\|\nabla\partial^\alpha\varrho\|_{L^2}^2+C\delta\|\nabla \varrho\|_{H^2}^2. \label{G3.28}
\end{align}
For the remaining term $K_6$, applying the assumption \eqref{G3.1}, Lemma \ref{L3.1}, and Young's inequality, and noticing that $0<\mu<1$, we deduce that
\begin{align}\label{G3.29}
K_{6}\lesssim\,&  \mu \|\nabla\partial^\alpha\varrho\|_{L^2}\|\nabla \varrho\|_{H^2}\|\nabla\Delta u\|_{H^1}^2 
\lesssim  \mu \|u\|_{H^3}\|\nabla\varrho \|_{H^2}^2 
\lesssim \delta \|\nabla\varrho\|_{H^2}^2.    
\end{align}
Putting the estimates \eqref{G3.27}–\eqref{G3.29} into \eqref{G3.26} and summing the results over all multi-indices $\alpha$ satisfying $|\alpha| \leq 2$, we finally obtain \eqref{G3.25}.
\end{proof}
Now, we proceed to estimate the mixed space-velocity derivatives of $ f $, that is, $ \partial^\alpha_x \partial^\beta_v f $. Since $ \|\partial^\alpha_x \partial^\beta_v \mathbf{P}f\|_{L_{x,v}^2} \lesssim \|\partial^\alpha f\|_{L_{x,v}^2} $, it suffices to estimate $ \|\partial_x^\alpha \partial_v^\beta \{\mathbf{I}-\mathbf{P}\}f\|_{L_{x,v}^2} $. By applying the operator $ \{\mathbf{I}-\mathbf{P}\} $ to \eqref{C1}$_3$ and utilizing the decomposition given in \eqref{G2.4}, we obtain  
\begin{align}\label{G3.30}
&\partial_t\{\mathbf{I}-\mathbf{P}\}f+v\cdot\nabla\{\mathbf{I}-\mathbf{P}\}f
+(1+\varrho)u\cdot\nabla_v\{\mathbf{I}-\mathbf{P}\}f-\frac{1}{2}(1+\varrho)u\cdot v\{\mathbf{I}-\mathbf{P}\}f\nonumber\\
\,&\quad =  (\varrho+1)\mathcal{L}\{\mathbf{I}-\mathbf{P}\}f\nonumber\\
 &\qquad +\mathbf{P}\Big(v\cdot\nabla\{\mathbf{I}-\mathbf{P}\}f+(1+\varrho)u\cdot\nabla_v \{\mathbf{I}-\mathbf{P}\}f
-\frac{1}{2}(1+\varrho)u\cdot v\{\mathbf{I}-\mathbf{P}\}f\Big)\nonumber\\
& \qquad -\{\mathbf{I}-\mathbf{P}\}\Big(v\cdot\nabla\mathbf{P}f+(1+\varrho)u\cdot\nabla_v \mathbf{P}f-\frac{1}{2}(1+\varrho)u\cdot v\mathbf{P}f\Big),
\end{align}
where we  have used the following facts:
\begin{align*}
\{\mathbf{I}-\mathbf{P}\}(u\cdot v\sqrt{M})=0, \quad \{\mathbf{I}-\mathbf{P}\}\big((1+\varrho)u\cdot v\sqrt{M}\big)=0,\quad \{\mathbf{I}-\mathbf{P}\}\mathcal{L}f=\mathcal{L}\{\mathbf{I}-\mathbf{P}\}f.
\end{align*}
The following lemma can be proved in the same spirit of   \cite{LMW-SIAM-2017,LNW-2025-preprint}
and we omit the details here.
\begin{lem}\label{L3.5}
Let $1\leq k\leq 3$.
For the classical solution $(\varrho, u,f)$ to the   NS-VFP system \eqref{C1}, there exist a positive constant $\lambda_5$ independent of $\mu$,
such that
 \begin{align}\label{G3.31}
& \frac{\mathrm{d}}{\mathrm{d} t}\sum_{1\leq k\leq3}C_k\sum_{\substack{|\beta|=k \\
|\alpha|+|\beta| \leq 3}}\|\partial^\alpha_x\partial^\beta_v\{\mathbf{I}-\mathbf{P}\} f\|_{L_{x,v}^2}^2+\lambda_5\sum_{\substack{1\leq|\beta|\leq 3 \\
|\alpha|+|\beta| \leq 3}}\|\partial^\alpha_x\partial^\beta_v\{\mathbf{I}-\mathbf{P}\} f\|_\nu^2\nonumber \\
&\quad  \lesssim   \|\nabla (a, b)\|_{H^2}^2+\sum_{|\alpha| \leq 3}\left\|\partial^\alpha\{\mathbf{I}-\mathbf{P}\} f\right\|_\nu^2 ,
\end{align}
 for any $0 \leq t <T$,  and $C_k$ are some positive constants.
\end{lem}

\subsection{Proof of Theorem \ref{T1.1}}
In this subsection, we aim to establish the global existence of solutions to the   NS-VFP system \eqref{C1} in $\mathbb{R}^3$.
\begin{proof}[Proof of Theorem \ref{T1.1}]
With Lemmas \ref{L3.1}–\ref{L3.5} established, we are now in a position to prove the global existence of classical solutions to the compressible NS-VFP system \eqref{C1} with the uniform $\mu>0$. To this end, we define the temporal energy functional
\begin{align}\label{G3.32}
\mathcal{E}(t):=\,&\|  u  \|_{ {H}^3}^2+\|  f \|_{ {H}^3_{x,v}}^2+\Big\|\sqrt{P^\prime(1)}\varrho\Big\|_{L^2}^2
+\sum_{1\leq |\alpha|\leq 3}
 \bigg\|\frac{\sqrt{P^\prime(1+\varrho)}}{1+\varrho}\partial^\alpha \varrho\bigg\|_{L^2}^2+\tau_{1}\mathcal{E}_0(t)
\nonumber\\
&+\tau_2\sum_{|\alpha|\leq 2}\int_{\mathbb{R}^3}\partial^\alpha u\cdot\partial^\alpha\nabla\varrho\mathrm{d}x+\tau_{3}\sum_{1\leq k\leq3}C_k\sum_{\substack{|\beta|=k \\
|\alpha|+|\beta| \leq 3}}\|\partial^\alpha_x\partial^\beta_v\{\mathbf{I}-\mathbf{P}\} f\|_{L_{x,v}^2}^2,   
\end{align}
where $0 < \tau_1, \tau_2, \tau_3 \ll 1$ and $\tau_3 \ll \tau_1$, with all parameters representing sufficiently small positive constants.

Under the assumption \eqref{G3.1}, recalling the definition of $\mathcal{E}_0(t)$, we find that
\begin{align}\label{G3.33}
\mathcal{E}(t)\backsim  \|(\varrho, u )(t)\|_{{H}^3}^2+\|f(t)\|_{{H}^3_{x,v}}^2,   
\end{align}
uniformly for all $0\leq t< T$.
By utilizing the effective dissipation mechanism, we obtain  
\begin{align*}  
\|\nabla u(t)\|_{H^2}^2 \lesssim \|b - u(t)\|_{H^3}^2 + \|\nabla b(t)\|_{H^2}^2,  
\end{align*}  
which implies that $\|\nabla u(t)\|_{H^2}^2 \leq \mathcal{D}(\varrho, u, f)(t)$. The sum of the 
inequalities \eqref{G3.3}, \eqref{G3.10}, $\tau_1\times $ \eqref{G3.24}, $\tau_2\times$ \eqref{G3.25}, and $\tau_3\times$ \eqref{G3.31} gives
\begin{align}\label{G3.34}
\frac{{\rm d}}{{\rm d}t}\mathcal{E}(t)+\lambda_6 \mathcal{D}(\varrho,u
,f)(t)+\mu \|\nabla u (t)\|_{H^3}^2 \lesssim \delta     \mathcal{D}(\varrho,u
,f)(t),
\end{align}
for any $0\leq t<T$, where the general positive constant $\lambda_6<\min\{\lambda_1,\lambda_2,\tau_1\lambda_3,\tau_2\lambda_4,\tau_3\lambda_5\}$.
Hence, for sufficiently small $0 < \delta < 1$, the time integration of \eqref{G3.34} implies that there exists a positive constant $\lambda_7 > 0$, independent of $\mu$, such that  
\begin{align}  \label{G3.35}
\mathcal{E}(t) + \lambda_7 \int_0^t \mathcal{D}(\varrho, u, f)(\tau) \, {\rm d}\tau \leq \mathcal{E}(0),  
\end{align}  
for all $0 \leq t < T$.
Thanks to \eqref{G3.33}, the validity of \eqref{G3.1} can be established by selecting \begin{align*}\mathcal{E}(0)\sim \|(\varrho_0,u_0)\|_{H^3}^2+\|f_0\|_{H_{x,v}^3}^2,\end{align*} which is sufficiently small.

The existence and uniqueness of the local solution for the compressible NS-VFP system \eqref{C1} can be established through the application of linearization   techniques and the Banach contraction mapping principle (cf. \cite{Gy-CPAM-2004, GY-iumj-2004,LMW-SIAM-2017}). For brevity, we omit the detailed proof here. By combining the continuity argument with the local existence result and the uniform a priori estimates given in \eqref{G3.35}, we are able to deduce the global existence of classical solutions $(\varrho, u, f)$ to the system \eqref{C1}. Furthermore, by applying the maximum principle, it follows that $F = M + \sqrt{M}f \geq 0$. This concludes the proof of Theorem \ref{T1.1}.
\end{proof}

\medskip

\section{Justification of Vanishing Viscosity Limit of the NS-VFP system}% \eqref{C1}}
In this section, we focus on proving the global existence of the compressible Euler-VFP system \eqref{C2} by considering the vanishing viscosity limit of the compressible NS-VFP system \eqref{C1}. Furthermore, we demonstrate that this inviscid limit is valid globally in time and complete the proof of Theorem \ref{T1.3}.
First, we begin the proof of Theorem \ref{T1.2}.

\subsection{Proof of Theorem \ref{T1.2}}
Let $(\varrho_0, u_0, f_0)$ satisfy \eqref{b-1}. By the density argument, one can construct a sequence $\{(\varrho_0^\mu, u_0^\mu, f_0^\mu)\}_{\mu \in (0,1)}$ such that $(\varrho_0^\mu, u_0^\mu) \to (\varrho_0, u_0)$ in $H^3$ and $f_0^\mu \to f_0$ in $H_{x,v}^3$. Consequently, recalling the constant $\eps_0 > 0$ given by \eqref{a-1}, we choose $\mu_0 > 0$ such that  
\begin{align*}  
\|(\varrho_0^\mu - \varrho_0, u_0^\mu - u_0)\|_{H^3}^2 + \|f_0^\mu - f_0\|_{H_{x,v}^3}^2 < \frac{\eps_0}{2}   
\end{align*}  
holds true for all $0 < \mu < \mu_0$. Combining this with \eqref{b-1}, we deduce that  
\begin{align*}  
\mathcal{E}_0^\mu = \|(\varrho_0^\mu, u_0^\mu)\|_{H^3}^2 + \|f_0^\mu\|_{H_{x,v}^3}^2 < \frac{\eps_0}{2} + \|(\varrho_0, u_0)\|_{H^3}^2 + \|f_0\|_{H_{x,v}^3}^2 \leq \frac{\eps_0}{2} + \eps_1 \leq \eps_0,  
\end{align*}  
provided that $\eps_1 \leq \frac{\eps_0}{2}$. According to Theorem \ref{T1.1}, for any $0 < \mu < \mu_0$, there exists a unique global solution $(\varrho^\mu, u^\mu, f^\mu)$ to the compressible NS-VFP system \eqref{C1}. This solution satisfies the regularity estimate \eqref{b3}, which is uniform in both $\mu$ and time $t$. As a direct consequence of \eqref{b3}, there exists a limit $(\varrho, u, f)$ such that as $\mu \to 0$, up to a subsequence,  
\begin{align*}  
&(\varrho^\mu, u^\mu) \rightharpoonup (\varrho, u) \quad \text{weakly$^*$ in } L^\infty(\mathbb{R}^+; H^3),\\  
&\quad \quad \,\,f^\mu \rightharpoonup f  \quad\quad\,\,\,\, \text{weakly$^*$ in } L^\infty(\mathbb{R}^+; H_{x,v}^3).  
\end{align*}

To show the strong convergence of $(\varrho^\mu, u^\mu)$, we need to  rigorously justify the convergence of  all nonlinear terms in \eqref{C1}$_1$-\eqref{C1}$_2$ in the sense of distributions.
To this end,  it is essential to analyze the time derivatives $\partial_t \varrho^\mu$ and $\partial_t u^\mu$.
%  and $\partial_t f^\mu$. 
Specifically, based on the system \eqref{C1} and the estimate \eqref{a-2}, we can derive that 
\begin{align*}  
\|\partial_t\varrho^\mu\|_{L^{\infty}_t(H^2)} \lesssim&\, \|\nabla u^\mu\|_{L^{\infty}_t(H^2)}\|\nabla\varrho^\mu\|_{L^{\infty}_t(H^2)}+\|\nabla u^\mu\|_{L^{\infty}_t(H^2)}\nonumber\\
\lesssim&\,\big(1+\|\varrho^\mu\|_{L^{\infty}_t(H^3)}\big)\|u^\mu\|_{L^{\infty}_t(H^3)}\nonumber\\
\lesssim&\, C,
\end{align*}  
and  
\begin{align*}
\|\partial_t u^\mu\|_{L^{\infty}_t(H^1)}\lesssim&\, \|u^\mu\cdot\nabla u^\mu\|_{L^{\infty}_t(H^1)} +\|b^\mu-u^\mu \|_{L^{\infty}_t(H^1)}+ \mu\Big\| \frac{ \Delta u^\mu }{1+\varrho^\mu}\Big\|_{L^{\infty}_t(H^1)}\\
&+\Big\|\frac{P^\prime(1+\varrho^\mu)}{1+\varrho^\mu}\nabla \varrho^\mu\Big\|_{L^{\infty}_t(H^1)}+\|a^{\mu}u^\mu\|_{L^{\infty}_t(H^1)}\\
\lesssim &\,\|u^\mu\|_{L^{\infty}_t(H^2)}\|\nabla u^\mu\|_{L^{\infty}_t(H^2)}+(1+\|\varrho^\mu\|_{L^{\infty}_t(H^2)}) \|\nabla \varrho^\mu\|_{L^{\infty}_t(H^1)}\\
& +\|a^{\mu}\|_{L^{\infty}_t(H^2)} \|\nabla u^\mu\|_{L^{\infty}_t(H^2)} +\mu (1+\|\varrho^\mu\|_{L^{\infty}_t(H^2)}) \|\nabla u^\mu\|_{L^{\infty}_t(H^2)}\nonumber\\
&+\|b^\mu-u^\mu \|_{L^{\infty}_t(H^1)}\nonumber\\
\lesssim&\, C,
\end{align*}
where $C$ denotes a constant independent of $\mu$.
Here, we have utilized the condition $0 < \mu < 1$.

By applying the Aubin-Lions lemma, we deduce that, up to a subsequence extraction,  
\begin{align*}  
(\varrho^\mu, u^\mu) \rightarrow (\varrho, u) \quad \text{strongly in } C_{\rm loc}(\mathbb{R}_+; H^2_{\rm loc}).  
\end{align*}  
Thus, \eqref{b2} is satisfied. Consequently, one can directly verify that the solution $(\varrho, u, f)$, with $\varrho = \rho - 1$ and $f = \frac{F - M}{\sqrt{M}}$, satisfies the Euler-VFP system \eqref{C2} in the sense of distributions. Therefore, the proof of Theorem \ref{T1.2} is completed.
\hfill $\Box$

\smallskip 

Next, we proceed to prove Theorem \ref{T1.3}.

\subsection{Proof of Theorem \ref{T1.3}}

We are now in a position to establish the global-in-time convergence of the inviscid limit. For $\mu \in (0,1)$, let $(\varrho^\mu, u^\mu, f^\mu)$ denote the global solution to the compressible NS-VFP system \eqref{C1}, as obtained in Theorem \ref{T1.1}. We recall that the equations governing $(\varrho^\mu, u^\mu, f^\mu)$ are given by  
\begin{equation}\left\{
\begin{aligned}\label{G4.1} 
&\partial_{t}\varrho^\mu+(\varrho^\mu+1){\rm div} u^\mu+\nabla\varrho^\mu\cdot u^\mu=0,\\  
&\partial_{t}u^\mu+u^\mu\cdot\nabla u^\mu+{\frac{P^{\prime}(1+\varrho^\mu)}{1+\varrho^\mu}}\nabla\varrho^\mu-\frac{\mu\Delta u^\mu}{1+\varrho^\mu}={b^\mu-u^\mu-a^\mu u^\mu},\\
& \partial_{t}f^\mu+v\cdot\nabla_x f^\mu+u^\mu\cdot\nabla_{v}f^\mu-{\frac{1}{2}}u^\mu\cdot vf^\mu-u^\mu\cdot v\sqrt{M}\\
&\quad= \mathcal{L}f^\mu+\varrho^\mu\Big(\mathcal{L}f^\mu-u^\mu\cdot\nabla_v f^\mu+\frac{1}{2}u^\mu\cdot vf^\mu+u^\mu\cdot v\sqrt{M}\Big).
\end{aligned}\right.
\end{equation}
Similarly, for the global solution $(\varrho, u, f)$ established in Theorem \ref{T1.2}, it   satisfies the compressible Euler-VFP system: 
\begin{equation}\left\{
\begin{aligned}\label{G4.2} 
&\partial_{t}\varrho +(\varrho +1){\rm div} u +\nabla\varrho \cdot u =0,\\  
&\partial_{t}u +u\cdot\nabla u +{\frac{P^{\prime}(1+\varrho )}{1+\varrho }}\nabla\varrho ={b -u -a  u },\\
& \partial_{t}f +v\cdot\nabla_x f +u \cdot\nabla_{v}f -{\frac{1}{2}}u \cdot vf -u \cdot v\sqrt{M}\\
&\quad= \mathcal{L}f +\varrho \Big(\mathcal{L}f -u\cdot\nabla_v f+\frac{1}{2}u \cdot vf +u \cdot v\sqrt{M}\Big).\\  
\end{aligned}\right.
\end{equation}
Define the differences 
\begin{equation*}
\begin{aligned}
&(\widetilde{\varrho}^\mu,\widetilde{u}^\mu, \widetilde{f}^\mu):=(\varrho^\mu-\varrho ,u^\mu-u , f^\mu-f ),
\end{aligned}
\end{equation*}
and
\begin{equation*}
\begin{aligned}
(\widetilde{a}^\mu,\widetilde{b}^\mu):=(a^\mu-a , b^\mu-b ).
\end{aligned}
\end{equation*}
It follows from \eqref{G4.1} and \eqref{G4.2} that  $(\widetilde{\varrho}^\mu,\widetilde{u}^\mu,\widetilde{f}^\mu)$ satisfies  
\begin{equation}\label{G4.3}
\left\{
\begin{aligned}
&\partial_t \widetilde{\varrho}^\mu+{\rm div}\widetilde{u}^\mu=\widetilde{F}_1,\\
&\partial_{t}\widetilde{u}^\mu+ {P^\prime(1 )} \nabla\widetilde{\varrho}^\mu-(\widetilde{b}^\mu-\widetilde{u}^\mu)=-a \widetilde{u}^\mu-\widetilde{a}^\mu u^\mu+\frac{\mu\Delta u^\mu}{1+\varrho^\mu}+\widetilde{F}_2+\widetilde{F}_3,\\
&\partial_{t}\widetilde{f}^\mu+v\cdot\nabla_x \widetilde{f}^\mu+  \varrho u^\mu\cdot\nabla_{v}\widetilde{f}^\mu+u^\mu\cdot\nabla_{v}\widetilde{f}^\mu+ \widetilde{u}^\mu\cdot\nabla_v f\\
& \quad +\frac{1}{2}(1+\varrho) {u}^\mu\cdot v \widetilde{f} ^{\mu}+\frac{1}{2}(1+\varrho) \widetilde{u}^\mu\cdot v  {f}   -\widetilde{u}^\mu\cdot v\sqrt{M}-\mathcal{L}\widetilde{f}^\mu=\widetilde{F}_4+\widetilde{F}_5,
\end{aligned}\right.
\end{equation}
where $\widetilde{F}_i$, $i=1,\dots, 5$, are given by
\begin{equation*}
\begin{aligned}
\widetilde{F}_1:=&\,-\widetilde{u}^\mu\cdot \nabla \varrho-u^\mu\cdot \nabla \widetilde{\varrho}^\mu-\widetilde{\varrho}^\mu {\rm div}u-\varrho{\rm div}\widetilde{u}^\mu ,\\
\widetilde{F}_2:=&\,-\Big(\frac{P^\prime(1+\varrho^\mu)}{1+\varrho^\mu}-\frac{P^\prime(1+\varrho) }{1+\varrho }\Big) \nabla \varrho-\Big(\frac{P^\prime(1+\varrho^\mu)}{1+\varrho^\mu}-P^\prime(1)\Big)\nabla\widetilde{\varrho}^\mu,\\
\widetilde{F}_3:=&\, -\widetilde{u}^\mu\cdot \nabla u -u^\mu\cdot\nabla \widetilde{u}^\mu  ,\\
\widetilde{F}_4:=&\,\widetilde\varrho^\mu \Big(-u^\mu\cdot\nabla_v f^\mu+\frac{1}{2}u^\mu\cdot vf^\mu\Big) -\varrho  \widetilde{u}^\mu\cdot\nabla_v f ,\\
\widetilde{F}_5:=&\,\widetilde\varrho^\mu(\mathcal{L}\{\mathbf{I}-\mathbf{P}\}f^\mu+\mathcal{L}\mathbf{P}f^\mu+u^\mu\cdot v\sqrt{M})+\varrho (\mathcal{L}\mathbf{P}\widetilde{f}^\mu+\widetilde{u}^\mu\cdot v\sqrt{M}).
\end{aligned}
\end{equation*}

We first establish the estimates of $\|(\widetilde{\varrho}^\mu,\widetilde{u}^\mu)\|_{H^1}$ and $\| \widetilde{f}^\mu \|_{H_{x,v}^1}$. To achieve this, we introduce the following functional:
 \begin{align}\label{G4.4}  
\widetilde{\mathcal{X}}^\mu(t) :=&\sup_{\tau\in[0,t]}\big(\|(\widetilde{\varrho}^\mu,\widetilde{u}^\mu )(\tau)\|_{H^1}^2 +\|\widetilde{f}^\mu(\tau)\|_{H_{x,v}^1}^2\big)+\int_{0}^t  \|\nabla(\widetilde{\varrho}^\mu,\widetilde{a}^\mu,\widetilde{b}^\mu )(\tau)\|_{L^2}^2 {\rm d}\tau \nonumber\\  
& +\int_0^t \|(\widetilde{b}^\mu-\widetilde{u}^\mu)(\tau)\|_{H^1}^2   {\rm d}\tau+\int_0^t\sum_{|\alpha|\leq 1}\|\{\mathbf{I}-\mathbf{P}\}\partial^\alpha\widetilde{f}^\mu(\tau)\|_{\nu}^2{\rm d}\tau\nonumber\\  
&+\int_0^t \|\nabla_v \{\mathbf{I}-\mathbf{P}\}\widetilde{f}^\mu(\tau)\|_{\nu}^2{\rm d}\tau.  
\end{align}  
It is obvious that 
\begin{align}\label{G4.5}  
\int_0^t \|\nabla \widetilde{u}^\mu(\tau)\|_{L^2}^2 {\rm d}\tau \lesssim& \int_0^t\big( \|\nabla \widetilde{b}^\mu(\tau)\|_{L^2}^2+\| (  \widetilde{b}^\mu-\widetilde{u}^\mu)(\tau)\|_{H^1}^2 \big){\rm d}\tau \lesssim \widetilde{\mathcal{X}}^\mu(t),  
\end{align}  
and  
\begin{align}\label{G4.6}  
\int_0^t \|\nabla \widetilde{f}^\mu(\tau)\|_{L_{x,v}^2}^2 {\rm d}\tau \lesssim& \int_0^t\big( \|\nabla (\widetilde{a}^\mu,\widetilde{b}^\mu )(\tau)\|_{L^2}^2+ \|\nabla \{\mathbf{I}-\mathbf{P}\}\widetilde{f}^\mu(\tau)\|_{\nu}^2\big){\rm d}\tau \lesssim \widetilde{\mathcal{X}}^\mu(t).  
\end{align}  

\smallskip 
Below we estimate the terms in \eqref{G4.4} through Lemmas \ref{L4.1}--\ref{L4.5}.

\begin{lem}\label{L4.1}
It holds that
\begin{align}\label{G4.7}
&\sup_{\tau\in[0,t]}\big( \| {P^\prime(1)}\|\widetilde\varrho^\mu(\tau) \|_{L^2}^2+\| \widetilde{u}^\mu  (\tau)\|_{L^2}^2+\|\widetilde{f}^\mu(\tau)\|_{L_{x,v}^2}^2\big) \nonumber\\
&\quad + \int_{0}^t  \|(\widetilde{b}^\mu-\widetilde{u}^\mu)(\tau)\|_{L^2}^2{\rm d}\tau  
+\int_{0}^t \|  \{\mathbf{I}-\mathbf{P}\}\widetilde{f}^\mu(\tau)\|_{\nu}^2   {\rm d}\tau \nonumber\\
  &\qquad \leq\widetilde{\mathcal{X}}^\mu(0)+C\mu^2+C \big(\eps_0^\frac{1}{2}  +\eps_1^\frac{1}{2}\big)\widetilde{\mathcal{X}}^\mu(t),
\end{align}
where $C>0$ is a constant independent of    $\mu$ and time $t$.
\end{lem}
\begin{proof}
It follows from    \eqref{G2.5} and \eqref{G4.3} that
\begin{align}\label{G4.8}
\frac{1}{2}\big(&P^\prime(1)\|\widetilde{\varrho}^\mu\|_{L^2}^2+\|\widetilde{u}^\mu\|_{L^2}^2+\|\widetilde{f}^\mu\|_{L_{x
,v}^{2}}^2\big)+\int_0^t\|\widetilde{b}^\mu-\widetilde{u}^\mu\|_{L^2}^{2} {\rm d}\tau\nonumber\\
&+\int_0^t\!\int_{\mathbb{R}^3}(\varrho+1)\langle-\mathcal{L}\{\mathbf{I}-\mathbf{P}\}\widetilde{f}^\mu,\widetilde{f}^\mu\rangle {\rm d}x {\rm d}\tau\nonumber\\ 
\leq\,&\frac{1}{2}\int_0^t\!\int_{\mathbb{R}^3} (1+\varrho)u^\mu\cdot \langle v \widetilde{f}^\mu,\widetilde{f}^\mu\rangle {\rm d}x{\rm d}\tau-\int_0^t\!\int_{\mathbb{R}^3}\langle u^\mu\cdot\nabla_{v}\widetilde{f}^\mu,\widetilde{f}^\mu\rangle{\rm d}x{\rm d}\tau -\int_0^t\!\int_{\mathbb{R}^3} \widetilde{a}^\mu  {u}^{\mu}\cdot
 \widetilde{u}^{\mu}  {\rm d}x{\rm d}\tau\nonumber\\
 &+\frac{1}{2}\int_0^t\!\int_{\mathbb{R}^3} (1+\varrho)\widetilde{u}^\mu\cdot \langle v  {f} ,\widetilde{f}^\mu\rangle {\rm d}x{\rm d}\tau-\int_0^t\int_{\mathbb{R}^3}\langle \widetilde{u}^\mu\cdot\nabla_{v} {f} ,\widetilde{f}^\mu\rangle{\rm d}x{\rm d}\tau -\int_0^t\!\int_{\mathbb{R}^3} {a}^\mu 
 |\widetilde{u}^{\mu}|^2 {\rm d}x{\rm d}\tau\nonumber\\
& +\mu\int_0^t\!\int_{\mathbb{R}^3} \frac{\Delta u^\mu\cdot\widetilde{u}^\mu }{1+\varrho^\mu} {\rm d}x{\rm d}\tau+\int_0^t\!\int_{\mathbb{R}^3} \widetilde{F}_1   \widetilde{\varrho}^{\mu}{\rm d}x{\rm d}\tau+\int_0^t\!\int_{\mathbb{R}^3} \widetilde{F}_2 \cdot \widetilde{u}^{\mu}{\rm d}x{\rm d}\tau\nonumber\\
&+\int_0^t\!\int_{\mathbb{R}^3} \widetilde{F}_3 \cdot \widetilde{u}^\mu{\rm d}x{\rm d}\tau
+\int_0^t\!\int_{\mathbb{R}^3\times\mathbb{R}^3} \widetilde{F}_4 \widetilde{f}^\mu{\rm d}x{\rm d}v{\rm d}\tau+\int_0^t\!\int_{\mathbb{R}^3\times\mathbb{R}^3} \widetilde{F}_5 \widetilde{f}^\mu{\rm d}x{\rm d}v{\rm d}\tau\nonumber\\
&+\frac{1}{2}\big(P^\prime(1)\|\widetilde{\varrho}^\mu(0)\|_{L^2}^2+\|\widetilde{u}^\mu(0)\|_{L^2}^2+\|\widetilde{f}^\mu(0)\|_{L_{x
,v}^{2}}^2\big)\nonumber\\
\equiv:\,&\sum_{j=1}^{12}\widetilde{I}_j+\frac{1}{2}\big(P^\prime(1)\|\widetilde{\varrho}^\mu(0)\|_{L^2}^2+\|\widetilde{u}^\mu(0)\|_{L^2}^2+\|\widetilde{f}^\mu(0)\|_{L_{x
,v}^{2}}^2\big).    
\end{align}

Similarly to \eqref{G3.6}, by applying the decomposition \eqref{G2.4}, utilizing Lemmas \ref{L2.1}–\ref{L2.2}, and using the property that $v^k\sqrt{M}\lesssim1$ for all $k\geq 0$, we obtain 
\begin{align}\label{G4.9}
\widetilde{I}_1+\widetilde{I}_2+\widetilde{I}_3
\lesssim&\, \big(1+\|\varrho\|_{L_t^\infty(H^3)}\big)    \int_0^{t}\big( \|u^\mu\|_{L^{\infty}} \|\{\mathbf{I-P}\}\widetilde{f}^\mu\|_{\nu}^2
+\|\widetilde{b}^\mu-\widetilde{u}^\mu\|_{L^2}\|u^\mu\|_{L^3}\|\widetilde{a}^\mu\|_{L^6}  \big){\rm d}\tau\nonumber \\
& +\big(1+\|\varrho\|_{L_t^\infty(H^3)}\big) \int_0^t\|u^\mu\|_{L^3} \|(\widetilde{a}^\mu,\widetilde{b}^\mu)\|_{L^6}\|\{\mathbf{I-P}\}\widetilde{f}^\mu\|_{\nu}^2 {\rm d}\tau\nonumber\\
&+\int_0^t\|\varrho\|_{H^3}\|u^{\mu}\|_{H^2}\|\nabla (\widetilde{a}^{\mu},\widetilde{b}^{\mu})\|_{L^2}^2{\rm d}\tau\nonumber\\
\lesssim&\,  \big(1+\eps_1^{\frac{1}{2}} \big) \|u^\mu \|_{L_t^\infty(H^2)}\int_0^t \big( \|\widetilde{b}^\mu-\widetilde{u}^\mu\|_{L^2}^2+\|\nabla(\widetilde{a}^\mu,\widetilde{b}^\mu\|_{L^2}^2
+\|\{\mathbf{I-P}\}\widetilde{f}^\mu\|_{\nu}^2\big){\rm d}\tau \nonumber\\
\lesssim&\,  \big(1+\eps_0^\frac{1}{2}+\eps_1^{\frac{1}{2}} \big) \widetilde{\mathcal{X}}^\mu(t).
\end{align}
Here, we have used the facts that 
\begin{align*}
\frac{1}{2}u^\mu\cdot v\mathbf{P}\widetilde{f}^\mu-u^\mu\cdot\nabla_v \mathbf{P}\widetilde{f}^\mu=(\widetilde{a}^\mu u^\mu\cdot v-u^\mu\cdot \widetilde{b}^\mu+u^\mu\cdot v \widetilde{b}^\mu\cdot v)\sqrt{M},  
\end{align*}
and
\begin{align*}
\int_{\mathbb{R}^3}( u^\mu\cdot v \widetilde{b}^\mu\cdot v   )\mathbf{P}\widetilde{f}^\mu \sqrt{M}{\rm d}v=\widetilde{a}^\mu u^\mu \cdot \widetilde{b}^\mu. 
\end{align*}
Taking a similar arguments to that used in   \eqref{G4.9}, we have
\begin{align}\label{G4.10}
\widetilde{I}_4+\widetilde{I}_5+\widetilde{I}_6\lesssim&\,    \big(1+\eps_0^\frac{1}{2}+\eps_1^{\frac{1}{2}} \big) \widetilde{\mathcal{X}}^\mu(t). 
\end{align}
For the term $\widetilde{I}_7$, which affects the convergence rate of the inviscid limit, we apply integration by parts,
H\"{o}lder's and Young's inequalities to obtain  
\begin{align}\label{G4.11}
\widetilde{I}_7=&\, -\mu\int_0^t \! \int_{\mathbb{R}^3}\frac{\nabla u^\mu}{1+\varrho^\mu} {\rm div}\widetilde{u}^\mu {\rm d}x{\rm d}\tau-\mu\int_0^t \! \int_{\mathbb{R}^3} {\nabla u^\mu}\cdot\nabla\Big( \frac{1}{1+\varrho^\mu} \Big)  \cdot \widetilde{u}^\mu {\rm d}x{\rm d}\tau\nonumber\\
\lesssim&\, \mu\int_0^t \|\nabla u^\mu\|_{L^2}\|\nabla\widetilde{u}^\mu\|_{L^2}{\rm d}\tau+\mu\int_0^t \|\nabla u^\mu\|_{L^2}\Big\|\nabla\Big( \frac{1}{1+\varrho^\mu}  \Big) \Big\|_{L^3}\|\widetilde{u}^\mu\|_{L^6}{\rm d}\tau\nonumber\\
\lesssim&\, \mu \big(1+\|\varrho^\mu\|_{L_t^\infty(H^3)}\big)\|\nabla\widetilde{u}^\mu\|_{L_t^2(L^2)}\|\nabla u^\mu\|_{L_t^2(L^2)}\nonumber\\
\lesssim&\, \big( \eps_0^\frac{1}{2}+\eps_1^\frac{1}{2}    \big)\widetilde{\mathcal{X}}^\mu(t)+\mu^2.
\end{align}
For the terms $\widetilde{I}_8$, $\widetilde{I}_9$, and $\widetilde{I}_{10}$, using H\"{o}lder's inequality and Lemmas \ref{L2.1}--\ref{L2.2} leads to  
\begin{align}
\widetilde{I}_8\lesssim&\, \int_0^t \|\widetilde\varrho^\mu\|_{L^3} \big(\|(\widetilde{\varrho}^\mu,\widetilde{ u}^\mu)\|_{L^6}\|\nabla(\varrho,u)\|_{L^2}+\|u^\mu\|_{L^6}\|\nabla\widetilde{\varrho}^\mu\|_{L^2}+\|\varrho\|_{L^6}\|\nabla\widetilde{u}^\mu\|_{L^2}             \big)  {\rm d}\tau \nonumber\\
\lesssim&\, \|(\widetilde{\varrho}^\mu,\widetilde{u}^\mu)\|_{L_t^\infty(H^1)} \|\nabla(\widetilde{\varrho}^\mu,\widetilde{u}^\mu)\|_{L_t^2(L^2)}\big(\| \nabla( {\varrho} , {u} )\|_{L_t^2(L^2)}+\|\nabla {u}^\mu\|_{L_t^2(L^2)}\big) \nonumber\\
\lesssim&\, \big( \eps_0^\frac{1}{2}+\eps_1^\frac{1}{2}    \big)\widetilde{\mathcal{X}}^\mu(t),\label{G4.12-1}\\
\widetilde{I}_9\lesssim&\,\int_0^t \|\widetilde{u}^\mu\|_{L^3}\big(\|\nabla\varrho\|_{L^2}\| \widetilde\varrho^\mu\|_{L^6} + \|\nabla\widetilde{\varrho}^\mu\|_{L^2}\|\varrho^\mu\|_{L^6} \big) {\rm d}\tau \nonumber\\
\lesssim&\, \|\widetilde{u}^\mu\|_{L_t^\infty(H^1)}\|\nabla\widetilde{\varrho}^\mu\|_{L_t^2(L^2)}
\big(\|\nabla\varrho\|_{L_t^2(L^2)}+\|\nabla\varrho^\mu\|_{L_t^2(L^2)}\big)\nonumber\\
\lesssim&\, \big( \eps_0^\frac{1}{2}+\eps_1^\frac{1}{2}    \big)\widetilde{\mathcal{X}}^\mu(t),\label{G4.12-2}\\
\widetilde{I}_{10}\lesssim&\, \int_0^t \|\widetilde{u}^\mu\|_{L^3}\big(\|\nabla u\|_{L^2}\| \widetilde{u}^\mu\|_{L^6} + \|\nabla\widetilde{u}^\mu\|_{L^2}\|u^\mu\|_{L^6} \big) {\rm d}\tau \nonumber\\
\lesssim&\, \|\widetilde{u}^\mu\|_{L_t^\infty(H^1)}\|\nabla\widetilde{u}^\mu\|_{L_t^2(L^2)}\big(\|\nabla u\|_{L_t^2(L^2)}+\|\nabla u^\mu\|_{L_t^2(L^2)}\big)\nonumber\\
\lesssim&\, \big( \eps_0^\frac{1}{2}+\eps_1^\frac{1}{2}    \big)\widetilde{\mathcal{X}}^\mu(t).\label{G4.12}
\end{align}
For the term $\widetilde{I}_{11}$,  applying H\"{o}lder's inequality, the macro-micro decomposition \eqref{G2.4},  and Lemmas \ref{L2.1}–\ref{L2.2}, yields 
\begin{align}\label{G4.13}
\widetilde{I}_{11} \lesssim&\, \int_0^t \|\widetilde{f}^\mu\|_{L_{x,v}^2}\|\widetilde{\varrho}^\mu\|_{L^6}\|u^\mu\|_{L^6}\big(\|\nabla\{\mathbf{I}-\mathbf{P\}}f^\mu\|_{\nu}+\|\nabla(a^\mu,b^\mu)\|_{L^2} \big){\rm d}\tau\nonumber\\
&+\int_0^t  \|\widetilde{f}^\mu\|_{L_{x,v}^2}\|\varrho\|_{L^6}\|\widetilde{u}^\mu\|_{L^6}\big(\|\nabla\{\mathbf{I}-\mathbf{P\}}f^\mu\|_{\nu}+\|\nabla(a^\mu,b^\mu)\|_{L^2} \big){\rm d}\tau\nonumber\\
\lesssim&\,\big(\eps_0^\frac{1}{2}+\eps_1^\frac{1}{2}  \big) \|\widetilde{f}^\mu\|_{L_t^\infty(L_{x,v}^2)}
\|\nabla(\widetilde\varrho^\mu,\widetilde{u}^\mu)\|_{L_t^2(L^2)}\big(\|\nabla\{\mathbf{I}-\mathbf{P\}}f^\mu\|_{L_t^2(\nu)}^2+\|\nabla(a^\mu,b^\mu)\|_{L^2}^2\big)\nonumber\\
\lesssim &\, \big( \eps_0^\frac{1}{2}+\eps_1^\frac{1}{2}    \big)\widetilde{\mathcal{X}}^\mu(t).
\end{align}
Notice that $\mathcal{L}\mathbf{P}f^\mu=-b^\mu v\sqrt{M}$ and 
$\mathcal{L}\mathbf{P}\widetilde{f}^\mu=-\widetilde{b}^\mu v\sqrt{M}$.
Similar to \eqref{G4.13},  for the remaining   $\widetilde{I}_{12}$, it can be estimated as follows:
\begin{align}\label{G4.14}
 \widetilde{I}_{12}\lesssim&\,  \int_0^t \|\widetilde{\varrho}^\mu\|_{L^3}  \big(\|\nabla\widetilde{b}^\mu\|_{L^2} \|b^\mu-u^\mu\|_{L^2}+\|\{\mathbf{I}-\mathbf{P}\}f^\mu\|_{\nu}\|\nabla\{\mathbf{I}-\mathbf{P}\}\widetilde{f}^\mu\|_{\nu}\big){\rm d}\tau \nonumber\\
 &+\int_0^t \|\widetilde{\varrho}^\mu\|_{L^3}\|\{\mathbf{I}-\mathbf{P}\}f^\mu\|_{\nu}\|\nabla(\widetilde{a}^\mu,\widetilde{b}^\mu)\|_{L^2}{\rm d}\tau+\int_0^t \|\varrho\|_{L^3}\|\nabla \widetilde{b}^\mu\|_{L^2}\|\widetilde{b}^\mu-\widetilde{u}^\mu\|_{L^2}{\rm d}\tau\nonumber\\
 \lesssim&\, \|\widetilde{\varrho}^\mu\|_{L_t^\infty(H^1)} \|\nabla(\widetilde{a}^\mu,\widetilde{b}^\mu)\|_{L_t^2(L^2)} \big(\|\{\mathbf{I}-\mathbf{P}\}f^\mu\|_{L_t^2(\nu)}+\|\nabla\{\mathbf{I}-\mathbf{P}\}f^\mu\|_{L_t^2(\nu)}   \big)\nonumber\\
 &+\|\nabla(\widetilde{a}^\mu,\widetilde{b}^\mu)\|_{L_t^2(L^2)}\big(\|\widetilde{\varrho}^\mu\|_{L_t^\infty(H^1)} \|b^\mu-u^\mu\|_{L_t^2(L^2)} +\| {\varrho} \|_{L_t^\infty(H^1)} \|\widetilde{ b}^\mu-\widetilde{u}^\mu\|_{L_t^2(L^2)}\big)\nonumber\\
\lesssim &\, \big( \eps_0^\frac{1}{2}+\eps_1^\frac{1}{2}    \big)\widetilde{\mathcal{X}}^\mu(t).
\end{align}

Substituting the estimates \eqref{G4.9}–\eqref{G4.14} into \eqref{G4.8}, and applying \eqref{G2.5} and \eqref{G3.2}, we ultimately derive the desired \eqref{G4.7}.
\end{proof}

\begin{lem}\label{L4.2}
It holds that
\begin{align}\label{G4.15}
&\sup_{\tau\in[0,t]}\bigg( 
 \bigg\|\frac{\sqrt{P^\prime(1+\varrho(\tau))}}{1+\varrho(\tau)}\nabla\widetilde{\varrho}^\mu(\tau)\bigg\|_{L^2}^2+\|\nabla\widetilde{u}^\mu(\tau)\|_{L^2}^2+\|\nabla \widetilde{f}^\mu(\tau)\|_{L_{x
,v}^{2}}^2\bigg)\nonumber\\
& \quad + \int_{0}^t  \|\nabla (\widetilde{b}^\mu-\widetilde{u}^\mu)(\tau)\|_{L^2}^2{\rm d}\tau  
+\int_{0}^t \|  \nabla\{\mathbf{I}-\mathbf{P}\}\widetilde{f}^\mu(\tau)\|_{\nu}^2   {\rm d}\tau  \nonumber\\ &\qquad\leq\widetilde{\mathcal{X}}^\mu(0)+C\mu^2+C \big(\eps_0^\frac{1}{2}  +\eps_1^\frac{1}{2}\big)\widetilde{\mathcal{X}}^\mu(t),
\end{align}
where $C>0$ is a constant independent of    $\mu$ and time $t$.
\end{lem}
\begin{proof}
By applying the operator $\nabla$ to \eqref{G4.3} and employing the energy method outlined in Lemma \ref{L3.2}, one has
\begin{align}\label{G4.16}
\frac{1}{2}\bigg( & \bigg\|\frac{\sqrt{P^\prime(1+\varrho)}}{1+\varrho}\nabla\widetilde{\varrho}^\mu\bigg\|_{L^2}^2+\|\nabla\widetilde{u}^\mu\|_{L^2}^2+\|\nabla \widetilde{f}^\mu\|_{L_{x
,v}^{2}}^2\bigg)+\int_0^t\|\nabla(\widetilde{b}^\mu-\widetilde{u}^\mu)\|_{L^2}^{2} {\rm d}\tau\nonumber\\
&+\int_0^t\!\int_{\mathbb{R}^3}(\varrho+1)\langle-\mathcal{L}\nabla\{\mathbf{I}-\mathbf{P}\}\widetilde{f}^\mu,\nabla\widetilde{f}^\mu\rangle {\rm d}x {\rm d}\tau\nonumber\\ 
\leq\,&\frac{1}{2}\int_0^t\!\int_{\mathbb{R}^3} \big \langle \nabla\big((1+\varrho)u^\mu\cdot v \widetilde{f}^\mu\big),\nabla\widetilde{f}^\mu\big\rangle {\rm d}x{\rm d}\tau-\int_0^t\!\int_{\mathbb{R}^3}\big\langle \nabla \big(u^\mu\cdot\nabla_{v}\widetilde{f}^\mu\big),\nabla\widetilde{f}^\mu\big\rangle{\rm d}x{\rm d}\tau \nonumber\\
&-\int_0^t\!\int_{\mathbb{R}^3} \nabla(\widetilde{a}^\mu  {u}^{\mu})\cdot
 \nabla\widetilde{u}^{\mu}  {\rm d}x{\rm d}\tau+\frac{1}{2}\int_0^t\!\int_{\mathbb{R}^3}  \big\langle \nabla\big((1+\varrho)\widetilde{u}^\mu\cdot v  {f} \big),\nabla\widetilde{f}^\mu\big\rangle {\rm d}x{\rm d}\tau \nonumber\\
 &-\int_0^t\int_{\mathbb{R}^3}\big\langle \nabla\big(\widetilde{u}^\mu\cdot\nabla_{v} {f} \big),\widetilde{f}^\mu\big\rangle{\rm d}x{\rm d}\tau -\int_0^t\!\int_{\mathbb{R}^3}\nabla( {a}^\mu \widetilde{u}^\mu)\cdot\nabla
  \widetilde{u}^{\mu} {\rm d}x{\rm d}\tau\nonumber\\
& +\int_0^t\!\int_{\mathbb{R}^3}\nabla\varrho\cdot\langle\mathcal{L}\{\mathbf{I}-\mathbf{P}\}\widetilde{f}^\mu,\nabla\widetilde{f}^\mu 
 \rangle {\rm d}x{\rm d}\tau+\mu\int_0^t\!\int_{\mathbb{R}^3} \nabla\Big(\frac{\Delta u^\mu }{1+\varrho^\mu}\Big) \cdot\nabla\widetilde{u}^\mu {\rm d}x{\rm d}\tau\nonumber\\
& +\int_0^t\!\int_{\mathbb{R}^3} \frac{P^\prime(1+\varrho)}{(1+\varrho)^2}\nabla(\widetilde{F}_1+\varrho{\rm div}\widetilde{u}^\mu  ) \cdot\nabla\widetilde{\varrho}^{\mu}{\rm d}x{\rm d}\tau+\frac{1}{2}\int_0^t\!\int_{\mathbb{R}^3}\partial_t \bigg[  \frac{P^\prime (1+\varrho)}{(1+\varrho)^2} \bigg] |\nabla \widetilde\varrho^\mu|^2{\rm d}\tau \nonumber\\
&+\int_0^t\!\int_{\mathbb{R}^3} \nabla \bigg(\bigg( \frac{P^\prime(1+\varrho^\mu)}{1+\varrho^\mu}
-\frac{P^\prime(1+\varrho)}{1+\varrho }  \bigg)\nabla\varrho^\mu\bigg) \cdot \nabla\widetilde{u}^{\mu}{\rm d}x{\rm d}\tau+\int_0^t\!\int_{\mathbb{R}^3} \nabla\widetilde{F}_3 \cdot \nabla\widetilde{u}^\mu{\rm d}x{\rm d}\tau\nonumber\\
&+\int_0^t\!\int_{\mathbb{R}^3\times\mathbb{R}^3} \nabla \widetilde{F}_4 \cdot\nabla\widetilde{f}^\mu{\rm d}x{\rm d}v{\rm d}\tau+\int_0^t\!\int_{\mathbb{R}^3\times\mathbb{R}^3} \nabla\widetilde{F}_5 \cdot\nabla\widetilde{f}^\mu{\rm d}x{\rm d}v{\rm d}\tau\nonumber\\
&+\frac{1}{2}\bigg(  \bigg\|\frac{\sqrt{P^\prime(1+\varrho)}}{1+\varrho}\nabla\widetilde{\varrho}^\mu(0)\bigg\|_{L^2}^2+\|\nabla\widetilde{u}^\mu(0)\|_{L^2}^2+\|\nabla \widetilde{f}^\mu(0)\|_{L_{x
,v}^{2}}^2\bigg)\nonumber\\
\equiv:\,&\sum_{i=1}^{14}\widetilde{J}_i+\frac{1}{2}\bigg(  \bigg\|\frac{\sqrt{P^\prime(1+\varrho)}}{1+\varrho}\nabla\widetilde{\varrho}^\mu(0)\bigg\|_{L^2}^2+\|\nabla\widetilde{u}^\mu(0)\|_{L^2}^2+\|\nabla \widetilde{f}^\mu(0)\|_{L_{x
,v}^{2}}^2\bigg).    
\end{align}

Based on the macro-micro decomposition \eqref{G2.4} and Lemmas \ref{L2.1}–\ref{L2.2}, 
for the terms $\widetilde{J}_1$, $\widetilde{J}_2$, and $\widetilde{J}_3$, one gets 
\begin{align}
\widetilde{J}_1\lesssim&\, \int_0^t \|\nabla\widetilde{f}^\mu\|_{L_{x,v}^2}\big(\|\nabla(\widetilde{a}^\mu,\widetilde{b}^\mu)\|_{L^2}+\|\nabla\{\mathbf{I}-\mathbf{P}\}\widetilde{f}^\mu\|_{\nu}  \big)\| u^\mu\|_{L^6} \|\nabla\varrho\|_{L^6} {\rm d}\tau \nonumber\\
&+\int_0^t \big(1+\|\varrho\|_{ L^\infty}\big)\|\nabla u^\mu\|_{L^3}\big(\|\nabla(\widetilde{a}^\mu,\widetilde{b}^\mu)\|_{L^2}+\|\nabla\{\mathbf{I}-\mathbf{P}\}\widetilde{f}^\mu\|_{\nu}  \big)\|\nabla\widetilde{f}^\mu\|_{L_{x,v}^2}{\rm d}\tau  \nonumber\\
&+\int_0^t \big(1+\|\varrho\|_{ L^\infty}\big)\|  u^\mu\|_{L^\infty}\big(\|\nabla(\widetilde{a}^\mu,\widetilde{b}^\mu)\|_{L^2}+\|\nabla\{\mathbf{I}-\mathbf{P}\}\widetilde{f}^\mu\|_{\nu}  \big)\|\nabla\widetilde{f}^\mu\|_{L_{x,v}^2}{\rm d}\tau  \nonumber\\
\lesssim&\, \big(1+\eps_1^\frac{1}{2}\big) \|\widetilde{f}^\mu\|_{L_t^\infty(L_v^2(H^1))}\big(\|\nabla(\widetilde{a}^\mu,\widetilde{b}^\mu)\|_{L_t^2(L^2)}+\|\nabla\{\mathbf{I}-\mathbf{P}\}\widetilde{f}^\mu\|_{L_t^2(\nu)}  \big) \|\nabla u^\mu\|_{L_t^2(H^1)}\nonumber\\
\lesssim&\,  \big(\eps_0^\frac{1}{2}  +\eps_1^\frac{1}{2}\big)\widetilde{\mathcal{X}}^\mu(t),\label{G4.17-1}\\
\widetilde{J}_2\lesssim&\,\int_0^t \|\nabla\widetilde{f}^\mu\|_{L_{x,v}^2} \|\nabla u^\mu\|_{L^3}\big(\|\nabla\{\mathbf{I}-\mathbf{P}\}\widetilde{f}^\mu\|_{\nu}+\|\nabla(\widetilde{a}^\mu,\widetilde{b}^\mu)\|_{L^2}\big) {\rm d}\tau \nonumber\\
&+\int_0^t \|\nabla\widetilde{f}^\mu\|_{L_{x,v}^2} \| u^\mu\|_{L^\infty}\big(\|\nabla\{\mathbf{I}-\mathbf{P}\}\widetilde{f}^\mu\|_{\nu}+\|\nabla(\widetilde{a}^\mu,\widetilde{b}^\mu)\|_{L^2}\big) {\rm d}\tau \nonumber\\
\lesssim&\, \|\widetilde{f}^\mu\|_{L_t^\infty(L_v^2(H^1))} \|\nabla u^\mu\|_{L_t^2(H^1)} \big(\|\nabla\{\mathbf{I}-\mathbf{P}\}\widetilde{f}^\mu\|_{L_t^2(\nu)}+\|\nabla(\widetilde{a}^\mu,\widetilde{b}^\mu)\|_{L_t^2(L^2)}\big)\nonumber\\
\lesssim&\,  \big(\eps_0^\frac{1}{2}  +\eps_1^\frac{1}{2}\big)\widetilde{\mathcal{X}}^\mu(t),\label{G4.17-2}\\
\widetilde{J}_3\lesssim&\, \int_0^t(\|\nabla\widetilde{a}^\mu\|_{L^2}\|u^\mu\|_{L^\infty}+\|\widetilde{a}^\mu\|_{L^6}\|\nabla u^\mu\|_{L^3} )  \|\nabla\widetilde{u}^\mu\|_{L^2}{\rm d}\tau\nonumber\\
\lesssim&\, \|\nabla\widetilde{a}^\mu\|_{L_t^2(L^2)}\|\nabla\widetilde{u}^\mu\|_{L_t^2(L^2)}\|u^\mu\|_{L_t^\infty(H^2)}\nonumber\\
\lesssim&\,  \big(\eps_0^\frac{1}{2}  +\eps_1^\frac{1}{2}\big)\widetilde{\mathcal{X}}^\mu(t).\label{G4.17}
\end{align}
Similar to the estimates in \eqref{G4.17-1}--\eqref{G4.17}, we can deduce  that
\begin{align}\label{G4.18}
\widetilde{J}_4+\widetilde{J}_5+\widetilde{J}_6\lesssim    \big(\eps_0^\frac{1}{2}  +\eps_1^\frac{1}{2}\big)\widetilde{\mathcal{X}}^\mu(t). 
\end{align}
For the term $\widetilde{J}_7$, by applying the macro-micro decomposition \eqref{G2.4} and utilizing the property that $v^k \sqrt{M} \lesssim 1$  for any $k \geq 0$, we obtain  
\begin{align}\label{G4.19}
\widetilde{J}_7\lesssim&\, \int_0^t\|\nabla\varrho\|_{L^3} \big(\|\nabla\{\mathbf{I}-\mathbf{P}\}\widetilde{f}
^\mu\|_{\nu}+\|\nabla(\widetilde{a}^\mu,\widetilde{b}^\mu)\|_{L^2} 
  \big) \|\nabla \widetilde{f}^\mu\|_{L_{x,v}^2} {\rm d}\tau   \nonumber\\
  \lesssim&\, \|\nabla\varrho\|_{L_t^2(H^1)}\big(\|\nabla\{\mathbf{I}-\mathbf{P}\}\widetilde{f}
^\mu\|_{L_t^2(\nu)}+\|\nabla(\widetilde{a}^\mu,\widetilde{b}^\mu)\|_{L_t^2(L^2)} 
  \big) \|  \widetilde{f}^\mu\|_{L_t^\infty(L_v^2(H^1))} \nonumber\\
  \lesssim &\,   \big(\eps_0^\frac{1}{2}  +\eps_1^\frac{1}{2}\big)\widetilde{\mathcal{X}}^\mu(t). 
\end{align}
For the term $\widetilde{J}_8$, applying H\"{o}lder's inequality  in conjunction with Lemma \ref{L2.1}, we arrive at
\begin{align}\label{G4.20}
\widetilde{J}_8\lesssim&\, \mu\int_0^t\Big(\Big\|\nabla\Big( \frac{1}{1+\varrho^\mu} \Big)\Big\|_{L^\infty} \|\Delta u^\mu\|_{L^2} +\|\nabla\Delta u^\mu\|_{L^2}\Big)\|\nabla\widetilde{u}^\mu\|_{L^2}{\rm d}\tau\nonumber\\
\lesssim&\, \mu \big(1+\|\varrho\|_{L_t^\infty(H^3)}\big) \|\nabla u^\mu\|_{L_t^2(H^2)}\|\nabla\widetilde{u}^\mu\|_{L_t^2(L^2)}\nonumber\\
\lesssim&\, \mu^2+ \big(\eps_0^\frac{1}{2}+\eps_1^\frac{1}{2}\big)\widetilde{\mathcal{X}}^\mu(t).
\end{align}
For the term $\widetilde{J}_9$, by direct calculation, it follows that  
\begin{align}\label{G4.21}
\widetilde{J}_9\lesssim&\, \big(1+\eps_1^\frac{1}{2}\big)\int_0^t   \|\nabla\widetilde\varrho^\mu\|_{L^2} \big(\|\nabla\widetilde{u}^\mu\|_{L^2} \|\nabla\varrho\|_{L^\infty} +\|\widetilde{u}^\mu\|_{L^6}\|\nabla^2\varrho\|_{L^3}+\|\nabla u^\mu\|_{L^\infty}\|\nabla\widetilde{\varrho}^\mu\|_{L^2}\big){\rm d}\tau\nonumber\\  
&+\big(1+\eps_1^\frac{1}{2}\big)\int_0^t \big(\|\nabla\widetilde{\varrho}^\mu\|_{L^2}^2\|{\rm div}u^\mu\|_{L^\infty}+\|\nabla\widetilde{\varrho}^\mu\|_{L^2}^2\|{\rm div}u \|_{L^\infty} \big){\rm d}\tau\nonumber\\
&+\big(1+\eps_1^\frac{1}{2}\big)\int_0^t\|\nabla\widetilde{\varrho}^\mu\|_{L^2}\|\widetilde\varrho^\mu\|_{L^6}\|\nabla {\rm div}u\|_{L^3} {\rm d}\tau \nonumber\\
\lesssim&\, \|\nabla\widetilde{\varrho}^\mu\|_{L_t^2(L^2)}\|\nabla(\widetilde{\varrho}^\mu,\widetilde{u}^\mu)\|_{L_t^2(L^2)} \big( \|(\varrho,u)\|_{L_t^\infty(H^3)}+ \|u^\mu\|_{L_t^\infty(H^3)} \big)\nonumber\\
\lesssim&\,  \big(\eps_0^\frac{1}{2}+\eps_1^\frac{1}{2}\big)\widetilde{\mathcal{X}}^\mu(t).
\end{align}
For the term $\widetilde{J}_{10}$, by applying the estimate \eqref{NJK-aaaa} established in Lemma \ref{L3.2}, we have 
\begin{align}\label{G4.22}
\widetilde{J}_{10}\lesssim \int_0^t \|u\|_{H^3} \|\nabla\widetilde{\varrho}^\mu\|_{L^2}^2{\rm d}\tau \lesssim \|u\|_{L_t^\infty(H^3)}\|\nabla\widetilde{\varrho}^\mu\|_{L_t^2(L^2)}\lesssim\big(\eps_0^\frac{1}{2}+\eps_1^\frac{1}{2}\big)\widetilde{\mathcal{X}}^\mu(t).
\end{align}
For the terms $\widetilde{J}_{11}$ and $\widetilde{J}_{12}$, from integration by parts, 
 H\"{o}lder's inequality and Lemmas \ref{L2.1}--\ref{L2.2}, it holds  
\begin{align}\label{G4.23}
 \widetilde{J}_{11}\lesssim&\,\int_0^t \big(\|\nabla\widetilde{\varrho}^\mu\|_{L^2}\|\nabla\varrho^\mu\|_{L^\infty}+\|\widetilde{\varrho}^\mu\|_{L^6}\|\nabla^2\varrho^\mu\|_{L^3}\big)\|\nabla\widetilde{u}^\mu\|_{L^2}   {\rm d}\tau\nonumber\\
 \lesssim&\,\|\varrho^\mu\|_{L_t^\infty(H^3)}\|\nabla\widetilde\varrho^\mu\|_{L_t^2(L^2)}\|\nabla\widetilde{u}^\mu\|_{L_t^2(L^2)}\nonumber\\
 \lesssim&\, \big(\eps_0^\frac{1}{2}+\eps_1^\frac{1}{2}\big)\widetilde{\mathcal{X}}^\mu(t),\nonumber\\
 \widetilde{J}_{12}\lesssim&\, \int_0^t \big(\|\nabla\widetilde{u}^\mu\|_{L^2}^2(\|\nabla u\|_{L^\infty}+\|{\rm div}u^\mu\|_{L^\infty}+\|\nabla u^\mu\|_{L^\infty}) +\|\widetilde{u}^\mu\|_{L^6}\|\nabla^2 u\|_{L^3}\|\nabla\widetilde{u}^\mu\|_{L^2}\big){\rm d}\tau \nonumber\\
 \lesssim&\, \|\nabla\widetilde{u}^\mu\|_{L_t^2(L^2)} ^2 \big(\|u^\mu\|_{L_t^\infty(H^3)}+\|u \|_{L_t^\infty(H^3)}   \big)\nonumber\\
 \lesssim&\, \big(\eps_0^\frac{1}{2}+\eps_1^\frac{1}{2}\big)\widetilde{\mathcal{X}}^\mu(t).
\end{align}
Finally, for the remaining terms $\widetilde{J}_{13}$ and $\widetilde{J}_{14}$, by using the decomposition \eqref{G2.4} and Lemmas \ref{L2.1}--\ref{L2.3}, we get
\begin{align}\label{G4.24}
 \widetilde{J}_{13}\lesssim&\,    \int_0^t \|\nabla\widetilde{f}^\mu\|_{L_{x,v}^2}\|\nabla\widetilde{\varrho}^\mu\|_{L^2} \|u^\mu\|_{L^\infty}\bigg(\sum_{1\leq|\alpha|\leq 2}\|\partial^\alpha\{\mathbf{I}-\mathbf{P}\}f^\mu\|_{\nu}+\|\nabla(a^\mu,b^\mu)\|_{H^1}\bigg){\rm d}\tau\nonumber\\
 &+\int_0^t \|\nabla\widetilde{f}^\mu\|_{L_{x,v}^2}\|\widetilde\varrho^\mu\|_{L^6}\|\nabla u^\mu\|_{L^6} \big( \|\nabla\{\mathbf{I}-\mathbf{P}\}f^\mu\|_{\nu}+\|\nabla(a^\mu,b^\mu)\|_{L^2}    \big){\rm d}\tau\nonumber\\
&+\int_0^t \|\nabla\widetilde{f}^\mu\|_{L_{x,v}^2}\|\widetilde\varrho^\mu\|_{L^6}\|  u^\mu\|_{L^6} \big( \|\nabla^2\{\mathbf{I}-\mathbf{P}\}f^\mu\|_{\nu}+\|\nabla^2(a^\mu,b^\mu)\|_{L^2}    \big){\rm d}\tau\nonumber\\
&+\int_0^t \|\nabla\widetilde{f}^\mu\|_{L_{x,v}^2}\|\nabla\widetilde{u}^\mu\|_{L^2}  \|\varrho\|_{L^\infty} \bigg(\sum_{1\leq|\alpha|\leq 2}\|\partial^\alpha\{\mathbf{I}-\mathbf{P}\}f^\mu\|_{\nu}+\|\nabla(a^\mu,b^\mu)\|_{H^1}\bigg) {\rm d}\tau\nonumber\\
&+\int_0^t \|\nabla\widetilde{f}^\mu\|_{L_{x,v}^2} \| \widetilde{u}^\mu\|_{L^6} \|\nabla \varrho\|_{L^6} \big( \|\nabla\{\mathbf{I}-\mathbf{P}\}f^\mu\|_{\nu}+\|\nabla(a^\mu,b^\mu)\|_{L^2}    \big){\rm d}\tau\nonumber\\
&+\int_0^t \|\nabla\widetilde{f}^\mu\|_{L_{x,v}^2} \| \widetilde{u}^\mu\|_{L^6} \|  \varrho\|_{L^6} \big( \|\nabla^2\{\mathbf{I}-\mathbf{P}\}f^\mu\|_{\nu}+\|\nabla^2(a^\mu,b^\mu)\|_{L^2}    \big){\rm d}\tau\nonumber\\
\lesssim&\,\eps_0^\frac{1}{2} \|\nabla\widetilde{f}^\mu\|_{L_t^\infty(L_{x,v}^2)} \|(\widetilde{\varrho}^\mu,\widetilde{u}^\mu)\|_{L_t^\infty(H^1)}  \bigg(\sum_{1\leq|\alpha|\leq 2}\|\partial^\alpha\{\mathbf{I}-\mathbf{P}\}f^\mu\|_{L_t^2(\nu)}+\|\nabla(a^\mu,b^\mu)\|_{L_t^2(H^1)}\bigg) \nonumber\\
\lesssim&\, \big(\eps_0^\frac{1}{2}+\eps_1^\frac{1}{2}\big)\widetilde{\mathcal{X}}^\mu(t),\nonumber\\
\widetilde{J}_{14}\lesssim &\, \int_0^t\big( \|\nabla \{\mathbf{I}-\mathbf{P}\}\widetilde{f}^\mu\|_{\nu} +\|\nabla(\widetilde{a}^\mu,\widetilde{b}^\mu)\|_{L^2}\big)(\|\widetilde{\varrho}^\mu\|_{L^6}+\|\nabla\widetilde{\varrho}^\mu\|_{L^2} )\sum_{|\alpha|\leq
2}\|\partial^\alpha\{\mathbf{I}-\mathbf{P}\}f^\mu\|_{\nu} {\rm d}\tau \nonumber\\
&+\int_{0}^t \|\nabla\widetilde{b}^\mu\|_{L^2} \big( \|\widetilde{\varrho}^\mu\|_{L^6}\|\nabla(b^\mu-u^\mu)\|_{L^3}+\|\nabla\widetilde{\varrho}^\mu\|_{L^2}\|\nabla(b^\mu-u^\mu)\|_{L^\infty}  \big){\rm d}\tau\nonumber\\
&+\int_{0}^t \|\nabla\widetilde{b}^\mu\|_{L^2} \big( \| {\varrho} \|_{L^\infty}\|\nabla(\widetilde{b}^\mu-\widetilde{u}^\mu)\|_{L^2}+\|\nabla{\varrho} \|_{L^3} \| \widetilde{b}^\mu-\widetilde{u}^\mu \|_{L^6}  \big){\rm d}\tau\nonumber\\
\lesssim&\, \|\widetilde\varrho^\mu\|_{L_t^\infty(H^1)}\big( \|\nabla \{\mathbf{I}-\mathbf{P}\}\widetilde{f}^\mu\|_{L_t^2(\nu)} +\|\nabla(\widetilde{a}^\mu,\widetilde{b}^\mu)\|_{L_t^2(L^2)}\big)\sum_{|\alpha|\leq
2}\|\partial^\alpha\{\mathbf{I}-\mathbf{P}\}f^\mu\|_{L_t^2(\nu)}\nonumber\\
&+ \|\nabla\widetilde{b}^\mu\|_{L_t^2(L^2)}\big(\|\widetilde\varrho^\mu\|_{L_t^\infty(H^1)} \|\nabla(b^\mu-a^\mu)\|_{L_t^2(H^2)} +\|\varrho\|_{L_t^\infty(H^3)} \|\nabla(\widetilde{b}^\mu-\widetilde{a}^\mu)\|_{L_t^2(H^2)}\big)\nonumber\\
\lesssim&\, \big(\eps_0^\frac{1}{2}+\eps_1^\frac{1}{2}\big)\widetilde{\mathcal{X}}^\mu(t).
\end{align}

Plugging the estimates \eqref{G4.17}–\eqref{G4.24} into \eqref{G4.16} and applying \eqref{G2.5},
 we obtain the desired   estimate \eqref{G4.15}.
\end{proof}

\begin{lem}\label{L4.3}
It holds that
\begin{align}\label{G4.25}
\int_0^t \|\nabla(\widetilde{a}^\mu,\widetilde{b}^\mu)\|_{L^2}^2{\rm d}\tau\leq C\mu^2+C \widetilde{\mathcal{X}}^\mu(t)+\widetilde{\mathcal{X}}^\mu(0),
\end{align}
where $C>0$ is a constant independent of $\mu$ and time $t$.
\end{lem}
\begin{proof}
It follows from the systems \eqref{G4.1} and \eqref{G4.2} that we can derive the expressions for $\widetilde{a}^{\mu}$ and $\widetilde{b}^{\mu}$:
\begin{equation}\label{G4.26}
\left\{\begin{aligned}
&\partial_{t}\widetilde{a}^{\mu}+{\rm div} \widetilde{b}^{\mu}=0,\\
&\partial_{t} \widetilde{b}^{\mu}_i+\partial_{i} \widetilde{a}^{\mu}+\sum_{j=1}^3\partial_j\Gamma_{ij} (\{\mathbf{I}-\mathbf{P}\}\widetilde{f}^{\mu})\\
&\quad=\widetilde{\varrho}^{\mu}(u_i^{\mu}-b_i^{\mu})+(1+\varrho)(\widetilde{u}^{\mu}_i-\widetilde{b}_i^{\mu})-\widetilde{\varrho}^{\mu}u_i^{\mu} a^{\mu}+(1+\varrho)(\widetilde{u}^{\mu}_ia^{\mu}+u_i\widetilde{a}^{\mu}),  \\
&\partial_{i}\widetilde{b}^{\mu}_j+\partial_j \widetilde{b}^{\mu}_i-\widetilde{\varrho}^{\mu}(u_i^{\mu}b_j^{\mu}+u_j^{\mu}b_i^{\mu})-(1+\varrho)(\widetilde{u}_i^{\mu}b^{\mu}_j+u_i\widetilde{b}_j^{\mu}+\widetilde{u}_j^{\mu}b^{\mu}_i+u_j\widetilde{b}_i^{\mu})\\
& \quad=-\partial_t \Gamma_{ij}(\{\mathbf{I}-\mathbf{P}\}\widetilde{f}^{\mu})+\Gamma_{ij}(\widetilde{\mathfrak{l}}^{\mu}
+\widetilde{\mathfrak{r}}^{\mu}+\widetilde{\mathfrak{s}}^{\mu}),  
 \end{aligned}
 \right.
\end{equation}
for $1\leq i,j\leq 3$. Here, $\widetilde{\mathfrak{l}}^{\mu}$, $\widetilde{\mathfrak{r}}^{\mu}$ and $\widetilde{\mathfrak{s}}^{\mu}$ are given by
\begin{align*} 
\widetilde{\mathfrak{l}}^{\mu}:=\,&\mathcal{L}\{\mathbf{I}-\mathbf{P}\}\widetilde{f}^{\mu}-v\cdot\nabla\{\mathbf{I}-\mathbf{P}\}\widetilde{f}^{\mu}, \nonumber \\
\widetilde{\mathfrak{r}}^{\mu}:=\,&-(\widetilde{u}^{\mu}\cdot\nabla_v\{\mathbf{I}-\mathbf{P}\}f^{\mu}+u\cdot\nabla_v\{\mathbf{I}-\mathbf{P}\}\widetilde{f}^{\mu})+\frac{1}{2}(\widetilde{u}^{\mu}\cdot v\{\mathbf{I}-\mathbf{P}\}f^{\mu}+u\cdot v\{\mathbf{I}-\mathbf{P}\}\widetilde{f}^{\mu}),  \nonumber\\
\widetilde{\mathfrak{s}}^{\mu}:=\,&\frac{\widetilde{\varrho}^{\mu}}{\sqrt{M}}\nabla_v\cdot (\nabla_v (\sqrt{M}\{\mathbf{I}-\mathbf{P}\}f^{\mu} )+v\sqrt{M}\{\mathbf{I}-\mathbf{P}\}f^{\mu}-u^{\mu}\sqrt{M}\{\mathbf{I}-\mathbf{P}\}f^{\mu}             )\nonumber\\
&+\frac{ {\varrho} }{\sqrt{M}}\nabla_v\cdot (\nabla_v (\sqrt{M}\{\mathbf{I}-\mathbf{P}\}\widetilde{f}^{\mu} )+v\sqrt{M}\{\mathbf{I}-\mathbf{P}\}\widetilde{f}^{\mu}           )\nonumber\\
&-\frac{ {\varrho} }{\sqrt{M}}\nabla_v\cdot ( \widetilde{u}^{\mu}\sqrt{M}\{\mathbf{I}-\mathbf{P}\}f^{\mu}  +u  \sqrt{M}\{\mathbf{I}-\mathbf{P}\}\widetilde{f}^{\mu}                  ).
\end{align*}
We first compute the dissipation associated with $\widetilde{a}^{\mu}$. It follows from \eqref{G4.26}$_1$–\eqref{G4.26}$_2$ that  
\begin{align}\label{G4.27}
 \|  \nabla \widetilde{a}^{\mu}\|_{L^2}^2 
=\,&\sum_{i=1}^3\int_{\mathbb{R}^3} \partial_i \widetilde{a}^{\mu} \partial_i \widetilde{a}^{\mu}\mathrm{d}x\nonumber\\
=\,&\sum_{i=1}^3\int_{\mathbb{R}^3} \partial_i \widetilde{a}^{\mu} \bigg(-\partial_{t} \widetilde{b}^{\mu}_i-\sum_{j=1}^3\partial_j\Gamma_{ij}(\{\mathbf{I}-\mathbf{P}\}\widetilde{f}^{\mu}) 
+ \widetilde{\varrho}^{\mu} (u^{\mu}_i-b^{\mu}_i)        \bigg)\mathrm{d}x\nonumber\\
&+\sum_{i=1}^3\int_{\mathbb{R}^3}\partial_i \widetilde{a}^{\mu}\big((1+\varrho)(\widetilde{u}^{\mu}_i-\widetilde{b}_i^{\mu})-\widetilde{\varrho}^{\mu}u_i^{\mu}a^{\mu}+(1+\varrho)(\widetilde{u}^{\mu}_ia^{\mu}+u_i\widetilde{a}^{\mu})   \big){\rm d}x\nonumber\\
=\,&-\frac{{\rm d}}{{\rm d}t}\sum_{i=1}^3\int_{\mathbb{R}^3} \partial_i \widetilde{a}^{\mu}  \widetilde{b}^{\mu}_i\mathrm{d}x+\sum_{i=1}^3\int_{\mathbb{R}^3} \partial_i\partial_t \widetilde{a}^{\mu}
  \widetilde{b}^{\mu}_i\mathrm{d}x\nonumber\\
&+\sum_{i=1}^3\int_{\mathbb{R}^3} \partial_i \widetilde{a}^{\mu} \bigg(-\sum_{j=1}^3\partial_j\Gamma_{ij}(\{\mathbf{I}-\mathbf{P}\}\widetilde{f}^{\mu}) 
+ \widetilde{\varrho}^{\mu} (u^{\mu}_i-b^{\mu}_i)+(1+\varrho)(\widetilde{u}^{\mu}_i-\widetilde{b}_i^{\mu})        \bigg)\mathrm{d}x\nonumber\\
&+\sum_{i=1}^3\int_{\mathbb{R}^3} \partial_i \widetilde{a}^{\mu}\big( -\widetilde{\varrho}^{\mu}u_i^{\mu}a^{\mu}+(1+\varrho)(\widetilde{u}^{\mu}_ia^{\mu}+u_i\widetilde{a}^{\mu})   \big){\rm d}x.    
\end{align}
From \eqref{G4.26}$_1$, we have
\begin{align}\label{G4.28}
\sum_{i=1}^3\int_{\mathbb{R}^3} \partial_i\partial_t \widetilde{a}^{\mu}
  \widetilde{b}^{\mu}_i\mathrm{d}x=-\int_{\mathbb{R}^3} \partial_t \widetilde{a}^{\mu}
 \nabla\cdot \widetilde{b}^{\mu}\mathrm{d}x =\| \nabla\cdot \widetilde{b}^{\mu}\|_{L^2}^2.      
\end{align}
For the remaining terms on the right-hand side of \eqref{G4.27}, we compute
\begin{align}\label{G4.29}
&\sum_{i=1}^3\int_{\mathbb{R}^3} \partial_i \widetilde{a}^{\mu} \bigg(-\sum_{j=1}^3\partial_j\Gamma_{ij}(\{\mathbf{I}-\mathbf{P}\}\widetilde{f}^{\mu}) 
+ \widetilde{\varrho}^{\mu} (u^{\mu}_i-b^{\mu}_i)+(1+\varrho)(\widetilde{u}^{\mu}_i-\widetilde{b}_i^{\mu})        \bigg)\mathrm{d}x\nonumber\\
&+\sum_{i=1}^3\int_{\mathbb{R}^3} \partial_i \widetilde{a}^{\mu}\big( -\widetilde{\varrho}^{\mu}u_i^{\mu}a^{\mu}+(1+\varrho)(\widetilde{u}^{\mu}_ia^{\mu}+u_i\widetilde{a}^{\mu})   \big)  {\rm d}x  \nonumber\\
\leq\,& \frac{1}{4}\|\nabla \widetilde{a}^{\mu}\|_{L^2}^2+C\|\nabla\{\mathbf{I}-\mathbf{P}\}\widetilde{f}^{\mu}\|_{L_{x,v}^2}^2+C(1+\|\varrho\|_{H^3})^2\|\widetilde{u}^{\mu}-\widetilde{b}^{\mu}\|_{L^2}^2\nonumber\\
&+C\| \widetilde{\varrho}^{\mu}\|_{H^1}^2  \| \nabla (u ^{\mu}, a^{\mu}, b^{\mu})\|_{H^{2}}^2 +C(1+\|\varrho\|_{H^3})^2\big(\|\widetilde{u}^{\mu}\|_{H^1}^2\|\nabla a^{\mu}\|_{H^2}^2+\|u\|_{H^3}^2\|\widetilde{a}^{\mu}\|_{L^6}^2\big)\nonumber\\
\leq\,&\Big(\frac{1}{4}+C(\eps_0+\eps_1)\Big)\|\nabla \widetilde{a}^{\mu}\|_{L^2}^2+C\|\widetilde{b}^{\mu}-\widetilde{u}^{\mu}\|_{L^2}^2+C\| \{\mathbf{I}-\mathbf{P}\}\widetilde{f}^{\mu} \|_{L_v^2(H^1)}^2\nonumber\\
&+C\|\widetilde{u}^{\mu}\|_{H^1}^2\|\nabla a^{\mu}\|_{L^2}^2
+C\|\widetilde{\varrho}^{\mu}\|_{H^1}^2  \|\nabla (u ^{\mu}, a^{\mu}, b^{\mu})\|_{H^{2}}^2.
\end{align}
Putting the estimates \eqref{G4.28}--\eqref{G4.29} into \eqref{G4.27}, we arrive at
\begin{align}\label{G4.30}
&-\frac{{\rm d}}{{\rm d}t} \int_{\mathbb{R}^3}\widetilde{a}^{\mu} \nabla\cdot \widetilde{b}^{\mu}\mathrm{d}x+\Big(\frac{3}{4}-C\eps_0\Big)\|\nabla \widetilde{a}^{\mu}\|_{L^2}^2\nonumber\\
\,&\quad \leq\|\nabla\cdot \widetilde{b}^{\mu}\|_{L^2}^2+C\|\widetilde{u}^{\mu}-\widetilde{b}^{\mu}\|_{L^2}^2+C\| \{\mathbf{I}-\mathbf{P}\}\widetilde{f}^{\mu} \|_{L_v^2(H^1)}^2+ C\|\widetilde{u}^{\mu}\|_{H^1}^2\|\nabla a^{\mu}\|_{L^2}^2\nonumber\\
&\qquad +C\| \widetilde{\varrho}^{\mu}\|_{H^1}^2  \| \nabla (u ^{\mu}, a^{\mu}, b^{\mu})\|_{H^{2}}^2.
\end{align}

Next, we shall estimate the dissipation of $\widetilde{b}^\mu$. It is obvious that  
\begin{align} \label{G4.31}
\sum_{i,j=1}^3\| \partial_i\widetilde{b}^{\mu}_j+\partial_j\widetilde{b}^{\mu}_i \|_{L^2}^2=2\|\nabla  \widetilde{b}^{\mu}\|_{L^2}^2+2\|\nabla\cdot  \widetilde{b}^{\mu}\|_{L^2}^2.  
\end{align}
From \eqref{G4.26}$_3$, we arrive at
\begin{align} \label{G4.32}
&\sum_{i,j=1}^3\| \partial_i\widetilde{b}^{\mu}_j+\partial_j\widetilde{b}^{\mu}_i \|_{L^2}^2\nonumber\\
=\,&\sum_{i,j=1}^3\int_{\mathbb{R}^3} (\partial_i\widetilde{b}^{\mu}_j+\partial_j\widetilde{b}^{\mu}_i)
 \big((1+\varrho)(\widetilde{u}_i^{\mu}b^{\mu}_j+u_i\widetilde{b}_j^{\mu}+\widetilde{u}_j^{\mu}b^{\mu}_i+u_j\widetilde{b}_i^{\mu})+\widetilde{\varrho}^{\mu}(u_i^{\mu}b_j^{\mu}+u_j^{\mu}b_i^{\mu}) \big)\mathrm{d}x        \nonumber\\
&+\sum_{i,j=1}^3\int_{\mathbb{R}^3} (\partial_i\widetilde{b}^{\mu}_j+\partial_j\widetilde{b}^{\mu}_i)
 \big( -\partial_t \Gamma_{ij}(\{\mathbf{I}-\mathbf{P}\}\widetilde{f}^{\mu})+\Gamma_{ij}(\widetilde{\mathfrak{l}}^{\mu}+\widetilde{\mathfrak{r}}^{\mu}+\widetilde{\mathfrak{s}}^{\mu})\big)\mathrm{d}x   \nonumber\\
=\,&-\frac{{\rm d}}{{\rm d}t}\sum_{i,j=1}^3\int_{\mathbb{R}^3} (\partial_i\widetilde{b}^{\mu}_j+\partial_j\widetilde{b}^{\mu}_i) \Gamma_{ij}(\{\mathbf{I}-\mathbf{P}\}\widetilde{f}^{\mu})\mathrm{d}x\nonumber\\
&+\sum_{i,j=1}^3\int_{\mathbb{R}^3} (\partial_i\partial_t \widetilde{b}^{\mu}_j+\partial_j\partial_t \widetilde{b}^{\mu}_i)\partial^\alpha\Gamma_{ij}(\{\mathbf{I}-\mathbf{P}\}\widetilde{f}^{\mu})\mathrm{d}x\nonumber\\
&+\sum_{i,j=1}^3\int_{\mathbb{R}^3} (\partial_i\widetilde{b}^{\mu}_j+\partial_j\widetilde{b}^{\mu}_i) 
\big((1+\varrho)(\widetilde{u}_i^{\mu}b^{\mu}_j+u_i\widetilde{b}_j^{\mu}+\widetilde{u}_j^{\mu}b^{\mu}_i+u_j\widetilde{b}_i^{\mu})+\widetilde{\varrho}^{\mu}(u_i^{\mu}b_j^{\mu}+u_j^{\mu}b_i^{\mu}) \big)\mathrm{d}x\nonumber\\
&+\sum_{i,j=1}^3\int_{\mathbb{R}^3} (\partial_i\widetilde{b}^{\mu}_j+\partial_j\widetilde{b}^{\mu}_i)  \Gamma_{ij}(\widetilde{\mathfrak{l}}^{\mu}+\widetilde{\mathfrak{r}}^{\mu}+\widetilde{\mathfrak{s}}^{\mu})  \mathrm{d}x.
\end{align}
By using \eqref{G3.1}, Lemmas \ref{L2.1}--\ref{L2.2}, we obtain
\begin{align}\label{G4.33}
&\sum_{i,j=1}^3\int_{\mathbb{R}^3} (\partial_i\partial_t \widetilde{b}^{\mu}_j+\partial_j\partial_t \widetilde{b}^{\mu}_i)\partial^\alpha\Gamma_{ij}(\{\mathbf{I}-\mathbf{P}\}\widetilde{f}^{\mu})\mathrm{d}x\nonumber\\
=\,& -2\sum_{i,j=1}^3\int_{\mathbb{R}^3} \partial_t \widetilde{b}^{\mu}_i
 \partial_j\Gamma_{ij}(\{\mathbf{I}-\mathbf{P}\}\widetilde{f}^{\mu})\mathrm{d}x\nonumber\\
=\,& 2\sum_{i,j=1}^3\int_{\mathbb{R}^3} \bigg(\partial_i \widetilde{a}^{\mu}+\sum_{m=1}^3\partial_m\Gamma_{im}(\{\mathbf{I}-\mathbf{P}\}\widetilde{f}^{\mu})       \bigg) \partial_j\Gamma_{ij}(\{\mathbf{I}-\mathbf{P}\}\widetilde{f}^{\mu})\mathrm{d}x\nonumber\\
&-2\sum_{i,j=1}^3\int_{\mathbb{R}^3} \big((1+\varrho)(\widetilde{u}^{\mu}_i-\widetilde{b}^{\mu}_i)+\widetilde{\varrho}(u_i^{\mu}-b_i^{\mu})  \big) \partial_j\Gamma_{ij}(\{\mathbf{I}-\mathbf{P}\}\widetilde{f}^{\mu})\mathrm{d}x\nonumber\\
&-2 \sum_{i,j=1}^3\int_{\mathbb{R}^3}    \big(-\widetilde{\varrho}^{\mu}u_i^{\mu} a^{\mu}+(1+\varrho)(\widetilde{u}^{\mu}_ia^{\mu}+u_i\widetilde{a}^{\mu})\big)\partial_j\Gamma_{ij}(\{\mathbf{I}-\mathbf{P}\}\widetilde{f}^{\mu})\mathrm{d}x\nonumber\\
\leq\,& \frac{1}{4}\|\nabla \widetilde{a}^{\mu}\|_{L^2}^2+C\|\nabla \{\mathbf{I}-\mathbf{P}\}f \|_{L_{x,v}^2}^2+C(1+\|\varrho\|_{H^{3}})^2\|\widetilde{u}-\widetilde{b}^{\mu}\|_{L^2}^2+C\| \widetilde{\varrho}^{\mu}\|_{L^2}^2\| (u_i^{\mu}-b_i^{\mu})\|_{L^{\infty}}^2\nonumber\\
&+C\| \widetilde{\varrho}^{\mu}\|_{L^6}^2\| u_i^{\mu} \|_{L^{6}}^2\|a^{\mu}\|_{L^6}^2 +C(1+\|\varrho\|_{H^3})^2(\|\widetilde{u}^{\mu}\|_{L^2}^2\|  a^{\mu}\|_{L^{\infty}}^2+\|\widetilde{a}^{\mu}\|_{L^6}^2\|  u\|_{L^3}^2)\nonumber\\
\leq\,&\Big(\frac{1}{4}+C(\eps_0+\eps_1)\Big)\|\nabla \widetilde{a}^{\mu}\|_{L^2}^2
+C\|\nabla \{\mathbf{I}-\mathbf{P}\}\widetilde{f}^{\mu} \|_{L_{x,v}^2 }^2
+C\|\widetilde{u}^{\mu}-\widetilde{b}^{\mu}\|_{L^2}^2
\nonumber\\
&+C\|\nabla a^{\mu}\|_{L^2}^2\|\widetilde{u}^{\mu}\|_{L^2}^2
+C(1+ \eps_0+\eps_1 )\| \widetilde{\varrho}^{\mu}\|_{H^1}^2\| \nabla(u ^{\mu},b ^{\mu})\|_{H^{1}}^2.
\end{align}
It is easy to verify that
\begin{align}\label{G4.34}
&\sum_{i,j=1}^3\int_{\mathbb{R}^3} (\partial_i\widetilde{b}^{\mu}_j+\partial_j\widetilde{b}^{\mu}_i) 
\big((1+\varrho)(\widetilde{u}_i^{\mu}b^{\mu}_j+u_i\widetilde{b}_j^{\mu}+\widetilde{u}_j^{\mu}b^{\mu}_i+u_j\widetilde{b}_i^{\mu})+\widetilde{\varrho}^{\mu}(u_i^{\mu}b_j^{\mu}+u_j^{\mu}b_i^{\mu}) \big)\mathrm{d}x\nonumber\\
&+\sum_{i,j=1}^3\int_{\mathbb{R}^3} (\partial_i\widetilde{b}^{\mu}_j+\partial_j\widetilde{b}^{\mu}_i)  \Gamma_{ij}(\widetilde{\mathfrak{l}}^{\mu}+\widetilde{\mathfrak{r}}^{\mu}+\widetilde{\mathfrak{s}}^{\mu})  \mathrm{d}x\nonumber\\
\leq\,&\frac{1}{2}\sum_{i,j=1}^3\|\partial_i\widetilde{b}^{\mu}_j+\partial_j\widetilde{b}^{\mu}_i\|_{L^2}^2+C\sum_{i,j=1}^3
 \|  (1+\varrho)(\widetilde{u}_i^{\mu}b^{\mu}_j+u_i\widetilde{b}_j^{\mu}+\widetilde{u}_j^{\mu}b^{\mu}_i+u_j\widetilde{b}_i^{\mu})  \|_{L^2}^2\nonumber\\
&+C\sum_{i,j=1}^3\big(\|\partial^\alpha\Gamma_{ij}(\widetilde{\mathfrak{l}}^{\mu})\|_{L ^2}^2
+\|\partial^\alpha\Gamma_{ij}(\widetilde{\mathfrak{r}}^{\mu})\|_{L ^2}^2+\|\partial^\alpha\Gamma_{ij}(\widetilde{\mathfrak{s}}^{\mu})\|_{L^2}^2                  \big)\nonumber\\
&+C\sum_{i,j=1}^3\|\widetilde{\varrho}^{\mu}(u_i^{\mu}b_j^{\mu}+u_j^{\mu}b_i^{\mu})\|_{L^2}^2.
\end{align}
Notice that the function $\Gamma_{ij}(\cdot)$ can incorporate any velocity derivative and velocity weight.
 Through direct calculations, we deduce  
\begin{align}\label{G4.35}
&\sum_{i,j=1}^3
 \|  (1+\varrho)(\widetilde{u}_i^{\mu}b^{\mu}_j+u_i\widetilde{b}_j^{\mu}+\widetilde{u}_j^{\mu}b^{\mu}_i+u_j\widetilde{b}_i^{\mu})  \|_{L^2}^2+\sum_{i,j=1}^3\|\widetilde{\varrho}^{\mu}(u_i^{\mu}b_j^{\mu}+u_j^{\mu}b_i^{\mu})\|_{L^2}^2\nonumber\\
&+\sum_{i,j=1}^3\big(\|\partial^\alpha\Gamma_{ij}(\widetilde{\mathfrak{l}}^{\mu})\|_{L ^2}^2
+\|\partial^\alpha\Gamma_{ij}(\widetilde{\mathfrak{r}}^{\mu})\|_{L^2}^2+\|\partial^\alpha\Gamma_{ij}(\widetilde{\mathfrak{s}}^{\mu})\|_{L^2}^2                  \big) \nonumber\\
\lesssim\,& (1+\|\varrho\|_{H^3})^2(\|\widetilde{u}^{\mu}\|_{L^2}^2\|\nabla b^{\mu}\|_{H^1}^2+\|\widetilde{b}^{\mu}\|_{L^2}^2\|\nabla u^{\mu}\|_{H^1}^2 )+\|\widetilde{\varrho}^{\mu}\|_{L^6}^2\|u ^{\mu} \|_{L^{6}}^2\|b^{\mu}\|_{L^6}^2\nonumber\\
&+(1+\|u\|_{H^3}^2)\|\{\mathbf{I}-\mathbf{P}\}\widetilde{f}^{\mu}\|_{L_v^2(H^1)}^2+\|\widetilde{u}^{\mu}\|_{L^6}^2\|\{\mathbf{I}-\mathbf{P}\} {f}^{\mu}\|_{L_v^2(H^3)}^2\nonumber\\
& +(1+\|u^{\mu}\|_{H^3}^2)\|\widetilde{\varrho}^{\mu}\|_{L^6}^2\|\{\mathbf{I}-\mathbf{P}\} {f}^{\mu}\|_{L_v^2(H^3)}^2+(1+\|u\|_{H^3}^2)\|\varrho\|_{L^{\infty}}^2\|\{\mathbf{I}-\mathbf{P}\}\widetilde{f}^{\mu}\|_{L_v^2(H^1)}^2\nonumber\\
&+\|\varrho\|_{L^{\infty}}^2\|\{\mathbf{I}-\mathbf{P}\} {f}^{\mu}\|_{L_v^2(H^3)}^2\|\widetilde{u}^{\mu}\|_{H^1}^2\nonumber\\
\lesssim\,&\|(\widetilde{u}^{\mu},\widetilde{b}^{\mu})\|_{H^1}^2\big(\|\nabla( {u}^{\mu}, {b}^{\mu})\|_{H^1}^2+\|\{{\mathbf{I}-\mathbf{P}\} {f}^{\mu}\|_{L_v^2(H^3)}^2}\big)+\|\{\mathbf{I}-\mathbf{P}\}\widetilde{f}^{\mu}\|_{L_v^2(H^1)}^2\nonumber\\
&+\|\widetilde{\varrho}^{\mu}\|_{H^1}^2 \big(\|\nabla (u ^{\mu} ,b^{\mu})\|_{H^{2}}^2+\|\{{\mathbf{I}-\mathbf{P}\} {f}^{\mu}\|_{L_v^2(H^3)}^2} \big). 
\end{align}
Substituting \eqref{G4.35} into \eqref{G4.34}, we obtain
\begin{align}\label{G4.36}
&\sum_{i,j=1}^3\int_{\mathbb{R}^3} (\partial_i\widetilde{b}^{\mu}_j+\partial_j\widetilde{b}^{\mu}_i) 
\big((1+\varrho)(\widetilde{u}_i^{\mu}b^{\mu}_j+u_i\widetilde{b}_j^{\mu}+\widetilde{u}_j^{\mu}b^{\mu}_i+u_j\widetilde{b}_i^{\mu})+\widetilde{\varrho}^{\mu}(u_i^{\mu}b_j^{\mu}+u_j^{\mu}b_i^{\mu}) \big)\mathrm{d}x\nonumber\\
&+\sum_{i,j=1}^3\int_{\mathbb{R}^3} (\partial_i\widetilde{b}^{\mu}_j+\partial_j\widetilde{b}^{\mu}_i)  \Gamma_{ij}(\widetilde{\mathfrak{l}}^{\mu}+\widetilde{\mathfrak{r}}^{\mu}+\widetilde{\mathfrak{s}}^{\mu})  \mathrm{d}x\nonumber\\
\,&\quad \leq \frac{1}{2}\sum_{i,j=1}^3\|\partial_i\widetilde{b}^{\mu}_j+\partial_j\widetilde{b}^{\mu}_i\|_{L^2}^2+C\|(\widetilde{u}^{\mu},\widetilde{b}^{\mu})\|_{H^1}^2\big(\|\nabla( {u}^{\mu}, {b}^{\mu})\|_{H^1}^2+\|\{{\mathbf{I}-\mathbf{P}\} {f}^{\mu}\|_{L_v^2(H^3)}^2}\big)\nonumber\\
&\qquad +C\|\widetilde{\varrho}^{\mu}\|_{H^1}^2 \big(\|\nabla (u ^{\mu} ,b^{\mu})\|_{H^{2}}^2+\|\{{\mathbf{I}-\mathbf{P}\} {f}^{\mu}\|_{L_v^2(H^3)}^2} \big)+C\|\{\mathbf{I}-\mathbf{P}\}\widetilde{f}^{\mu}\|_{L_v^2(H^1)}^2.
\end{align}

By inserting \eqref{G4.31}, \eqref{G4.33}, and \eqref{G4.36} into \eqref{G4.32}, it can be shown 
\begin{align}\label{G4.37}
&\frac{{\rm d}}{{\rm d}t} \sum_{i,j=1}^3\int_{\mathbb{R}^3}(\partial_i\widetilde{b}^{\mu}_j+\partial_j\widetilde{b}^{\mu}_i) \Gamma_{ij}(\{\mathbf{I}-\mathbf{P}\}\widetilde{f}^{\mu})\mathrm{d}x+\|\nabla \widetilde{b}^{\mu}\|_{L^2}^2+\|\nabla\cdot \widetilde{b}^{\mu}\|_{L^2}^2\nonumber\\
\,&\quad \leq \Big(\frac{1}{4}+C(\eps_0+\eps_1)\Big)\|\nabla \widetilde{a}^{\mu}\|_{L^2}^2+C\|(\widetilde{u}^{\mu},\widetilde{b}^{\mu})\|_{H^1}^2\big(\|\nabla( {u}^{\mu}, a^{\mu},{b}^{\mu})\|_{H^2}^2+\|\{{\mathbf{I}-\mathbf{P}\} {f}^{\mu}\|_{L_v^2(H^3)}^2}\big)\nonumber\\
&\qquad +C\|\widetilde{u}^{\mu}-\widetilde{b}^{\mu}\|_{L^2}^2+C\| \{\mathbf{I}-\mathbf{P}\}\widetilde{f}^{\mu} \|_{L_v^2(H^1)}^2 \nonumber\\
&\qquad +C\| \widetilde{\varrho}^{\mu}\|_{H^1}^2 \big(\| \nabla (u ^{\mu} ,b^{\mu})\|_{H^{2}}^2+\| \{{\mathbf{I}-\mathbf{P}\} {f}^{\mu}\|_{L_v^2(H^3)}^2} \big).   
\end{align}

Let's introduce a new temporal functional $\widetilde{\mathcal{E}}_0^\mu(t)$ as follows:
\begin{align}\label{G4.38}
\widetilde{\mathcal{E}}^{\mu}_0(t):= \sum_{i,j=1}^3\int_{\mathbb{R}^3} (\partial_i\widetilde{b}^{\mu}_j+\partial_j\widetilde{b}^{\mu}_i) \Gamma_{ij}(\{\mathbf{I}-\mathbf{P}\}\widetilde{f}^{\mu})\mathrm{d}x-  \widetilde{a}^{\mu} \nabla\cdot \widetilde{b}^{\mu}\mathrm{d}x. 
\end{align}

Combining \eqref{G4.30} and  \eqref{G4.37}, and using the definition of $\widetilde{\mathcal{E}}_0^\mu(t)$ leads to  
\begin{align*}
& \frac{{\rm d}}{{\rm d}t}\widetilde{\mathcal{E}}^{\mu}_0(t)+\|\nabla \widetilde{b}^{\mu}\|_{L^2}^2+\Big(\frac{1}{2}-C(\eps_0+\eps_1)  \Big)\|\nabla \widetilde{a}^{\mu}\|_{L^2}^2\nonumber\\
\,&\quad \leq C\|(\widetilde{u}^{\mu},\widetilde{b}^{\mu})\|_{H^1}^2\big(\|\nabla( {u}^{\mu}, a^{\mu},{b}^{\mu})\|_{H^2}^2+\|\{{\mathbf{I}-\mathbf{P}\} {f}^{\mu}\|_{L_v^2(H^3)}^2}\big)\nonumber\\
&\qquad +C\|\widetilde{\varrho}^{\mu}\|_{H^1}^2 \big(\| \nabla (u ^{\mu} ,b^{\mu})\|_{H^{2}}^2+\| \{{\mathbf{I}-\mathbf{P}\} {f}^{\mu}\|_{L_v^2(H^3)}^2} \big) \nonumber\\
&\qquad +C\| \{\mathbf{I}-\mathbf{P}\}\widetilde{f}^{\mu} \|_{L_v^2(H^1)}^2+C\|\widetilde{u}^{\mu}-\widetilde{b}^{\mu}\|_{L^2}^2,
\end{align*}
which gives rises to
\begin{align}\label{G4.39}
&\widetilde{\mathcal{E}}^{\mu}_0(t) +\int_0^t \|\nabla(\widetilde{a}^{\mu},\widetilde{b}^{\mu})(\tau)\|_{L^2}^2{\rm d}\tau\nonumber\\
\leq\,&
C\|(\widetilde{u}^{\mu},\widetilde{b}^{\mu})\|_{L_t^{\infty}(H^1)}^2\big(\|\nabla( {u}^{\mu}, a^{\mu},{b}^{\mu})\|_{L_t^2(H^2)}^2+\|\{{\mathbf{I}-\mathbf{P}\} {f}^{\mu}\|_{L_t^2(L_v^2(H^3))}^2}\big)+\widetilde{\mathcal{E}}^{\mu}_0(0)\nonumber\\
&+C\| \widetilde{\varrho}^{\mu}\|_{L_t^{\infty}(H^1)}^2 \big(\| \nabla (u ^{\mu} ,b^{\mu})\|_{L_t^2(H^{2})}^2+\| \{{\mathbf{I}-\mathbf{P}\} {f}^{\mu}\|_{L_t^2(L_v^2(H^3))}^2} \big)\nonumber\\
&+C\| \{\mathbf{I}-\mathbf{P}\}\widetilde{f}^{\mu} \|_{L_t^2(L_v^2(H^1))}^2 
+C\|\widetilde{u}^{\mu}-\widetilde{b}^{\mu}\|_{L_t^2(L^2)}^2\nonumber\\
\leq\,& C(1+\eps_0+\eps_1) \widetilde{\mathcal{X}}^\mu(t)+C(\eps_0+\eps_1) \mu^2+\widetilde{\mathcal{E}}^{\mu}_0(0).
\end{align}

By Young's inequality, we have
\begin{align}\label{G4.40}
|\widetilde{\mathcal{E}}_0^\mu(t) |\lesssim \|(\widetilde{\varrho}^\mu,\widetilde{u}^\mu)\|_{H^1}^2+\|\widetilde{f}^\mu\|_{L_v^2(H^1)}^2.   
\end{align}

From the estimates \eqref{G4.7}, \eqref{G4.15}, \eqref{G4.39}, and \eqref{G4.40},  the inequality \eqref{G4.25} follows directly. 
% Therefore, the proof of Lemma \ref{L4.3} is completed.
\end{proof}

\begin{lem}\label{L4.4}
It holds that
\begin{align}\label{G4.41}
\int_0^t \|\nabla\widetilde{\varrho}^\mu \|_{L^2}^2{\rm d}\tau\leq C\mu^2+C \widetilde{\mathcal{X}}^\mu(t)+\widetilde{\mathcal{X}}^\mu(0),
\end{align}
where $C>0$ is a constant independent of $\mu$ and time $t$.
\end{lem}
\begin{proof} 
From   \eqref{G4.3}$_2$, it holds that
\begin{align}\label{G4.42}
P^\prime(1) \|\nabla   \widetilde{\varrho}^\mu \|_{L^2}^2=&\,-\int_{\mathbb{R}^3} \nabla    \widetilde{\varrho}^\mu  \cdot \partial_t \widetilde{u}^\mu \mathrm{d} x+\int_{\mathbb{R}^3} \nabla    \widetilde{\varrho}^\mu\cdot  \frac{\mu\Delta u^\mu}{1+\varrho^\mu}     \mathrm{d} x\nonumber+\int_{\mathbb{R}^3}\nabla\widetilde{\varrho}^\mu\cdot(\widetilde{b}^\mu-\widetilde{u}^\mu)   \mathrm{d} x \\
& -\int_{\mathbb{R}^3} \nabla   \widetilde{\varrho}^\mu\cdot   ( \widetilde{a}^\mu u^\mu+a \widetilde{u}^\mu+\widetilde{u}^\mu\cdot\nabla u+u^\mu\cdot\nabla\widetilde{u}^\mu )\ \mathrm{d} x + \int_{\mathbb{R}^3}  \nabla   \widetilde{\varrho}^\mu\cdot \widetilde{F}_2 \mathrm{d} x\nonumber\\  
\equiv:&\, \sum_{j=1}^5\widetilde{K}_j.
\end{align}
Using \eqref{G4.3}$_1$, integration by parts, Lemma \ref{L2.1}, and Young's inequality yields
\begin{align}\label{G4.43}
\widetilde{K}_1=&\,-\frac{\mathrm{d}}{\mathrm{d} t} \int_{\mathbb{R}^3} \nabla    \widetilde{\varrho}^\mu \cdot \widetilde{u}^\mu \mathrm{d} x+\int_{\mathbb{R}^3}  {\rm div}   \widetilde{u}^\mu \big ( (1+ {\varrho}^\mu) {\rm div} \widetilde{u}^\mu+\widetilde{\varrho}^\mu{\rm div}u+\widetilde{u}^\mu \cdot \nabla \varrho+u^\mu\cdot\nabla\widetilde{\varrho}^\mu\big) \mathrm{d} x \nonumber\\   
\leq &\,-\frac{\mathrm{d}}{\mathrm{d} t} \int_{\mathbb{R}^3} \nabla    \widetilde{\varrho}^\mu \cdot \widetilde{u}^\mu \mathrm{d} x+C\|\nabla\widetilde{u}^\mu\|_{L^2}^2+C\big(\eps_0^\frac{1}{2}+\eps_1^\frac{1}{2}\big)\|\nabla(\widetilde{\varrho}^\mu,\widetilde{u}^\mu)\|_{L^2}^2.
\end{align}
For the term $\widetilde{K}_2$, by utilizing H\"{o}lder's and Young's inequalities, we reach
\begin{align} 
\widetilde{K}_2   \leq&\,\mu\|\nabla\widetilde{\varrho}^\mu\|_{L^2}\|\Delta u^\mu\|_{L^2}\leq \frac{1}{5}P^\prime (1) \|\nabla\widetilde{\varrho}^\mu\|_{L^2}^2+C   \mu^2 \|\nabla^2u^\mu\|_{L^2}^2.\label{G4.44}
\end{align}
Similarly, for the remaining terms $\widetilde{K}_3 , \widetilde{K}_4 $, and $ \widetilde{K}_5 $,
we have 
\begin{align}
\widetilde{K}_3   \leq &\, \|\nabla\widetilde{\varrho}^\mu\|_{L^2}\|\widetilde{b}^\mu-\widetilde{u}^\mu\|_{L^2} 
\leq  \frac{1}{5}P^\prime(1)\|\nabla\widetilde{\varrho}^\mu\|_{L^2}^2+ C\|\widetilde{b}^\mu-\widetilde{u}^\mu\|_{L^2}^2,\label{G4.45}\\
\widetilde{K}_4\leq&\, \frac{1}{5}P^\prime(1)\|\nabla\widetilde{\varrho}^\mu\|_{L^2}^2+C\| ( \widetilde{a}^\mu u^\mu+a \widetilde{u}^\mu+\widetilde{u}^\mu\cdot\nabla u+u^\mu\cdot\nabla\widetilde{u}^\mu )\|_{L^2}^2 \nonumber\\
\leq&\, \frac{1}{5}P^\prime(1)\|\nabla\widetilde{\varrho}^\mu\|_{L^2}^2+C\big(\eps_0^\frac{1}{2}
+\eps_1^\frac{1}{2}\big)\|\nabla(\widetilde{a}^\mu,\widetilde{b}^\mu)\|_{L^2}^2,\label{G4.45-1}\\
\widetilde{K}_5\leq &\, \frac{1}{5}P^\prime(1)\|\nabla\widetilde{\varrho}^\mu\|_{L^2}^2+C\|\widetilde{\varrho}^\mu\|_{L^6}^2\|\nabla\varrho\|_{L^3}^2+C\|\varrho^\mu\|_{L^\infty}^2\|\nabla\widetilde{\varrho}^\mu\|_{L^2}^2\nonumber\\
\leq&\,  \frac{1}{5}P^\prime(1)\|\nabla\widetilde{\varrho}^\mu\|_{L^2}^2
+C\big(\eps_0^\frac{1}{2}+\eps_1^\frac{1}{2}\big)\|\nabla(\widetilde{a}^\mu,\widetilde{b}^\mu)\|_{L^2}^2.\label{G4.45-2}
\end{align}
Inserting the estimates \eqref{G4.43}–\eqref{G4.45-2} into \eqref{G4.42} yields
\begin{align}\label{G4.46}
&\frac{{\rm d}}{{\rm d}t}\int_{\mathbb{R}^3} \nabla\widetilde{\varrho}^\mu  \cdot\widetilde{u}^\mu {\rm d}x+P^\prime(1)\|\nabla\widetilde{\varrho}^\mu\|_{L^2}^2\nonumber\\
&\quad \lesssim \|\nabla\widetilde{u}^\mu\|_{L^2}^2+\|\widetilde{b}^\mu-\widetilde{u}^\mu\|_{L^2}^2+ \big(\eps_0^\frac{1}{2}+\eps_1^\frac{1}{2}\big)\|\nabla(\widetilde{a}^\mu,\widetilde{b}^\mu)\|_{L^2}^2.
\end{align}

By utilizing the facts that
\begin{align*}
\|\nabla\widetilde{u}^\mu\|_{L^2}^2\lesssim   \|\nabla(\widetilde{b
}^\mu-\widetilde{u}^\mu)\|_{L^2}^2 +\|\nabla\widetilde{b}^\mu\|_{L^2}^2, 
\end{align*}
and
\begin{align*}
\bigg|  \int_{\mathbb{R}^3} \nabla\widetilde{\varrho}^\mu  \cdot\widetilde{u}^\mu {\rm d}x \bigg| \lesssim \|\widetilde{u}^\mu\|_{L^2}^2+\|\nabla\widetilde{\varrho}^\mu\|_{L^2}^2,
\end{align*}
and integrating \eqref{G4.46} over $[0,T]$, we subsequently combine Lemmas \ref{L4.1}--\ref{L4.3} to derive \eqref{G4.41}.
\end{proof}

Similar to the proof of Lemma \ref{L4.2}, 
we apply $\nabla_v$ to \eqref{G4.3}$_3$ and then employ the energy method to derive  
\begin{align}\label{G4.47}
&\sup_{\tau\in[0,t]}\|\nabla_v\{\mathbf{I}-\mathbf{P}\}\widetilde{f}^\mu(\tau)\|_{L_{x,v}^2}^2+ \int_0^t\|\nabla_v \{\mathbf{I}-\mathbf{P}\}\widetilde{f}^\mu(\tau)\|_{\nu}^2{\rm d}\tau \nonumber\\
&\quad \leq C\sum_{|\alpha|\leq 1}\|\partial^\alpha\{\mathbf{I}-\mathbf{P}\}f\|_{\nu}^2+C\|\nabla(\widetilde{a}^\mu,\widetilde{b}^\mu)\|_{L^2}^2+\widetilde{\mathcal{X}}^\mu(0).
\end{align}  
The detailed derivation is omitted for brevity. 
By combining Lemmas \ref{L4.1}--\ref{L4.2}, we obtain the following lemma.
\begin{lem}\label{L4.5}
It holds that
\begin{align}\label{G4.48}
&\sup_{\tau\in[0,t]}\|\nabla_v\{\mathbf{I}-\mathbf{P}\}\widetilde{f}^\mu(\tau)\|_{L_{x,v}^2}^2 \nonumber\\
&\quad + \int_0^t\|\nabla_v \{\mathbf{I}-\mathbf{P}\}\widetilde{f}^\mu(\tau)\|_{\nu}^2{\rm d}\tau\leq \widetilde{\mathcal{X}}^\mu(0)+C\mu^2+C \big(\eps_0^\frac{1}{2}  +\eps_1^\frac{1}{2}\big)\widetilde{\mathcal{X}}^\mu(t),
\end{align}
where $C>0$ is a constant independent of  $\mu$ and time $t$.
\end{lem}

With Lemmas \ref{L4.1}–\ref{L4.5} in hand, we are now in a position to prove Theorem \ref{T1.3}.
\begin{proof}[Proof of Theorem \ref{T1.3}]
From Lemmas \ref{L4.1}--\ref{L4.5} and the assumption \eqref{c1}, we deduce that  
\begin{align*}  
\widetilde{\mathcal{X}}^\mu(t) \lesssim \big(\eps_0^\frac{1}{2} + \eps_1^\frac{1}{2}\big)\widetilde{\mathcal{X}}^\mu(t) + \mu^2.  
\end{align*}  
Under the smallness assumptions of $\eps_0^\frac{1}{2}$ and $\eps_1^\frac{1}{2}$, it follows that  
\begin{align*}  
\widetilde{\mathcal{X}}^\mu(t) \lesssim \mu^2.  
\end{align*}  
Therefore, the estimate \eqref{c2} is rigorously verified, which completes the proof of Theorem \ref{T1.3}.       
\end{proof}

\medskip

\section{Analysis of the Linearized  Compressible Euler-VFP System}

\subsection{Linear analysis}
In this section, we derive the time-decay rates of $(\varrho, u, f)$ for the linearized Euler-VFP system of 
\eqref{C2}. Subsequently, we establish a higher-order Lyapunov inequality for the low-frequency decomposition of $(\varrho, u, f)$. To achieve this, we begin by considering the following linearized Cauchy problem:
\begin{equation}\label{G5.1}
\left\{\begin{aligned}
&\partial_t \varrho + {\rm div}u = 0, \\
&\partial_t u    +P^\prime(1) \nabla \varrho - (b - u) = 0, \\
&\partial_t f + v \cdot \nabla f - u \cdot v \sqrt{M}  - \mathcal{L} f =  {S}_f,
\end{aligned}\right.
\end{equation} 
with the initial data 
\begin{equation}\label{G5.2}
(\varrho,u,f)|_{t=0}=(\varrho_0,u_0,f_0)=(\varrho(0,x),u(0,x),f(0,x,v)), \quad (x, v)\in\mathbb{R}^{3}\times \mathbb{R}^3.
\end{equation}
Here, the source term $S_f$ in the linearized Fokker–Planck equation \eqref{G5.1}$_3$ is expressed as  
\begin{align*}  
S_f = \mathrm{div}_v G - \frac{1}{2} v \cdot G + h,  
\end{align*}  
where $G = G(t, x, v) \in \mathbb{R}^3$ and $h = h(t, x, v) \in \mathbb{R}$ satisfy  
\begin{align*}  
\mathbf{P}_0 G = 0, \quad \mathbf{P} h = 0.  
\end{align*}  

For the sake of presentation simplicity, we use $ U(t) = (\varrho(t), u(t), f(t)) $ to denote the solution of the problem \eqref{G5.1}--\eqref{G5.2}, and $ U_0 $ to represent the initial data, i.e., $ U_0 = (\varrho_0, u_0, f_0) $. Applying Duhamel's principle \eqref{G5.1}--\eqref{G5.2}, the solution can be expressed as  
\begin{align}\label{G5.3}  
U(t) = \mathbb{A}(t)U_0 + \int_0^t \mathbb{A}(t - s)(0, 0, S_f(s)) \, \mathrm{d}s,  
\end{align}  
where $ \mathbb{A}(t) $ denotes the solution operator associated with the case $ S_f = 0 $.
Applying the operator $\phi_0(D_x)$ to \eqref{G5.3}, we obtain  
\begin{align}\label{G5.4}  
U^L(t) = \mathbb{A}(t)U^L_{0} + \int_{0}^{t}\mathbb{A}(t-s)\big(0,0, {S}_f^L(s)\big)\,{\rm d}s,  
\end{align}  
where $U^L(t) := \big(\varrho^L(t), u^L(t),   f^L(t)\big)$. In the following, we will derive the time-decay properties of $U(t)$ and $U^L(t)$ .
\begin{thm}\label{T5.1}
Let $1\leq p\leq2$. For any $\alpha$, $\alpha^{\prime}$ with 
$\alpha^{\prime}\leq\alpha$ and $m=|\alpha-\alpha^{\prime}|$, we have 
\begin{align}
& \|\partial^{\alpha}\mathbb{A}(t)U_{0}\|_{\mathcal{Z}_{2}}
\lesssim (1+t )^{-\frac{3}{2}(\frac{1}{p}-\frac{1}{2})-\frac{m}{2}}
\big(\|\partial^{\alpha^{\prime}}U_{0}\|_{\mathcal{Z}_{p}}+\|\partial^{\alpha}U_{0}\|_{\mathcal{Z}_{2}}\big),
\label{G5.5}\\
&\Big\|\partial^{\alpha}\int_{0}^{t}\mathbb{A}(t-s)
\big( 0,0,S_f(s)\big){\rm d}s \Big\|_{\mathcal{Z}_{2}}^{2} \lesssim\int_0^t(1+t-s)^{-3(\frac{1}{p}-\frac{1}{2})-m}\nonumber\\
&\qquad  \qquad  \qquad  \qquad  
\times \big(\big\|\partial^{\alpha^{\prime}}\big(G(s),\nu^{-1/2}h(s)\big)\big\|_{\mathcal{Z}_{p}}^{2}
+\big\|\partial^{\alpha}\big(G(s),\nu^{-1/2}h(s)\big)\big\|_{\mathcal{Z}_{2}}^{2}\big){\rm d}s, \label{G5.6}
\end{align}
and
\begin{align}
 &\|\partial^{\alpha}\mathbb{A}(t)U_{0}^L\|_{\mathcal{Z}_{2}}
\lesssim  (1+t )^{-\frac{3}{2}(\frac{1}{p}-\frac{1}{2})-\frac{m}{2}}
\|\partial^{\alpha^{\prime}}U_{0}\|_{\mathcal{Z}_{p}},
  \label{G5.7}\\
&\Big\|\partial^{\alpha}\int_{0}^{t}\mathbb{A}(t-s)
\big( 0,0,S^L_f(s)\big){\rm d}s\Big\|_{\mathcal{Z}_{2}}^{2} \nonumber\\  
&\quad \lesssim\int_0^t(1+t-s)^{-3(\frac{1}{p}-\frac{1}{2})-m}   \big\|\partial^{\alpha^{\prime}}\big(G(s),\nu^{-1/2}h(s)\big) \big\|_{\mathcal{Z}_{p}}^{2}\label{G5.8}
 {\rm d}s, 
\end{align}
holds for all $t\geq 0$.
\end{thm}
\begin{proof}
Motivated by the hypercoercivity argument presented in 
\cite{SK-HMJ-1985}, we establish the proof of Theorem \ref{T5.1} 
through using Fourier analysis techniques. By performing the Fourier
 transform on $x$ to the equation \eqref{G5.1}, we obtain
\begin{equation}\label{G5.9}
\left\{\begin{aligned}
&\partial_{t}\widehat{\varrho}+i\xi\cdot \widehat{u}=0,  \\
&\partial_t \widehat{u}+i\xi P^{\prime}(1)\widehat{\varrho} -(\widehat{b}-\widehat{u})=0,\\
&\partial_{t}\widehat{f}+iv\cdot \xi  \widehat{f}-\widehat{u}\cdot v\sqrt{M}=\mathcal{L}\widehat{f}+\nabla_v\cdot\widehat{G}-\frac{1}{2}v\cdot\widehat{G}+\widehat{h},  
\end{aligned}\right.
\end{equation}
Multiplying \eqref{G5.9}$_1$ and \eqref{G5.9}$_2$ by $P^{\prime}(1)\overline{\widehat{\varrho}}$ and $\overline{\widehat{u}}$, respectively, and integrating the resulting expressions, adding them up, and taking real part of
the results  yield   
\begin{align}\label{G5.10}
\frac{1}{2}\partial_t(P^{\prime}(1)|\widehat{\varrho}|^2+|\widehat{u}|^2 )        
+|\widehat{u}|^2-\mathfrak{Re}   (\widehat{u}|\widehat{b})=0.
\end{align}
Multiplying \eqref{G5.9}$_3$ by $\overline{\widehat{f}}$ and integrating in $  v$ over $\mathbb{R}^3$, we have 
\begin{align*}
&\frac{1}{2}\partial_t \|\widehat{f}\|_{L_v^2}^2+\mathfrak{Re}\int_{\mathbb{R}^3} 
(-\mathcal{L}\{\mathbf{I}-\mathbf{P}\}\widehat{f}|\{\mathbf{I}-\mathbf{P}\}\widehat{f}\mathrm{d}v   )+|\widehat{b}|^2-\mathfrak{Re}(\widehat{u}|\widehat{b})\nonumber\\
 &\quad =\mathfrak{Re}\int_{\mathbb{R}^3}\Big(\nabla_v\cdot\widehat{G}-\frac{1}{2}v\cdot\widehat{G}|
\{\mathbf{I}-\mathbf{P}\}\widehat{f}\Big)\mathrm{d}v
+\mathfrak{Re}\int_{\mathbb{R}^3}(\widehat{h}|\{\mathbf{I}-\mathbf{P}\}\widehat{f})\mathrm{d}v.
\end{align*}
Here the following facts have been utilized:  
\begin{align*}
 \nabla_v\cdot G-\frac{1}{2}v\cdot G\perp {\rm Rang}\mathbf{P},\quad \mathbf{P}h=0.
\end{align*}
It can be inferred from \eqref{G2.5}, $\mathbf{P}_0G=0$, and   Cauchy-Schwarz's inequality that
\begin{align}\label{G5.11}
&\frac{1}{2}\partial_t \|\widehat{f}\|_{L_v^2}^2+\lambda_0 |\{\mathbf{I}-\mathbf{P}\}\widehat{f}|_{\nu}^2+|\widehat{b}|^2-\mathfrak{Re}(\widehat{u}|\widehat{b}) \lesssim \|\widehat{G}\|_{L^2_v}^2+\|\nu^{-\frac{1}{2}}\widehat{h}\|_{L^2_v}^2.
\end{align}
Combining \eqref{G5.10} and \eqref{G5.11}, one gets
\begin{align}\label{G5.12}
&\frac{1}{2}\partial_t(P^{\prime}(1)|\widehat{\varrho}|^2+\|\widehat{f}\|_{L_v^2}^2+|\widehat{u}|^2 )+ \lambda_0 |\{\mathbf{I}-\mathbf{P}\}\widehat{f}|_{\nu}^2+   |\widehat{b}-\widehat{u}|^2 
\lesssim   \|\widehat{G}\|_{L^2_v}^2+\|\nu^{-\frac{1}{2}}\widehat{h}\|_{L^2_v}^2 .
\end{align}
Following the same approach as in \cite[Lemma \ref{L3.1}]{DL-KRM-2013}, we obtain the following estimates for $\widehat{a}$ and $\widehat{b}$:
\begin{align}\label{G5.13}
\partial_t \mathfrak{Re}\mathcal{E}_1(\widehat{f})+\frac{ \lambda_8  |\xi |^2}{1+|\xi |^2} (|\widehat{a
}|^2+|\widehat{b}|^2)\lesssim \|\{\mathbf{I}-\mathbf{P}\}\widehat{f}\|_{L_v^2}^2
+|\widehat{u}-\widehat{b}|^2+\|G\|_{L_v^2}^2+\|\nu^\frac{1}{2}\widehat{h}\|_{L_v^2}^2,   
\end{align}
for some $\lambda_8>0$ independent of $\mu$, 
where 
\begin{align*}
\mathcal{E}_1(\widehat{f}):=\frac{1}{1+|\xi |^2}\sum_{i,j=1}^3\big(\{i\xi _i\widehat{b}_j+   i\xi _j\widehat{b}_i\}|\Gamma_{ij}(\{\mathbf{I}-\mathbf{P}\}\widehat{f})  
   \big)-\frac{1}{1+|\xi |^2}(\widehat{a}|i\xi \cdot\widehat{b}).   
\end{align*}

The additional consideration arises only when analyzing the dissipation of $\varrho$. To address this, we multiply \eqref{G5.9}$_2$ by $\overline{i\xi \widehat{\varrho}}$ and utilize \eqref{G5.9}$_1$, leading to the following derivation:
\begin{align*}
P^{\prime}(1)|\xi |^2|\widehat{\varrho}|^2 &=(\{-\partial_t\widehat{u}-\widehat{u}+\widehat{b}\}|i\xi \widehat{\varrho})=-\partial_t(\widehat{u}|i\xi \widehat{\varrho})+|\xi \cdot \widehat{u}|^2-
(\{\widehat{u}-\widehat{b}\}|i\xi \widehat{\varrho}),
\end{align*}
which  together with Young's inequality gives rise to
\begin{align}\label{G5.14}
\partial_t\frac{\mathfrak{Re}(\widehat{u}|i\xi \widehat{\varrho})}
{1+|\xi |^2}+\lambda_9P^{\prime}(1)\frac{|\xi |^2}{1+|\xi |^2}|
\widehat{\varrho}|^2\lesssim |\widehat{u}-\widehat{b}|^2+\frac{|\xi \cdot\widehat{u}|^2}{1+|\xi |^2},
\end{align}
for some $\lambda_9>0$ independent of $\mu$.

Let's define a functional $\mathcal{E}_{\mathcal{F}}(\widehat{\varrho},\widehat{u},\widehat{f})$ by
\begin{align*}
\mathcal{E}_{\mathcal{F}}(\widehat{\varrho},\widehat{u},\widehat{f})
:=P^{\prime}(1)|\widehat{\varrho}|^2+|\widehat{u}|^2+\|f\|_{L^2_v}^2
+\tau_4\mathfrak{Re}\mathcal{E}_1(\widehat{f})+\tau_5\frac{\mathfrak{Re}(\widehat{u}|i\xi \widehat{\varrho})}{1+|\xi |^2},
\end{align*}
where $0<\tau_4,\tau_5\ll1$ are small constants. 
Thanks to the facts that  
$$
|\mathcal{E}_1(\widehat{f})|\lesssim |\widehat{u}|^2+\|\widehat{f}\|_{L_v}^2,
\quad \Big|\frac{(\widehat{u}|i\xi \widehat{\varrho})}{1+|\xi |^2}\Big|\lesssim 
|\widehat\varrho|^2+|\widehat{u}|^2,
$$
we conclude that
\begin{align*}
\mathcal{E}_{\mathcal{F}}(\widehat{\varrho},\widehat{u},\widehat{f})\backsim |\widehat{\varrho}|^2+|\widehat{u}|^2+\|\widehat{f}\|_{L_v^2}^2.    
\end{align*}
Thus, for any $0 < \tau_5 \ll \tau_4 \ll 1$, it holds
\begin{align}\label{G5.15}
\partial_t\mathcal{E}_{\mathcal{F}}(\widehat{\varrho},\widehat{u},\widehat{f})
+\frac{\lambda_{10}|\xi |^2}{1+|\xi |^2}\mathcal{E}_{\mathcal{F}}(\widehat{\varrho},\widehat{u},\widehat{f})\lesssim   \|\widehat{G}\|_{L^2_v}^2+\|\nu^{-\frac{1}{2}}\widehat{h}\|_{L^2_v}^2 ,  
\end{align}
for some $\lambda_{10}>0$ independent of $\mu$,
Here, we have utilized the following inequalities:  
\begin{align*}  
 |\xi |^2|\hat{b}|^2 + | \hat{b} -  \hat{u}|^2 \geq&\, |\xi |^2|\hat{b}|^2 + |\xi |^2| \hat{b} - \hat{u}|^2  \geq  \frac{1}{2}|\xi |^2|\hat{u}|^2, \quad\qquad\qquad\,\,\,\,\,\,  |\xi | \leq \frac{r_0}{2},& \\  
|\xi |^2|\hat{b}|^2 +  | \hat{b} -  \hat{u}|^2   \geq&\, (\frac{r_0}{2})^2|\hat{b}|^2 +  | \hat{b} -  \hat{u}|^2  \geq \min\Big\{1,\Big(\frac{r_0}{2}\Big)^2\Big\}\frac{1}{2}|\hat{u}|^2,  \,\,\quad |\xi | \geq \frac{r_0}{2}, &
\end{align*} 
where $r_0$ is defined in \eqref{G2.2}.
Combining \eqref{G5.15} with Gronwall's inequality, we deduce that
\begin{align}\label{G5.16}
 \mathcal{E}_{\mathcal{F}}(\widehat{\varrho},\widehat{u},\widehat{f})\leq e^{-\frac{\lambda_{10}|\xi |^2}{1+|\xi |^2}}  \mathcal{E}_{\mathcal{F}}(\widehat{\varrho}_0,\widehat{u}_0,\widehat{f}_0)+\int_{0}^t
 e^{-\frac{\lambda_{10}|\xi |^2}{1+|\xi |^2}(t-s)}(\|\widehat{G}(s)\|_{L^2_v}^2+\|\nu^{-\frac{1}{2}}\widehat{h}(s)\|_{L^2_v}^2)\mathrm{d}s.
\end{align}
As in \cite{Ks-1983} and \cite[Theorem 3.1]{DFT-2010-CMP}, with the aid of \eqref{G5.16}, we can directly obtain the estimates \eqref{G5.5}--\eqref{G5.8} in Theorem \ref{T5.1}. For brevity, we omit the detailed proof here.
\end{proof}

\subsection{Higher-order Lyapunov inequality}
To derive the optimal time decay rates of classical solutions and their spatial derivatives 
 to the  Euler-VFP system \eqref{C2}, we establish the following higher-order Lyapunov inequality.
\begin{prop}\label{P5.2}
For the classical solution $(\varrho, u,f)$ to the  Euler-VFP system \eqref{C2}, there exist a positive constant $\lambda_{11}$ independent of $\mu$, such that 
\begin{align}\label{G5.17}  
&\frac{\rm d}{{\rm d}t}\big( \|\nabla^k(\varrho,u )\|_{L^2}^2+\|\nabla^k f\|_{L_{x,v}^2}^2\big) + \lambda_{11}\big(  \|\nabla^k(\varrho,u)\|_{L^2}^2+\|\nabla^k f\|_{L_{x,v}^2}^2\big)  \lesssim \|\nabla^k(\varrho^L,a^L,b^L)\|_{L^2}^2,  
\end{align}  
for any $k=2,3$ and $0 \leq t < T$.
\end{prop}
The proof of Proposition \ref{P5.2} relies on Lemmas \ref{L5.3}-- \ref{L5.5}, which are presented below.

By considering the cases $|\alpha| = 2$ and $|\alpha| = 3$ in Lemma \ref{L3.2}, respectively, we can derive the following lemma. For the sake of brevity, the proof is omitted here.
\begin{lem}\label{L5.3}
For the classical solution $(\varrho, u,f)$ to the   Euler-VFP system \eqref{C2}, there exist a positive constant $\lambda_{12}$ independent of $\mu$, such that
\begin{align}\label{G5.18}
& \frac{\rm d}{{\rm d}t} 
\bigg(\bigg\|\frac{\sqrt{P^\prime(1+\varrho)}}{1+\varrho}\nabla^k \varrho\bigg\|_{L^2}^2+\|\nabla^k u\|_{L^2}^2+\|\nabla^k f\|_{L_{x,v}^2}^2\bigg) \nonumber\\ &\quad+\lambda_{12}\big(\|\nabla^k(b-u)\|_{L^2}^2+\|\nabla^k\{\mathbf{I}-\mathbf{P}\}f\|_{\nu}^2\big)
\lesssim \eps_1^\frac{1}{2}\|\nabla^k(\varrho ,a,b)\|_{L^2}^2,
\end{align}
for any $0 \leq t<T$. 
\end{lem}
\begin{rem}
We   emphasize that the significance of the dissipative structure $b - u$, that is,  
\begin{align*}  
\|\nabla^k u\|_{L^2}^2 \lesssim \|\nabla^k (b - u)\|_{L^2}^2 + \|\nabla^k b\|_{L^2}^2,  
\end{align*}  
  plays a crucial role in the proof of Lemma \ref{L5.3}.
\end{rem}

Next, we apply the low-high frequencies decomposition method (see \cite{LNW-2025-preprint, Ww-CMS-2024}) to obtain more refined energy estimates for the classical solutions of the Euler-VFP system \eqref{C2}. We begin by estimating the dissipation of $(a^{H}, b^{H})$. By applying the operator $\phi_1(D_x)$ (Recall its definition in \eqref{G2.1}) to the equation \eqref{G3.19}, we obtain  
\begin{equation*} 
\left\{\begin{aligned}
&\partial_{t}a^{H}+\nabla\cdot b^{H}=0,\\
&\partial_{t} b_i^{H}+\partial_{i} a^{H}+\sum_{j=1}^3\partial_j\Gamma_{ij}(\{\mathbf{I}-\mathbf{P}\}f)^{H}=\big((1+\varrho)(u_i-b_i)\big)^{H}+\big((1+\varrho)u_ia\big)^{H},  \\
&\partial_{i}b_j^{H}+\partial_j b_i^{H}-\big((1+\varrho)(u_ib_j+u_jb_i)\big)^{H}=-\partial_t \Gamma_{ij}(\{\mathbf{I}-\mathbf{P}\}f)^{H}+\Gamma_{ij}(\mathfrak{l}+\mathfrak{r}+\mathfrak{s})^{H},  
 \end{aligned}
 \right.
\end{equation*}
for $1\leq i,j\leq 3$, where $\mathfrak{l},\mathfrak{r},\mathfrak{s}$, and $\Gamma_{ij}(\cdot)$ are defined 
by \eqref{lrs-1}--\eqref{gammaa}, respectively. Therefore, we define the following temporal functional
$\mathcal{E}_1^H(t)$, which only contains the high frequency part of $f(t,x,v)$, as
\begin{align}\label{G5.19}
\mathcal{E}_1^H(t) := \, & \sum_{i, j=1}^3 \int_{\mathbb{R}^3} \nabla^{k-1}(\partial_i b_j^H+\partial_j b_i^H) \nabla^{k-1} \Gamma_{ij}(\{\mathbf{I}-\mathbf{P}\} f)^H \mathrm{d}x\nonumber\\
& -\int_{\mathbb{R}^3} \nabla^{k-1} a^H \nabla^{k-1}\nabla\cdot b^H \mathrm{d} x.
\end{align}

As shown in \cite[Lemma 5.2]{LNW-2025-preprint}, we proved the case $k=2$ in Lemma~\ref{L5.4}. The case $k=3$ in Lemma~\ref{L5.4} can be derived through a similar process and is therefore omitted for brevity.

\begin{lem}\label{L5.4}
For the classical solution $(\varrho, u,f)$ to the  Euler-VFP system \eqref{C2}, there exist a positive constant $\lambda_{13}$ independent of $\mu$, such that 
\begin{align}\label{G5.20}  
& \frac{\mathrm{d}}{\mathrm{d} t} \mathcal{E}_1^H(t)+\lambda_{13}\big(\|\nabla^k a^H\|_{L^2}^2+\|\nabla^k b^H\|_{L^2}^2\big)\nonumber\\
 &\quad \lesssim \|\nabla^k\{\mathbf{I}-\mathbf{P}\} f\|_{L_{x,v}^2}^2+\|\nabla^k(b-u)\|_{L^2}^2 +\eps_1\|\nabla^k(\varrho,a,b)\|_{L^2}^2,
\end{align}  
for any $k=2,3$ and $0 \leq t < T$.    
\end{lem}

Finally, we present the estimates for $\|\nabla^k \varrho^H\|_{L^2}$ with $k = 2, 3$.
\begin{lem}\label{L5.5}
For the classical solution $(\varrho, u,f)$ to the  Euler-VFP system \eqref{C2}, there exist a positive constant $\lambda_{14}$ independent of $\mu$, such that 
\begin{align}   \label{G5.21}
&-\frac{\rm d}{{\rm d}t}\int_{\mathbb{R}^3}\nabla^{k-1}{\rm div} u\cdot\nabla^{k-1}\varrho^H\mathrm{d}x+\lambda_{14}\|\nabla^k \varrho^H\|_{L^2}^{2}\nonumber\\
& \quad  \lesssim  \|\nabla^k(b-u) \|_{L^2}^2+\|\nabla^k\varrho^L\|_{L^2}^2+\|\nabla^k u\|_{L^2}^2 +\eps_1\|\nabla^2( a,b)\|_{L^2}^2, 
\end{align}  
for any $k=2,3$ and $0 \leq t < T$.    
\end{lem}
\begin{proof}
By applying the operator $\nabla^k$ ($k = 2, 3$) to the equation \eqref{C2}$_2$, and subsequently multiplying the resulting equation by $\nabla^k \varrho^H$, followed by integrating over $\mathbb{R}^3$, we derive   
\begin{align}\label{G5.22}
P^\prime(1)\|\nabla^k\varrho^H\|_{L^2}^2=\,&-\int_{\mathbb{R}^3}\nabla^k\varrho^H\cdot\nabla^{k-1}\partial_t u\mathrm{d}x+\int_{\mathbb{R}^3}\nabla^k\varrho^H\cdot\nabla^{k-1}(b-u)\mathrm{d}x \nonumber\\
&+\int_{\mathbb{R}^3}\nabla^k\varrho^H\cdot\nabla^{k-1}(au)\mathrm{d}x-\int_{\mathbb{R}^3}\nabla^k\varrho^H\cdot\nabla^{k-1}(u\cdot\nabla u)\mathrm{d}x\nonumber\\
&-\int_{\mathbb{R}^3}\nabla^k\varrho^H\cdot\nabla^{k-1}\bigg(\Big(\frac{P^{\prime}(1)}{1+\varrho}-P^{\prime}(1)       \Big)\nabla\varrho\bigg)\mathrm{d}x + \int_{\mathbb{R}^3}\nabla^k\varrho^H\cdot\nabla^k\varrho^L\mathrm{d}x\nonumber\\
\equiv:\,&\sum_{j=1}^{6}L_j.
\end{align}
By employing \eqref{C2}$_1$ and  integration by parts, and applying Lemmas \ref{L2.1}–\ref{L2.3} together with Young’s inequality, the term $L_1$ can be estimated as follows:
\begin{align}\label{G5.23}
L_{1}=& -\frac{{\rm d}}{{\rm d}t}\int_{\mathbb{R}^3}\nabla^k  \varrho^H\cdot\nabla^{k-1} u\mathrm{d}x+\int_{\mathbb{R}^3}\nabla^k  \partial_t\varrho^H\cdot\nabla^{k-1} u\mathrm{d}x\nonumber\\
=& -\frac{{\rm d}}{{\rm d}t}\int_{\mathbb{R}^3}\nabla^k  \varrho^H\cdot\nabla^{k-1} u\mathrm{d}x+\int_{\mathbb{R}^3}\nabla^{k-1}\big((\varrho+1){\rm div} u+\nabla\varrho\cdot \ u            \big)^H \cdot \nabla^{k-1}{\rm div} u\mathrm{d}x\nonumber\\
\leq& -\frac{{\rm d}}{{\rm d}t}\int_{\mathbb{R}^3}\nabla^k  \varrho^H\cdot\nabla^{k-1} u\mathrm{d}x+C\|\nabla^{k-1}{\rm div}u
\|_{L^2}^2+C\eps_1\|\nabla^k(\varrho,u)\|_{L^2}^2.
\end{align}
For the terms $L_2, \dots, L_6$, by making use of \eqref{G2.3}, Young's inequality, and Lemmas \ref{L2.1}--\ref{L2.3}, it holds that
\begin{align}
L_2\leq&\, \|\nabla^{k-1}\varrho^H\|_{L^2}\|\nabla^k(b-u)\|_{L^2}\nonumber\\
\leq&\, \frac{1}{6}P^\prime(1)\|\nabla^k \varrho^H\|_{L^2}^2+C\|\nabla^k(b-u)\|_{L^2}^2,\label{G5.24-1}\\
L_3\leq&\, \frac{1}{6}P^\prime(1)\|\nabla^k \varrho^H\|_{L^2}^2+C\|\nabla^{k-1}(a,u)\|_{L^6}^2\|(a,u)\|_{L^3}^2\nonumber\\
\leq&\, \frac{1}{6}P^\prime(1)\|\nabla^k \varrho^H\|_{L^2}^2+c\eps_1 \|\nabla^k(a,u)\|_{L^2}^2,\label{G5.24-2}\\
L_4\leq&\, \frac{1}{6}P^\prime(1)\|\nabla^k \varrho^H\|_{L^2}^2+C\|u\|_{L^\infty}^2\|\nabla^k u\|_{L^2}^2+C\|\nabla u\|_{L^3}^2\|\nabla^{k-1} u\|_{L^6}^2\nonumber\\
\leq&\, \frac{1}{6}P^\prime(1)\|\nabla^k \varrho^H\|_{L^2}^2+C\eps_1\|\nabla^k u\|_{L^2}^2,\label{G5.24-3}\\
L_5\leq&\, \frac{1}{6}P^\prime(1)\|\nabla^k \varrho^H\|_{L^2}^2+C\|\varrho\|_{L^\infty}^2\|\nabla^k \varrho\|_{L^2}^2+C\|\nabla \varrho\|_{L^3}^2\|\nabla^{k-1} \varrho\|_{L^6}^2\nonumber\\
\leq&\, \frac{1}{6}P^\prime(1)\|\nabla^k \varrho^H\|_{L^2}^2+C\eps_1\|\nabla^k \varrho\|_{L^2}^2,\label{G5.24-4}\\
L_6\leq&\,\frac{1}{6}P^\prime(1)\|\nabla^k \varrho^H\|_{L^2}^2+ C\|\nabla^k \varrho^L\|_{L^2}^2.\label{G5.24}
\end{align}
Then, \eqref{G5.21} follows by plugging the estimates \eqref{G5.23}--\eqref{G5.24} into \eqref{G5.22}, and hence Lemma \ref{L5.5} is proved.\end{proof}

Below we return to the proof of Proposition \ref{P5.2}.
\begin{proof}[Proof of Proposition \ref{P5.2}]
We define the new energy functional and its corresponding dissipation rate functional as follows:  
\begin{align}\label{G5.25}  
\mathcal{E}_1(t) :=&\, \|\nabla^k(\varrho, u)\|_{L^2}^{2} + \|\nabla^k f\|_{L_{x,v}^2}^2 + \tau_6 \mathcal{E}_1^H(t)+   \tau_7 \int_{\mathbb{R}^3} \nabla^{k-1} \varrho^H \cdot \nabla^{k-1} \mathrm{div} u   {\rm d}x,  \\\label{G5.26}  
\mathcal{D}_1(t) :=&\, \|\nabla^k(b - u)\|_{L^2}^{2} + \|\nabla^k\{\mathbf{I} - \mathbf{P}\} f\|_{\nu}^{2} + \|\nabla^k(\varrho^H, a^H, b^H)\|_{L^2}^2,  
\end{align}  
where $0 < \tau_6, \tau_7 \ll 1$ are   constants, and $k = 2, 3$.
It follows  \eqref{G2.3}, the definition of $\mathcal{E}_1^H(t)$ in \eqref{G5.19}, Young's inequality, we infer that
\begin{align}\label{G5.27}
\mathcal{E}_1(t)\backsim \|\nabla^k(\varrho,u)\|_{L^2}^2+\|\nabla^2 f\|_{L_{x,v}^2}^2,\quad k=2,3.    
\end{align}
Adding \eqref{G5.18}, $\tau_6\times$\eqref{G5.20}, and $\tau_7\times$\eqref{G5.21} gives  
\begin{align}\label{G5.28}  
\frac{{\rm d}}{{\rm d}t}\mathcal{E}_1(t) + \lambda_{15} \mathcal{D}_1(t) \lesssim \|\nabla^k(\varrho^L, a^L, b^L)\|_{L^2}^2,  
\end{align}  
where $\lambda_{15} > 0$, independent of $\mu$, satisfies $\lambda_{15} < \min\{\lambda_{12}, \tau_6\lambda_{13}, \tau_7\lambda_{14}\}$.
Here, we have utilized the following property of the dissipative structures $ b - u $:  
\begin{align*}  
\|\nabla^k u\|_{L^2}^2 \lesssim \|\nabla^k(b - u)\|_{L^2}^2 + \|\nabla^k b\|_{L^2}^2,  
\end{align*}  
and the properties derived from \eqref{G2.1}–\eqref{G2.3}:  
\begin{align*}  
\|\nabla^k (\varrho, a, b)\|_{L^2}^2 \lesssim \|\nabla^k (\varrho^H, a^H, b^H)\|_{L^2}^2 + \|\nabla^k (\varrho^L, a^L, b^L)\|_{L^2}^2.  
\end{align*}  

Adding $\lambda_{15}\times \|\nabla^k(\varrho^L,a^L,b^L)\|_{L^2}^2$ to both sides of \eqref{G5.28}, we obtain  
\begin{align}\label{G5.29}  
\frac{{\rm d}}{{\rm d}t}\mathcal{E}_1(t) + \lambda_{15} \mathcal{E}_1(t) \lesssim \|\nabla^k(\varrho^L,a^L,b^L)\|_{L^2}^2, \quad k = 2,3.  
\end{align}  
Here, we have used the definition in \eqref{G5.26}, which implies that 
\begin{align*}  
\mathcal{E}_1(t) \lesssim \|\nabla^k(\varrho^L,a^L,b^L)\|_{L^2}^2 + \mathcal{D}_1(t).  
\end{align*}  
Consequently, combining \eqref{G5.28} and \eqref{G5.29} allows us to derive \eqref{G5.17}, thereby completing the proof of Proposition \ref{P5.2}.
\end{proof}

\medskip

\section{Optimal Time-Decay Rates for  Classical  Solutions to the   Euler-VFP System}% \eqref{C2}}
This section is devoted to deriving the optimal time-decay rates for $\varrho, u, f$ 
and their spatial gradients to the compressible Euler-VFP system \eqref{C2}, by using the   analysis of the linearized system established in Theorem \ref{T5.1} and the higher-order Lyapunov inequality provided in Lemma \ref{L5.4}.

In contrast to \cite{DL-KRM-2013} on the Euler-VFP system \eqref{A6}, 
we do not need to impose an additional smallness of  $L^1$ norm  
 on the initial condition $(\varrho_0, u_0, f_0)$ to the  system \eqref{C2}. 
 Unlike \cite{LNW-2025-preprint,Ww-CMS-2024}, 
where they required the initial data is bounded in $L^1$ norm, 
here we extend the scope by  assuming  that   the initial data is 
 bounded in the $L^q$-norm for $1 \leq q < \frac{6}{5}$. Subsequently,
  we derive slower time-decay rates for $(\varrho, u, f)$, which will 
  be served as a foundation for the subsequent subsection, where faster time-decay rates will be established.

\subsection{Slower time-decay rates of $(\varrho,u,f)$}
Using Duhamel’s principle and the operator $\mathbb{A}(t) $,  we first reformulate the Cauchy problem   \eqref{C2} as
\begin{align}\label{G6.1}  
U(t)=\mathbb{A}(t)U_{0}+\int_{0}^{t}\mathbb{A}(t-s)\big( S_\varrho(s),S_u(s) , S_f(s)\big) {\rm d}s,
\end{align}  
where the source terms $S_{\varrho}$ and $S_{u}$ are given by  
\begin{align} 
S_\varrho:=\,&-\varrho{\rm div}u-\nabla\varrho\cdot u,  \label{G6.2-1}\\ 
S_u:=\,&-u\cdot\nabla u-\Big(\frac{P^{\prime}(1+\varrho)}{1+\varrho}-P^{\prime}(1)     \Big)\nabla\varrho-au.\label{G6.2}
\end{align}
For the last source term $S_f$, we derive its expression based on the macro-micro decomposition \eqref{G2.4}, leading to  
\begin{align}\label{G6.3}
{ S_f}:=\,&\varrho(\mathcal{L}f+u
\cdot v\sqrt{M}) +(1+\varrho)\Big[-u\cdot\nabla_vf+\frac{1}{2}u\cdot vf\Big]\nonumber\\
=\,&\varrho\mathcal{L}\{\mathbf{I}-\mathbf{P}\}f+\varrho\mathcal{L}\mathbf{P}f
+\varrho u\cdot v\sqrt{M}-(1+\varrho)u\cdot\nabla_v\{\mathbf{I}-\mathbf{P}_0\}f\nonumber\\
&+\frac{1}{2}(1+\varrho)u\cdot v\{\mathbf{I}-\mathbf{P}_0\}f-(1+\varrho)u\cdot\nabla_v \mathbf{P}_0f+\frac{1}{2}(1+\varrho)u\cdot v \mathbf{P}_0 f\nonumber\\
\equiv:\,&h+\varrho(u-b)\cdot v\sqrt{M}+\Big[\nabla_v\cdot G-\frac{1}{2}v\cdot G+(1+\varrho)u\cdot a v\sqrt{M}\Big], 
\end{align}
where $G:=-(1+\varrho)u\{\mathbf{I}-\mathbf{P}_0\}f$ and $h:=\varrho\mathcal{L}\{\mathbf{I}-\mathbf{P}\}f$.

Substituting the equations \eqref{G6.2} and \eqref{G6.3} into \eqref{G6.1}, and then applying the operator $\nabla^l\phi_0(D_x)$ to  $U(t)$, we obtain
\begin{align}\label{G6.4}
\nabla^l U^L(t)=\,\,&\mathbb{A}(t) \nabla^l U^L_0+\int_0^t\mathbb{A}(t-s)(\nabla^l S_\varrho^L(s),0,0)\mathrm{d}s+\int_0^t\mathbb{A}(t-s)(0, \nabla^l S^L_u(s),0)\mathrm{d}s\nonumber\\
&+\int_0^t\mathbb{A}(t-s)\Big(0,0,\nabla^l \Big(\nabla_v\cdot G-\frac{1}{2}v\cdot G \Big)^L\Big)\mathrm{d}s+\int_0^t\mathbb{A}(t-s)(0,0,\nabla^l h^L)\mathrm{d}s\nonumber\\
&+\int_0^t\mathbb{A}(t-s)\big(0,0,\nabla^l \big(\varrho (u-b)\cdot v\sqrt{M} \big)^L\big)\mathrm{d}s\nonumber\\
&+\int_0^t\mathbb{A}(t-s)\big(0,0,\nabla^l \big((1+\varrho)u\cdot av\sqrt{M}\big)^L\big)\mathrm{d}s\nonumber\\
\equiv:\,\,&\sum_{k=1}^7R_k(t).
\end{align}

Under the assumption that $\|(\varrho_0,u_0)\|_{L^q} + \|f_0\|_{\mathcal{Z}_q}$ is bounded for $1 \leq q < \frac{6}{5}$, we present below the slower time decay rates of $(\varrho, u, f)$.
We define
\begin{align}\label{G6.5}
\mathcal{E}_{\infty}(t):=\sup_{0\leq s\leq t}\,&\bigg\{\sum_{l=0}^1(1+s)^{\frac{3}{2} (\frac{1}{q} - \frac{1}{2} )+l}\big(\|\nabla^l(\varrho,u )\|_{L^2}^2 +\|\nabla^l f\|_{L_{x,v}^2}^2\big)\nonumber\\
&+\sum_{l=2}^3(1+s)^{ \frac{3}{2} (\frac{1}{q} - \frac{1}{2} ) +1}\big(\|\nabla^l(\varrho,u )\|_{L^2}^2+\|\nabla^l f\|_{L_{x,v}^2}^2 \big)   \bigg\}.
\end{align}

Now we handle the terms $R_i(t) ( i=1,\dots, 7)$ in \eqref{G6.4}. For any $ l = 0, 1 $, it follows from Theorem \ref{T5.1} that by taking $ p = \frac{4q}{q+2} \in [\frac{4}{3}, \frac{3}{2}) \subset [1, 2] $ in \eqref{G5.7} and applying the interpolation inequality, the term $R_1(t)$ can be estimated as   
\begin{align}\label{G6.6}
\|R_1(t)\|_{\mathcal{Z}_2}\lesssim (1+t)^{-\frac{3}{4} (\frac{1}{q} - \frac{1}{2} )-\frac{l}{2}}  \|U_0\|_{\mathcal{Z}_{\frac{4q}{q+2}}}  \lesssim (1+t)^{-\frac{3}{4} (\frac{1}{q} - \frac{1}{2} )-\frac{l}{2}} \|U_0\|_{\mathcal{Z}_q}^\frac{1}{2}\|U_0\|_{\mathcal{Z}_2}^\frac{1}{2}.
\end{align}

For the terms $R_2(t)$ and $R_3(t)$, by setting $p=0$, $m=l$ and $p=\frac{3}{2} $ , 
$m=l$ in \eqref{G5.7} in Theorem \ref{T5.1}, respectively, and applying Lemmas \ref{L2.1}–\ref{L2.2}, 
Lemma \ref{L2.4}, and the interpolation inequality,  we compute
\begin{align*} 
&\|R_2(t)\|_{\mathcal{Z}_2}+\|R_3(t)\|_{\mathcal{Z}_2}  \nonumber\\
\lesssim &\,
\int_0^{\frac{t}{2}}(1+t-s)^{-\frac{3}{4}-\frac{l}{2}}\big\|\big(S _\varrho (s),S _u(s)\big)\big\|_{\mathcal{Z}_1}\mathrm{d}s  +\int_{\frac{t}{2}}^{t}(1+t-s)^{-\frac{1}{4}-\frac{l}{2}}\big\|\big(S _\varrho (s),S _u(s) \big)\big\|_{\mathcal{Z}_{\frac{3 }{2}}}\mathrm{d}s \nonumber\\
\lesssim&\,\mathcal{E}_{\infty}(t)\int_0^{\frac{t}{2}}(1+t-s)^{-\frac{3}{4}-\frac{l}{2}} (1+s)^{-\frac{3}{2} ( \frac{1}{q}-\frac{1}{2} )}\mathrm{d}s\nonumber\\
&+\mathcal{E}_{\infty}(t)\int_{\frac{t}{2}}^{t}(1+t-s)^{-\frac{1}{4}-\frac{l}{2}} (1+s)^{-\frac{3}{2} ( \frac{1}{q}-\frac{1}{2}  )-\frac{1}{2}}\mathrm{d}s \nonumber\\
\lesssim&\,(1+t)^{-\frac{3}{2} ( \frac{1}{q}-\frac{1}{2}  )-\frac{l}{2}+\frac{1}{4}} \mathcal{E}_{\infty}(t) ,
\end{align*}
for $l=0,1$,
where we have utilized the following facts:
\begin{align*}
\|S_{\varrho} \|_{\mathcal{Z}_1}+\|S_{u} \|_{\mathcal{Z}_1}\lesssim &\,\|(\varrho,u)\|_{L^2} \|(\varrho,u,a)\|_{H^1}\nonumber\\
\lesssim&\,
(1+t)^{-\frac{3}{2} ( \frac{1}{q}-\frac{1}{2} )}\mathcal{E}_{\infty}(t),\nonumber\\
\|S _{\varrho}\|_{\mathcal{Z}_{\frac{3}{2}}}+\|S _{u}\|_{\mathcal{Z}_{\frac{3}{2}}}\lesssim &\, 
\|(\varrho,u)\|_{L^6}\|\nabla(\varrho,u)\|_{L^2}+\|a\|_{L^6}\|u\|_{L^2}\nonumber\\
\lesssim&\, \|\nabla(\varrho,u,a)\|_{H^1}\|(\varrho,u)\|_{H^1}\nonumber\\
\lesssim&\,(1+t)^{-\frac{3}{2} ( \frac{1}{q}-\frac{1}{2}  )-\frac{1}{2}}\mathcal{E}_{\infty}(t).
\end{align*}
Notice that $q \in [1, \frac{6}{5})$, which implies that $-\frac{3}{4} (\frac{1}{q} - \frac{1}{2} ) + \frac{1}{4} < 0$. Therefore, we obtain  
\begin{align}\label{G6.7}
\|R_2(t)\|_{\mathcal{Z}_2} + \|R_3(t)\|_{\mathcal{Z}_2} \lesssim (1+t)^{-\frac{3}{4}
 (\frac{1}{q} - \frac{1}{2} ) - \frac{l}{2}}\mathcal{E}_{\infty}(t),    
\end{align} 
for $l=0,1$.

For the term $R_4(t)$, by applying $\mathbf{P}_0G=0$ and selecting $p=1$, $m=l$, as well as $p=\frac{3}{2}$, $m=0$ in \eqref{G5.8}, respectively, one has
\begin{align}\label{G6.8}
\|R_4(t)\|_{\mathcal{Z}_2}^2\lesssim&\,  \int_0^{\frac{t}{2}}(1+t-s)^{-\frac{3}{2}-l}
\|(1+\varrho)u\{\mathbf{I}-\mathbf{P}_0\}f\|_{\mathcal{Z}_1}^2\mathrm{d}s\nonumber\\
&+ \int_{\frac{t}{2}}^t(1+t-s)^{-\frac{1}{2}}
\big\|\nabla^l\big((1+\varrho)u\{\mathbf{I}-\mathbf{P}_0\}f\big)\big\|_{\mathcal{Z}_{\frac{3}{2}}}^2\mathrm{d}s\nonumber\\
\lesssim&\,  \mathcal{E}_{ \infty}^2(t)\int_0^{\frac{t}{2}}(1+t-s)^{-\frac{3}{2}-l}
(1+s)^{-3 (\frac{1}{q} - \frac{1}{2} )}\mathrm{d}s\nonumber\\
&+ \mathcal{E}_{ \infty}^2(t)\int_{\frac{t}{2}}^t(1+t-s)^{-\frac{1}{2}}(1+s)^{- {3} (\frac{1}{q} - \frac{1}{2} )-1-l}\mathrm{d}s\nonumber\\
\lesssim&\, \Big((1+t)^{-\frac{3}{2}-l}+(1+t)^{- {3} (\frac{1}{q} - \frac{1}{2} )-\frac{1}{2}-l}\Big)\mathcal{E}_{ \infty}^2(t)\nonumber\\
\lesssim&\, (1+t)^{-\frac{3}{2} (\frac{1}{q} - \frac{1}{2} )-l}\mathcal{E}_{ \infty}^2(t),
\end{align}
for $l=0,1$. 

For the term $R_5(t)$, by setting $p=0$, $m=l$ and subsequently $p=\frac{3}{2}$, $m=l$ in \eqref{G5.8} in Theorem \ref{T5.1}, and by utilizing the condition $\mathbf{P}h=0$, we have
\begin{align}\label{G6.9}
\|R_5(t)\|_{\mathcal{Z}_2}^2   \lesssim\,& \int_0^{\frac{t}{2}}(1+t-s)^{-\frac{3}{2}-l}
\|\nu^{-\frac{1}{2}}\varrho\mathcal{L}\{\mathbf{I}-\mathbf{P}\}f\|_{\mathcal{Z}_1}^2\mathrm{d}s\nonumber\\
&+  \int_{\frac{t}{2}}^{t}(1+t-s)^{-\frac{1}{2}-l}
\|\nu^{-\frac{1}{2}}\varrho\mathcal{L}\{\mathbf{I}-\mathbf{P}\}f\|_{\mathcal{Z}_{\frac{3}{2}}}^2\mathrm{d}s\nonumber\\
\lesssim\,&  (1+t)^{-\frac{3}{2}-l}\mathcal{E}_{ \infty}(t)\int_0^{\frac{t}{2}}
\mathcal{D}(s)\mathrm{d}s +(1+t)^{^{-\frac{3}{2} ( \frac{1}{q}-\frac{1}{2}  )-1}}\mathcal{E}_{ \infty}(t)\int_{\frac{t}{2}}^t\mathcal{D}(s)\mathrm{d}s\nonumber\\
\lesssim\,& \tilde{\mathcal{E}}_0 (1+t)^{^{-\frac{3}{2} ( \frac{1}{q}-\frac{1}{2}  )-l}}\mathcal{E}_{ \infty}(t),
\end{align}
for $l=0,1$.

For the remaining terms $R_6(t)$ and $R_7(t)$, we can derive similar estimates by setting $p=1, m=0$ and $p=\frac{3}{2}, m=l$ in \eqref{G5.7} in Theorem \ref{T5.1}, respectively. Applying Lemma \ref{L2.1} and Lemma \ref{L2.4} subsequently yields  
\begin{align}\label{G6.10}
 \|R_6(t)\|_{\mathcal{Z}_2}+\|R_7(t)\|_{\mathcal{Z}_2} 
\lesssim\,&   \int_0^{\frac{t}{2}}(1+t-s)^{-\frac{3}{4}-\frac{l}{2}}\|(1+\varrho)u\cdot a v\sqrt{M}\|_{\mathcal{Z}_1}\mathrm{d}s\nonumber\\
&+  \int_{\frac{t}{2}}^{t}(1+t-s)^{-\frac{1}{4}}\big\|\nabla^l\big((1+\varrho)u\cdot a v\sqrt{M}\big)\big\|_{\mathcal{Z}_{\frac{3}{2}}}\mathrm{d}s\nonumber\\
&+\int_{0}^{\frac{t}{2}}(1+t-s)^{-\frac{3}{4}-\frac{l}{2}}\|\varrho(u-b)\cdot  v\sqrt{M}\|_{\mathcal{Z}_1}\mathrm{d}s\nonumber\\
&+\int_{\frac{t}{2}}^{t}(1+t-s)^{-\frac{1}{4}}\big\|\nabla^l\big(\varrho(u-b)\cdot  v\sqrt{M}\big)\big\|_{\mathcal{Z}_{\frac{3}{2}}}\mathrm{d}s\nonumber\\
\lesssim\,& \int_0^{\frac{t}{2}}(1+t-s)^{-\frac{3}{4}-\frac{l}{2}}\big((1+\|\varrho\|_{L^{\infty}})\|u\|_{L^2}\|a\|_{L^2}+\|\varrho\|_{L^2}\|(u,b)\|_{L^2}\big) 
  \mathrm{d}s\nonumber\\
&+\int_{\frac{t}{2}}^t(1+t-s)^{-\frac{1}{4}}(1+\|\varrho\|_{L^{\infty}})(\|\nabla^l u\|_{{L^2}}\|a\|_{L^6}+\|u\|_{L^6}\|\nabla^l a\|_{L^2})\mathrm{d}s\nonumber\\
&+\int_{\frac{t}{2}}^t(1+t-s)^{-\frac{1}{4}}\|\nabla^l \varrho\|_{L^2}\big(\|a\|_{L^6}\|u\|_{L^{\infty}}+\|(u,b)\|_{L^6}\big)\mathrm{d}s\nonumber\\
&+\int_{\frac{t}{2}}^t(1+t-s)^{-\frac{1}{4}}\|\nabla^l(b-u)\|_{L^2}\|\varrho\|_{L^6}\mathrm{d}s\nonumber\\
\lesssim\,&  \int_0^{\frac{t}{2}}(1+t-s)^{-\frac{3}{4}-\frac{l}{2}}\|( \varrho,u,a,b)\|_{L^2}^2 
  \mathrm{d}s\nonumber\\
&+ \int_{\frac{t}{2}}^t(1+t-s)^{-\frac{1}{4}}\|\nabla^l(\varrho,u,a,b)\|_{L^2}\|\nabla( \varrho,u,a,b)\|_{L^2}\mathrm{d}s\nonumber\\
\lesssim\,&  \mathcal{E}_{ \infty}(t)\int_0^{\frac{t}{2}}(1+t-s)^{-\frac{3}{4}-\frac{l}{2}}(1+s)^{-\frac{3}{2} (\frac{1}{q}-\frac{1}{2}  )}
  \mathrm{d}s\nonumber\\
&+  \mathcal{E}_{ \infty}(t)\int_{\frac{t}{2}}^{t}(1+t-s)^{-\frac{1}{4}}(1+s)^{-\frac{3}{2} (\frac{1}{q}-\frac{1}{2}  )-\frac{l}{2}-\frac{1}{2}}
  \mathrm{d}s\nonumber\\
\lesssim\,& (1+t)^{-\frac{3}{2} (\frac{1}{q}-\frac{1}{2} )+\frac{1}{4} -\frac{l}{2}}\mathcal{E}_{ \infty}(t)\nonumber\\
\lesssim&\,(1+t)^{-\frac{3}{4} (\frac{1}{q} - \frac{1}{2} ) - \frac{l}{2}}\mathcal{E}_{\infty}(t),
\end{align}
for $l=0,1$. 

Thus, inserting \eqref{G6.6}–\eqref{G6.10} into \eqref{G6.4} gives  
\begin{align*}
\|\nabla^l U^L(t)\|_{\mathcal{Z}_2}^2\lesssim (1+t)^{-\frac{3}{2} (\frac{1}{q} - \frac{1}{2} )-l}\Big(   \|U_0\|_{\mathcal{Z}_q}\|U_0\|_{\mathcal{Z}_2} 
+\tilde{\mathcal{E}}_0 \mathcal{E} _{ \infty}(t)+\big(\mathcal{E}_{ \infty}(t)\big)^2\Big),  
\end{align*}
for $l=0,1$.
Combining Proposition \ref{P5.2} with Gronwall's inequality, we have  
\begin{align*}
\|\nabla^k(\varrho,u)(t)\|_{L^2}^2+\|\nabla^k f(t)\|_{L_{x,v}^2}^2\lesssim e^{-\lambda_{11}t}\tilde{\mathcal{E}}_0   +\int_0^t e^{-\lambda_{11}(t-s)}\|\nabla^k U^L(s)\|_{\mathcal{Z}_2}^2{\rm d}s,
\end{align*}
for any $k=2,3$,
which along with \eqref{G2.3} implies that
\begin{align} \label{G6.11}
\|\nabla^k U (t)\|_{\mathcal{Z}_2}^2\lesssim (1+t)^{-\frac{3}{2} (\frac{1}{q} - \frac{1}{2} )-1}\Big( \|U_0\|_{\mathcal{H}^3}^2+  \|U_0\|_{\mathcal{Z}_q}\|U_0\|_{\mathcal{Z}_2} 
+\tilde{\mathcal{E}}_0 \mathcal{E} _{ \infty}(t)+\big(\mathcal{E}_{ \infty}(t)\big)^2\Big),     
\end{align}
for any $k=2,3$.
For $l=0,1$, thanks to \eqref{G2.2}--\eqref{G2.3} and \eqref{G6.11}, we also have
\begin{align}\label{G6.12}
\|\nabla^l U(t)\|_{\mathcal{Z}_2}^2\lesssim \,& \|\nabla^l U^L(t)\|_{\mathcal{Z}_2}^2+\|\nabla^l U^H(t)\|_{\mathcal{Z}_2}^2 \nonumber\\
\lesssim\,& \|\nabla^l U^L(t)\|_{\mathcal{Z}_2}^2+\|\nabla^k U(t)\|_{\mathcal{Z}_2}^2\nonumber\\
\lesssim&\,
(1+t)^{-\frac{3}{2} (\frac{1}{q} - \frac{1}{2} )-l}\Big( \|U_0\|_{\mathcal{H}^3}^2+  \|U_0\|_{\mathcal{Z}_q}\|U_0\|_{\mathcal{Z}_2} 
+\tilde{\mathcal{E}}_0 \mathcal{E} _{ \infty}(t)+\big(\mathcal{E}_{ \infty}(t)\big)^2\Big),    
\end{align}
for $l=0,1$.

From \eqref{G6.11}--\eqref{G6.12} and the definition of $\mathcal{E}_{\infty}(t)$ given in \eqref{G6.5}, we deduce that  
\begin{align*}  
\mathcal{E}_{\infty}(t) \lesssim \|U_0\|_{\mathcal{H}^3}^2 + \|U_0\|_{\mathcal{Z}_q}\|U_0\|_{\mathcal{Z}_2} + \tilde{\mathcal{E}}_0  \mathcal{E}_{\infty}(t) + \big(\mathcal{E}_{\infty}(t)\big)^2.  
\end{align*}  
Since $\tilde{\mathcal{E}}_0 $ is sufficiently small, it follows that  
\begin{align*}  
\mathcal{E}_{\infty}(t) \lesssim \|U_0\|_{\mathcal{H}^3}^2 + \|U_0\|_{\mathcal{Z}_q}\|U_0\|_{\mathcal{Z}_2} + \big(\mathcal{E}_{\infty}(t)\big)^2.  
\end{align*}  
Moreover, since $\|U_{0}\|_{\mathcal{H}^3}$ and $\|U_{0}\|_{\mathcal{Z}_q}\|U_{0}\|_{\mathcal{Z}_2}$ are sufficiently small, applying Lemma \ref{L2.5} yields  
\begin{align*}  
\mathcal{E}_{\infty}(t) \lesssim \|U_0\|_{\mathcal{H}^3}^2 + \|U_0\|_{\mathcal{Z}_q}\|U_0\|_{\mathcal{Z}_2},
\end{align*}  
which implies the decay estimate  
\begin{align}\label{decay-a}  
\|(\varrho,u)\|_{L^2}+\|f\|_{L_{x,v}^2} &\lesssim (1+t)^{-\frac{3}{4} (\frac{1}{q} - \frac{1}{2} )} \Big( \|U_0\|_{\mathcal{H}^3} + \|U_0\|_{\mathcal{Z}_q}^\frac{1}{2}\|U_0\|_{\mathcal{Z}_2}^\frac{1}{2} \Big), \nonumber \\  
\|\nabla(\varrho,u)\|_{H^2}+\|\nabla f\|_{L_v^2(H^2)} &\lesssim (1+t)^{-\frac{3}{4} (\frac{1}{q} - \frac{1}{2} )-\frac{1}{2}} \Big( \|U_0\|_{\mathcal{H}^3} + \|U_0\|_{\mathcal{Z}_q}^\frac{1}{2}\|U_0\|_{\mathcal{Z}_2}^\frac{1}{2} \Big).  
\end{align}

\subsection{Faster optimal time-decay rates of $(\varrho, u, f)$ and Proof of Theorem \ref {T1.4}} 
Noting that the time-decay rates given in \eqref{decay-a} are relatively slow, this subsection aims to establish the faster optimal time-decay rates for $(\varrho,u,f)$, and their spatial gradients.
\begin{proof}[Proof of Theorem \ref{T1.4}]
Denote
\begin{align}\label{G6.14}
\mathcal{E}_{1,\infty}(t):=\sup_{0\leq s\leq t}\,&\bigg\{\sum_{l=0}^1(1+s)^{ {3}  (\frac{1}{q} - \frac{1}{2} )+l}\big(\|\nabla^l(\varrho,u )\|_{L^2}^2 +\|\nabla^l f\|_{L_{x,v}^2}^2\big)\nonumber\\
&+\sum_{l=2}^3(1+s)^{ 3 (\frac{1}{q} - \frac{1}{2} ) +1}\big(\|\nabla^l(\varrho,u )\|_{L^2}^2+\|\nabla^l f\|_{L_{x,v}^2}^2 \big)   \bigg\}.
\end{align}
We provide a more refined energy estimate for the terms $R_i(t) (i=1,\dots, 7)$ in \eqref{G6.4}. To begin with, setting $p = q$ and $m = l$ in \eqref{G5.7} yields  
\begin{align}\label{G6.15}
\|R_1(t)\|_{\mathcal{Z}_2}\lesssim (1+t)^{-\frac{3}{2} ( \frac{1}{q}-\frac{1}{2}  )-\frac{l}{2}}\|U_0\|_{\mathcal{Z}_q},    
\end{align}
for $l=0,1$.
For the terms $R_2(t)$ and $R_3(t)$, we handle the zero-order nonlinear term $au$ separately in the following way:
\begin{align}\label{G6.16}
&\|R_2(t)\|_{\mathcal{Z}_2}+\|R_3(t)\|_{\mathcal{Z}_2}\nonumber\\
\lesssim\,&
\int_0^{\frac{t}{2}}(1+t-s)^{-\frac{3}{4}-\frac{l}{2
}}\|(S_\varrho,S_u) (s)\|_{\mathcal{Z}_1}\mathrm{d}s +\int_{\frac{t}{2}}^{t}(1+t-s)^{-\frac{1}{4}-\frac{l}{2}}\|(S_\varrho,S_u-au )(s)\|_{\mathcal{Z}_{\frac{3}{2}}}\mathrm{d}s\nonumber\\
&+\int_{\frac{t}{2}}^{t}(1+t-s)^{-\frac{1}{4}}\|\nabla^l(au )(s)\|_{\mathcal{Z}_{\frac{3}{2}}}\mathrm{d}s\nonumber\\
\lesssim\,&\int_0^{\frac{t}{2}}(1+t-s)^{-\frac{3}{4}-\frac{l}{2}}\big(\|(\varrho,u)\|_{L^2}\|\nabla (\varrho,u)\|_{L^2}+\|\varrho\|_{L^2}\|\nabla^2u\|_{L^2}+\|a\|_{L^2}\|u\|_{L^2}\big)\mathrm{d}s\nonumber\\
&+\int_{\frac{t}{2}}^t(1+t-s)^{-\frac{1}{4}-\frac{l}{2}}\big(\|(\varrho,u)\|_{L^6}\|\nabla(\varrho,u)\|_{L^2}+\|\varrho\|_{L^6}\|\nabla^2 u\|_{L^2}\big)\mathrm{d}s\nonumber\\
&+\int_{\frac{t}{2}}^t(1+t-s)^{-\frac{1}{4}}\big(\|\nabla^l a\|_{L^2}\|u\|_{L^6}+\|\nabla^l u\|_{L^2}\|a\|_{L^6}     \big)\mathrm{d}s\nonumber\\
\lesssim\,&\Big(\|U_0\|_{\mathcal{H}^3}+\|U_0\|_{\mathcal{Z}_q}^{\frac{1}{2}}\|U_0\|_{\mathcal{Z}_2}^{\frac{1}{2}}   \Big)\mathcal{E}^{\frac{1}{2}}_{1,\infty}(t)\int_0^{\frac{t}{2}}(1+t-s)^{-\frac{3}{4}-\frac{l}{2}}(1+s)^{-\frac{9}{4} (\frac{1}{q} - \frac{1}{2} )}
  \mathrm{d}s\nonumber\\
&+\Big(\|U_0\|_{\mathcal{H}^3}+\|U_0\|_{\mathcal{Z}_q}^{\frac{1}{2}}\|U_0\|_{\mathcal{Z}_2}^{\frac{1}{2}}   \Big)\mathcal{E}^{\frac{1}{2}}_{1,\infty}(t)\int_{\frac{t}{2}}^t(1+t-s)^{-\frac{1}{4}-\frac{l}{2}}(1+s)^{-\frac{9}{4} (\frac{1}{q} - \frac{1}{2} )-1}
  \mathrm{d}s\nonumber\\
&+\Big(\|U_0\|_{\mathcal{H}^3}+\|U_0\|_{\mathcal{Z}_q}^{\frac{1}{2}}\|U_0\|_{\mathcal{Z}_2}^{\frac{1}{2}}   \Big)\mathcal{E}^{\frac{1}{2}}_{1,\infty}(t)\int_{\frac{t}{2}}^t(1+t-s)^{-\frac{1}{4}}(1+s)^{-\frac{9}{4} (\frac{1}{q} - \frac{1}{2} )-\frac{1}{2}-\frac{l}{2}}
  \mathrm{d}s\nonumber\\
\lesssim\,& (1+t)^{-\frac{9}{4} (\frac{1}{q} - \frac{1}{2} )-\frac{l}{2}}\Big(\|U_0\|_{\mathcal{H}^3}+\|U_0\|_{\mathcal{Z}_q}^{\frac{1}{2}}\|U_0\|_{\mathcal{Z}_2}^{\frac{1}{2}}   \Big)\mathcal{E}^{\frac{1}{2}}_{1,\infty}(t),\nonumber\\
\lesssim\,& (1+t)^{-\frac{3}{2} (\frac{1}{q} - \frac{1}{2} )-\frac{l}{2}}\Big(\|U_0\|_{\mathcal{H}^3}+\|U_0\|_{\mathcal{Z}_q}^{\frac{1}{2}}\|U_0\|_{\mathcal{Z}_2}^{\frac{1}{2}}   \Big)\mathcal{E}^{\frac{1}{2}}_{1,\infty}(t), 
\end{align}
for $l
=0,1$.

For the terms $R_4(t)$ and $R_5(t)$, when $q \in [1, \frac{6}{5})$, similar to the derivations in \eqref{G6.8} and \eqref{G6.9}, we arrive at
\begin{align}\label{G6.17}
\|R_4(t)\|_{\mathcal{Z}_2}^2\lesssim&\, \int_0^{\frac{t}{2}}(1+t-s)^{-\frac{3}{2}-l}
\|(1+\varrho)u\{\mathbf{I}-\mathbf{P}_0\}f\|_{\mathcal{Z}_1}^2\mathrm{d}s\nonumber\\
&+ \int_{\frac{t}{2}}^t(1+t-s)^{-\frac{1}{2}}
\big\|\nabla^l\big((1+\varrho)u\{\mathbf{I}-\mathbf{P}_0\}f\big)\big\|_{\mathcal{Z}_{\frac{3}{2}}}^2\mathrm{d}s\nonumber\\
\lesssim&\,(1+t)^{-\frac{3}{2}-l}\Big(\|U_0\|_{\mathcal{H}^3}^2+\|U_0\|_{\mathcal{Z}_q}\|U_0\|_{\mathcal{Z}_2}   \Big)\mathcal{E}_{1,\infty}(t)\nonumber\\
\lesssim&\,(1+t)^{-3 ( \frac{1}{q}-\frac{1}{2}  )-l}\Big(\|U_0\|_{\mathcal{H}^3}^2+\|U_0\|_{\mathcal{Z}_q}\|U_0\|_{\mathcal{Z}_2}   \Big)\mathcal{E}_{1,\infty}(t),
\end{align}
for $l=0,1 $, and
\begin{align}\label{G6.18}
\|R_5(t)\|_{\mathcal{Z}_2}^2   \lesssim\,& \int_0^{\frac{t}{2}}(1+t-s)^{-\frac{3}{2}-l}
\|\nu^{-\frac{1}{2}}\varrho\mathcal{L}\{\mathbf{I}-\mathbf{P}\}f\|_{\mathcal{Z}_1}^2\mathrm{d}s\nonumber\\
&+ \int_{\frac{t}{2}}^{t}(1+t-s)^{-\frac{1}{2}-l}
\|\nu^{-\frac{1}{2}}\varrho\mathcal{L}\{\mathbf{I}-\mathbf{P}\}f\|_{\mathcal{Z}_{\frac{3}{2}}}^2\mathrm{d}s\nonumber\\
\lesssim\,& (1+t)^{-\frac{3}{2}-l} \mathcal{E} _{1,\infty}(t)\int_0^{\frac{t}{2}}
\mathcal{D}(s)\mathrm{d}s +(1+t)^{-3\big(\frac{1}{q}-\frac{1}{2}  \big)-1}\mathcal{E}_{1,\infty}(t)\int_{\frac{t}{2}}^t\mathcal{D}(s)\mathrm{d}s\nonumber\\
\lesssim\,& \tilde{\mathcal{E}}_0 (1+t)^{-3\big(\frac{1}{q}-\frac{1}{2}  \big)-l}\mathcal{E}_{1,\infty}(t),
\end{align}
for $l=0,1$.

For the last two terms $R_6(t)$ and $R_7(t)$, a similar argument to that used in \eqref{G6.10} makes it evident that  
\begin{align}\label{G6.19}
&\|R_6(t)\|_{\mathcal{Z}_2}+\|R_7(t)\|_{\mathcal{Z}_2}\nonumber\\
\lesssim\,&  \int_0^{\frac{t}{2}}(1+t-s)^{-\frac{3}{4}-\frac{l}{2}}\|( \varrho,u,a,b)\|_{L^2}^2 
  \mathrm{d}s\nonumber\\
&+ \int_{\frac{t}{2}}^t(1+t-s)^{-\frac{1}{4}}\|\nabla^l( \varrho,u,a,b)\|_{L^2}\|\nabla(\varrho,u,a,b)\|_{L^2}\mathrm{d}s\nonumber\\
\lesssim\,& \Big(\|U_0\|_{\mathcal{H}^3}+\|U_0\|_{\mathcal{Z}_q}^{\frac{1}{2}}\|U_0\|_{\mathcal{Z}_2}^{\frac{1}{2}}   \Big)\mathcal{E}^{\frac{1}{2}}_{1,\infty}(t)\int_0^{\frac{t}{2}}(1+t-s)^{-\frac{3}{4}-\frac{l}{2}}(1+s)^{-\frac{9}{4} (\frac{1}{q} - \frac{1}{2} )}
  \mathrm{d}s\nonumber\\
&+  \Big(\|U_0\|_{\mathcal{H}^3}+\|U_0\|_{\mathcal{Z}_q}^{\frac{1}{2}}\|U_0\|_{\mathcal{Z}_2}^{\frac{1}{2}}   \Big)\mathcal{E}^{\frac{1}{2}}_{1,\infty}(t)\int_{\frac{t}{2}}^{t}(1+t-s)^{-\frac{1}{4}}(1+s)^{-\frac{9}{4} (\frac{1}{q} - \frac{1}{2} )-\frac{1}{2}-\frac{l}{2}}
  \mathrm{d}s\nonumber\\
\lesssim\,& (1+t)^{-\frac{9}{4} (\frac{1}{q} - \frac{1}{2} )-\frac{l}{2}}\Big(\|U_0\|_{\mathcal{H}^3}+\|U_0\|_{\mathcal{Z}_q}^{\frac{1}{2}}\|U_0\|_{\mathcal{Z}_2}^{\frac{1}{2}}   \Big)\mathcal{E}^{\frac{1}{2}}_{1,\infty}(t),\nonumber\\
\lesssim\,& (1+t)^{-\frac{3}{2} (\frac{1}{q} - \frac{1}{2} )-\frac{l}{2}}\Big(\|U_0\|_{\mathcal{H}^3}+\|U_0\|_{\mathcal{Z}_q}^{\frac{1}{2}}\|U_0\|_{\mathcal{Z}_2}^{\frac{1}{2}}   \Big)\mathcal{E}^{\frac{1}{2}}_{1,\infty}(t), 
\end{align}
for $l
=0,1$.

Substituting all estimates \eqref{G6.15}–\eqref{G6.19} into \eqref{G6.4} yields  
\begin{align}\label{G6.20}
 \|\nabla^l U^L(t)\|_{\mathcal{Z}_2}^2\lesssim (1+t)^{-{3} (\frac{1}{q} - \frac{1}{2} )-l}\big( \|U_0\|_{\mathcal{H}^3}^2+  \|U_0\|_{\mathcal{Z}_q}\|U_0\|_{\mathcal{Z}_2} 
+\mathcal{E}_0\big)\mathcal{E} _{1, \infty}(t),     
\end{align}
for $l=0,1$. Similar to \eqref{G6.11}, using \eqref{G2.3}, one gets
\begin{align}\label{G6.21}
\|\nabla^k U (t)\|_{\mathcal{Z}_2}^2\lesssim (1+t)^{- {3}  (\frac{1}{q} - \frac{1}{2} )-1}\Big(\big( \|U_0\|_{\mathcal{H}^3}^2+  \|U_0\|_{\mathcal{Z}_q}\|U_0\|_{\mathcal{Z}_2} 
+\tilde{\mathcal{E}}_0 \big)\mathcal{E} _{1, \infty}(t)+\|U_0\|_{\mathcal{H}^3\cap\mathcal{Z}_q}^2\Big),     
\end{align}
for any $k=2,3$.
It follows from \eqref{G6.20} and \eqref{G6.21}, together with \eqref{G2.3}, that
\begin{align*} 
\|\nabla^l U(t)\|_{\mathcal{Z}_2}^2\lesssim \,& \|\nabla^l U^L(t)\|_{\mathcal{Z}_2}^2+\|\nabla^l U^H(t)\|_{\mathcal{Z}_2}^2 \nonumber\\
\lesssim\,& \|\nabla^l U^L(t)\|_{\mathcal{Z}_2}^2+\|\nabla^k U(t)\|_{\mathcal{Z}_2}^2\nonumber\\
\lesssim&\,
(1+t)^{- {3}  (\frac{1}{q} - \frac{1}{2} )-l}\Big(\big( \|U_0\|_{\mathcal{H}^3}^2+  \|U_0\|_{\mathcal{Z}_q}\|U_0\|_{\mathcal{Z}_2} 
+\tilde{\mathcal{E}}_0 \big)\mathcal{E} _{1, \infty}(t)+\|U_0\|_{\mathcal{H}^3\cap\mathcal{Z}_q}^2\Big),    
\end{align*}
for $l=0,1$.
Consequently, based on the definition of $\mathcal{E}_{1,\infty}(t)$   in \eqref{G6.14}, we   further derive  that
\begin{align*}
\mathcal{E}_{1,\infty}(t)\lesssim  \big(\|U_0\|_{\mathcal{H}^3}^2+\|U_0\|_{\mathcal{Z}_q}\|U_0\|_{\mathcal{H}^3}+\tilde{\mathcal{E}}_0    \big)\mathcal{E}_{1,\infty}(t)+\|U_0\|_{\mathcal{H}^3\cap{\mathcal{Z}_q}}^2.   
\end{align*}
Due to the smallness of $\|U_0\|_{\mathcal{H}^3}$ and $\tilde{\mathcal{E}}_0 $, we have  
$$
\mathcal{E}_{1,\infty}(t) \lesssim \|U_0\|_{\mathcal{H}^3 \cap \mathcal{Z}_q}^2,
$$  
which verifies that \eqref{d1-1} and \eqref{d1} hold.

Based on the time-decay rates established in \eqref{d1-1}--\eqref{d1}, 
we further investigate the time-decay rates of $(\varrho,u,f)$ in $L^p$-norm  
 for $p\in [2,+\infty]$. We observe that  
\begin{align*}
\|(\varrho,u )\|_{L^2} &\lesssim (1+t)^{-\frac{3}{2} ( \frac{1}{q}-\frac{1}{2} )}, \quad 
\|(\varrho,u )\|_{L^6} \lesssim \|\nabla(\varrho,u )\|_{L^2} \lesssim (1+t)^{-\frac{3}{2} 
( \frac{1}{q}-\frac{1}{2}  )-\frac{1}{2}}, \\
\|f\|_{L_v^2(L^{2})} &\lesssim (1+t)^{-\frac{3}{2} ( \frac{1}{q}-\frac{1}{2} )}, \quad 
\| f\|_{L_v^2(L^{6})} \lesssim \|\nabla f\|_{L_v^2(L^{2})} \lesssim (1+t)^{-\frac{3}{2} ( \frac{1}{q}-\frac{1}{2} )-\frac{1}{2}}.
\end{align*}  
For $p\in [2,6]$, applying the interpolation inequality yields  
\begin{align*}
\|(\varrho,u)\|_{L^p} &\lesssim \|(\varrho,u)\|_{L^2}^{\zeta}\|(\varrho,u )\|_{L^6}^{1-\zeta} \lesssim (1+t)^{-\frac{3}{2}(\frac{1}{q}-\frac{1}{p})}, \\
\|f\|_{L_v^2(L^{p})} &\lesssim \|f\|_{L_v^2(L^{2})}^{\zeta}\|f\|_{L_v^2(L^{6})}^{1-\zeta} \lesssim (1+t)^{-\frac{3}{2}(\frac{1}{q}-\frac{1}{p})},
\end{align*}  
where $\zeta = \frac{6-p}{2p} \in [0,1]$.  

By employing  Lemma \ref{L2.1}, we obtain  
\begin{align*}
\|(\varrho,u)\|_{L^{\infty}} &\lesssim \|\nabla(\varrho,u )\|_{H^{1}} \lesssim (1+t)^{-\frac{3}{2}(\frac{1}{q}-\frac{1}{6})}, \\
\|f\|_{L_v^2(L^{\infty})} &\lesssim \|\nabla f\|_{L_v^2(H^{1})} \lesssim (1+t)^{-\frac{3}{2}(\frac{1}{q}-\frac{1}{6})}.
\end{align*}  
For $p \in [6,\infty]$, using the interpolation inequality again leads to  
\begin{align*}
\|(\varrho,u )\|_{L^p} &\lesssim \|(\varrho,u )\|_{L^6}^{\zeta'} \|(\varrho,u )\|_{L^{\infty}}^{1-\zeta'} \lesssim (1+t)^{-\frac{3}{2}(\frac{1}{q}-\frac{1}{6})}, \\
\|f\|_{L_v^2(L ^{p})} &\lesssim \|f\|_{L_v^2(L ^{6})}^{\zeta'} \|f\|_{L_v^2(L ^{\infty})}^{1-\zeta'} \lesssim (1+t)^{-\frac{3}{2}(\frac{1}{q}-\frac{1}{6})},
\end{align*}  
where $\zeta' = \frac{6}{p} \in [0,1]$. This confirms that \eqref{d2-1}--\eqref{d2} holds.

Finally, we prove the last inequality \eqref{LNW-new}.
It follows from \eqref{C2}$_2$ and \eqref{G3.19}$_2$ that
\begin{align*}
\partial_t(b-u)+2(b-u)=&\,(2+\varrho)au-\varrho(b-u)+u\cdot\nabla u+\nabla a +\Big( \frac{P^\prime(\varrho+1)}{\varrho+1}-P^\prime(1)\Big)\nabla\varrho \nonumber\\
&+P^\prime(1)\nabla\varrho+{\rm div}\big\langle \{v\otimes v-{\rm Id}\}\sqrt{M},\{\mathbf{I}-\mathbf{P}\}f         \big\rangle\nonumber\\
\equiv:&\,\mathcal{Q}_1,
\end{align*} 
which gives rise to
\begin{align}\label{LNWG6.23}
b-u=e^{-2t}(b_0-u_0)+\int_0^t e^{-2(t-\tau)} \mathcal{Q}_1 {\rm d}\tau.    
\end{align}
Taking the $L^2$-norm on both sides of \eqref{LNWG6.23} and using the decay estimates \eqref{d1-1}--\eqref{d1}, we further have
\begin{align}\label{LNWG6.24}
\|(b-u)(t)\|_{L^2}\lesssim&\, e^{-2t}\|b_0-u_0\|_{L^2}+\big(1+\varepsilon^\frac{1}{2}_1\big)\int_0^t e^{-2(t-\tau)} \big(\|\nabla(a,b,\varrho)\|_{L^2}+ \|\nabla f\|_{L_v^2(L^2)}\big){\rm d\tau}\nonumber\\
&+\int_0^t e^{-2(t-\tau)} \big(  \|u\cdot\nabla u\|_{L^2} +\|a u\|_{L^2} +\|\varrho\|_{L^\infty} \|b-u\|_{L^2}     \big){\rm d}\tau\nonumber\\
\lesssim&\, e^{-2t}+\int_0^t e^{-2(t-\tau)} (1+\tau)^{-\frac{3}{2} (\frac{1}{q}-\frac{1}{2}   )-\frac{1}{2} }{\rm d \tau}\nonumber\\
&+ \bigg( \int_0^t  e^{-4(t-\tau)} (1+\tau)^{-3 (\frac{1}{q}-\frac{1}{2}  )-1}{\rm d}\tau  \bigg)^\frac{1}{2} \|b-u\|_{L_t^2(L^2)}       \nonumber\\
\lesssim&\, (1+t)^{-\frac{3}{2} (\frac{1}{q}-\frac{1}{2}   )-\frac{1}{2}}.
\end{align}
We now proceed to derive the remaining time - decay estimate for the microscopic component $\{\mathbf{I}-\mathbf{P}\}f$ in \eqref{LNW-new}. For this purpose, we consider the equation \eqref{C2}$_3$, from which we get
\begin{align}\label{LNWG6.25}
&\partial_t \{\mathbf{I}-\mathbf{P}\}f-(\varrho+1)\mathcal{L}{\{\mathbf{I}-\mathbf{P}\}}f+v\cdot\nabla\{\mathbf{I}-\mathbf{P}\}f+\partial_t\mathbf{P}f \nonumber\\
&\quad=-v\cdot\nabla\mathbf{P}f-(\varrho+1) u\cdot\nabla_v f+(\varrho+1)(u-b)\cdot v\sqrt{M}+\frac{1}{2}(\varrho+1)u\cdot v f.
\end{align}
By taking the $L^2_{x,v}$ inner product of \eqref{LNWG6.25} and using \eqref{G2.5-2}, H\"{o}lder's inequality, 
 Lemma \ref{L2.1}, and Young's inequality, we end up with
\begin{align*}
 &\frac{{\rm d}}{{\rm d}t}\|\{{\mathbf{I}-\mathbf{P}}\}f\|_{L_{x,v}^2}^2+\bar{\lambda}  \|\{{\mathbf{I}-\mathbf{P}}\}f\|_{\nu}^2\nonumber\\
&\quad \lesssim
\| \partial_t \mathbf{P}f\|_{L_{x,v}^2}^2+\big(1+\|(\varrho,u)\|_{H^2}^2 \big)\big(\|\nabla(a,b)\|_{L^2}^2+\|b-u\|_{L^2}^2+\|u\|_{H^2}^2\|\{\mathbf{I}-\mathbf{P}\}f\|_{\nu}^2\big),
\end{align*}
for some constant $\bar{\lambda}>0$,
which together with the fact that $\|\{\mathbf{I}-\mathbf{P}\}f\|_{L_{x,v}^2}\lesssim\|\{\mathbf{I}-\mathbf{P}\}f\|_{\nu}$ yields 
\begin{align}\label{LNWG6.26}
 &\|\{\mathbf{I}-\mathbf{P}\}f(t)\|_{L_{x,v}^2}^2\nonumber\\
 \lesssim&\, e^{-2\bar\lambda t}  \|\{\mathbf{I}-\mathbf{P}\}f_0\|_{L_{x,v}^2}^2+\int_0^t e^{-2(t-\tau)}\big( \|\partial_t{\mathbf{P}f} \|_{L_{x,v}^2}^2+\|\nabla(a,b)\|_{L^2}^2+ \|b-u\|_{L^2}^2  \big){\rm d}\tau \nonumber\\
 \lesssim&\, e^{-2\bar\lambda t}+ \int_0^t e^{-2(t-\tau)}  \|\partial_t{\mathbf{P}f} \|_{L_{x,v}^2}^2{\rm d}\tau+\int_0^t e^{-2(t-\tau)}(1+\tau)^{-\frac{5}{2}}{\rm d}\tau\nonumber\\
 \lesssim&\,(1+t)^{-3 ( \frac{1}{q}-\frac{1}{2}  )-1} + \int_0^t e^{-2(t-\tau)}  \|\partial_t{\mathbf{P}f} \|_{L_{x,v}^2}^2{\rm d}\tau.    
\end{align}
On the other hand, from the equations \eqref{G3.19}$_1$ to \eqref{G3.19}$_2$, it holds that
\begin{align}\label{LNWG6.27}
 \|\partial_t{\mathbf{P}f} \|_{L_{x,v}^2}^2\lesssim &\, \|\partial_t a\|_{L^2}^2+\|\partial_tb\|_{L^2}^2\nonumber\\
 \lesssim&\, (1+\|\varrho\|_{H^2}) \big(\|\nabla (a,b )\|_{L^2}+\|\nabla f\|_{L_{x,v}^2}+\|b-u\|_{L^2} +\|au\|_{L^2}+\|u\cdot b\|_{L^2}\big) \nonumber\\
 \lesssim&\, (1+t)^{-3 ( \frac{1}{q}-\frac{1}{2}  )-1}.
\end{align}
Putting the estimate \eqref{LNWG6.27} into \eqref{LNWG6.26}, we derive
\begin{align}\label{LNWG6.28}
 \|\{\mathbf{I}-\mathbf{P}\}f(t)\|_{L_{x,v}^2} \lesssim (1+t)^{-\frac{3}{2} ( \frac{1}{q}-\frac{1}{2}  )-\frac{1}{2}}.      
\end{align}
Combining the estimates \eqref{LNWG6.24} and \eqref{LNWG6.28}, we consequently obtain \eqref{LNW-new}.
Hence, we complete the proof of Theorem \ref{T1.4}. 
\end{proof}

\medskip

\section{The Periodic Domain Case}

In this section, we address the global existence of  $(\varrho,u,f)$ to the Euler-VFP system \eqref{C2} 
in the spatially periodic domain $\mathbb{T}^3$.  
\begin{thm}\label{T1.5} 
Assume that the initial data $(\varrho_0, u_0, f_0)$ satisfy $F_0 = M + \sqrt{M} f_0 \geq 0$, 
the norm $\|(\varrho_0, u_0)\|_{H^3(\mathbb{T}^3)} + \|f_0\|_{H_{x, v}^3(\mathbb{T}^3\times \mathbb{R}^3)}$ is sufficiently small, and  
\begin{gather*}  
\int_{\mathbb{T}^3} a_0 \, \mathrm{d}x = 0, \quad  
\int_{\mathbb{T}^3} \varrho_0 \, \mathrm{d}x = 0, \quad  
\int_{\mathbb{T}^3} \big(b_0 + (1 + \varrho_0) u_0\big) \, \mathrm{d}x = 0,  
\end{gather*}  
where  
\begin{align*}  
a_0 = \int_{\mathbb{R}^3} \sqrt{M} f_0(x, v) \, \mathrm{d}v, \quad  
b_0 = \int_{\mathbb{R}^3} v \sqrt{M} f_0(x, v) \, \mathrm{d}v.  
\end{align*}  
Then, the Euler-VFP system \eqref{C2} in $\mathbb{T}^3$ admits a unique global strong solution $(\varrho, u, f)$ satisfying $F = M + \sqrt{M} f \geq 0$, and  
\begin{gather*}  
\varrho, u \in C([0, \infty); H^3(\mathbb{T}^3)), \quad f \in C([0, \infty); H_{x,v}^3(\mathbb{T}^3\times\mathbb{R}^3)), \\  
\|(\varrho, u)(t)\|_{H^3(\mathbb{T}^3)} + \|f(t)\|_{H_{x, v}^3(\mathbb{T}^3\times\mathbb{R}^3)} 
\lesssim \big(\|(\varrho_0, u_0)\|_{H^3(\mathbb{T}^3)} + \|f_0\|_{H_{x, v}^3(\mathbb{T}^3\times\mathbb{R}^3)} 
\big) e^{-\varkappa t},  
\end{gather*}  
for some constant $\varkappa > 0$ and for all $t \geq 0$. 
\end{thm}

\begin{proof}

From \eqref{A2}, we obtain the conservation of mass for both the fluid and particles:  
\begin{gather*}  
\frac{\mathrm{d}}{\mathrm{d} t} \int\!\int_{\mathbb{T}^3\times\mathbb{R}^3} F \,\mathrm{d}x\,\mathrm{d}v = 0, \quad  
\frac{\mathrm{d}}{\mathrm{d} t} \int_{\mathbb{T}^3} \rho \,\mathrm{d}x = 0,  
\end{gather*}  
as well as the conservation of total momentum:  
\begin{align*}  
\frac{\mathrm{d}}{\mathrm{d} t} \left( \int_{\mathbb{T}^3} \rho u \,\mathrm{d}x 
+ \int\!\int_{\mathbb{T}^3\times \mathbb{R}^3} v F \,\mathrm{d}x\,\mathrm{d}v \right) = 0.  
\end{align*}  
Furthermore, based on the assumptions stated in Theorem \ref{T1.5}, we deduce that  
\begin{align}\label{G7.1}  
\int_{\mathbb{T}^3} a \,\mathrm{d}x = 0, \quad  
\int_{\mathbb{T}^3} \varrho \,\mathrm{d}x = 0, \quad  
\int_{\mathbb{T}^3} \big(b + (1+\varrho)u\big) \,\mathrm{d}x = 0,  
\end{align}  
for all $t \geq 0$.

Since the proofs of local existence and uniqueness for classical solutions $(\varrho, u, f)$ 
follow arguments similar to those in $\mathbb{R}^3$, we omit them here for brevity. 
The key step lies in deriving uniform a priori estimates for $(\varrho,u,f)$ that 
exhibit exponential decay. By employing the Poincar\'{e}'s inequality together with \eqref{G7.1}, we obtain  
\begin{align}
\|a\|_{L^2(\mathbb{T}^3)}\lesssim\,& \|\nabla a\|_{L^2(\mathbb{T}^3)},  \label{G7.2-1}\\
\|\varrho\|_{L^2(\mathbb{T}^3)}\lesssim\,& \|\nabla \varrho\|_{L^2(\mathbb{T}^3)},\label{G7.2-2}\\
\|u+b\|_{L^2(\mathbb{T}^3)}\lesssim\,&\|b+\varrho(1+u)\|_{L^2(\mathbb{T}^3)}+\|\varrho u\|_{L^2(\mathbb{T}^3)}\nonumber\\
\lesssim\,&\big\|\nabla\big(  b+\varrho(1+u)          \big)\big\|_{L^2(\mathbb{T}^3)}
+\|\varrho\|_{L^3}\|u\|_{L^{6}(\mathbb{T}^3)}\nonumber\\
\lesssim\,&\|\nabla(b,u)\|_{L^2(\mathbb{T}^3)}
+ \|\nabla u\|_{L^3(\mathbb{T}^3)}\|\varrho\|_{L^6(\mathbb{T}^3)}\nonumber\\
&+ \|u\|_{L^{\infty}(\mathbb{T}^3)}\|\nabla\varrho\|_{L^2(\mathbb{T}^3)}
+\|\varrho\|_{H^1(\mathbb{T}^3)}\|\nabla u\|_{L^2(\mathbb{T}^3)}\nonumber\\
\lesssim\,&  \|\nabla(\varrho,u,b)\|_{L^2(\mathbb{T}^3)}.\label{G7.2}
\end{align}
Therefore, we conclude from \eqref{G7.2} and the basic triangle inequality that 
\begin{align}
\|u\|_{L^2(\mathbb{T}^3)}\lesssim\,&  \|b-u\|_{L^2(\mathbb{T}^3)}+  \|u+b\|_{L^2(\mathbb{T}^3)}
\lesssim  \|b-u\|_{L^2(\mathbb{T}^3)}+ \|\nabla(\varrho,u,b)\|_{L^2(\mathbb{T}^3)},\label{G7.3}\\
\|b\|_{L^2(\mathbb{T}^3)}\lesssim\,& \|b-u\|_{L^2(\mathbb{T}^3)}+\|u\|_{L^2(\mathbb{T}^3)}
\lesssim \|b-u\|_{L^2(\mathbb{T}^3)}+\|\nabla(\varrho,u,b)\|_{L^2(\mathbb{T}^3)}.\label{G7.4}
\end{align}
By combining \eqref{G7.2-1}--\eqref{G7.4} with Lemma \ref{L2.2}, 
one has the estimates of $(\varrho, u, f)$ in   $L^p$-norm for  $ 2\leq p\leq\infty $  as follows:
\begin{align}\label{decay-c}
&\|(\varrho,u )\|_{L^p(\mathbb{T}^3)}+\|f\|_{L_v^2(L^p(\mathbb{T}^3))} 
\lesssim   \|\nabla (  \varrho, u, a,b )\|_{H^1(\mathbb{T}^3)} + \|b-u\|_{L^2(\mathbb{T}^3)}  .    
\end{align}

Now, we define the energy functional $\mathcal{E}_{\mathbb{T}}(t)$ 
and the corresponding dissipation rate functional $\mathcal{D}_{\mathbb{T}}(t)$ in $\mathbb{T}^3$, following the same definitions as those in $\mathbb{R}^3$ (see \eqref{G3.32} and \eqref{D0}). Similar to the estimates established in Lemmas \ref{L3.2}–\ref{L3.5} and \eqref{decay-c}, we eventually deduce that     
\begin{align}\label{G7.6}
\frac{{\rm d}}{{\rm d}t}\mathcal{E}_{\mathbb{T}}(t)+\lambda_{16} \mathcal{E}_{\mathbb{T}}(t)\leq 0,\quad 0\leq t<T,
\end{align}
for some constant $\lambda_{16}>0$.
Consequently, applying Gronwall's inequality yields
 \begin{align*}\mathcal{E}_{\mathbb{T}}(t)\leq e^{-\lambda_{16} t}\mathcal{E}_{\mathbb{T}}(0),\end{align*}
  which establishes the exponential decay of $\mathcal{E}_{\mathbb{T}}(t)$. 
  Since $\mathcal{E}_{\mathbb{T}}(t) \sim \|(\varrho,u)\|_{{H}^3(\mathbb{T}^3)}^2
  +\|f\|_{H_{x,v}^3(\mathbb{T}^3\times\mathbb{R}^3)}^2$ decaying exponentially in time, the proof of Theorem \ref{T1.5} is therefore completed.
\end{proof}

%
%
%\bigskip
%\textbf{Conflict of interest statement:}
%The authors confirm that there are no conflicts of interest associated with this research.
%
%\bigskip	
%\textbf{Data availability statement:}
% 	Data sharing is not applicable to this article as no data sets were generated or analyzed during the current study.
	
\medskip 
\section*{Acknowledgements} 
F. Li and   J. Ni  were supported by NSFC (Grant No.  12331007).  
 F. Li was also supported by the “333 Project" of Jiangsu Province.
D. Wang was supported in part by NSF grants DMS-2219384 and DMS-2510532.

\bibliographystyle{plain}

\end{document}